\documentclass[a4paper,12pt,final]{amsart}
\usepackage[utf8]{inputenc}
\usepackage[T1]{fontenc}
\usepackage[UKenglish]{babel}
\usepackage[a4paper,margin=28mm]{geometry}
\usepackage{verbatim}

\allowdisplaybreaks[4]
\usepackage{times}
\usepackage{dsfont,mathrsfs}
\usepackage{amsmath}
\usepackage{amsthm}
\usepackage{amssymb}
\usepackage{amsfonts}
\usepackage{latexsym}
\usepackage{booktabs}

\usepackage[colorlinks,pagebackref]{hyperref}
\usepackage{xcolor}

\newtheorem{theorem}{Theorem}[section]
\newtheorem{lemma}[theorem]{Lemma}
\newtheorem{proposition}[theorem]{Proposition}
\newtheorem{corollary}[theorem]{Corollary}
\newtheorem{assumption}[theorem]{Assumption}

\theoremstyle{definition}

\newtheorem{example}[theorem]{Example}
\newtheorem{remark}[theorem]{Remark}
\numberwithin{equation}{section}
\renewcommand{\labelenumi}{\roman{enumi})}
\renewcommand\theenumi\labelenumi
\renewcommand{\leq}{\leqslant}
\renewcommand{\le}{\leqslant}
\renewcommand{\geq}{\geqslant}
\renewcommand{\ge}{\geqslant}

 \newcommand{\Ll}{\langle}
 \newcommand{\Rr}{\rangle}
 \newcommand{\Z}{\mathbb{Z}}

\newcommand{\tl}{\tilde}
\newcommand{\Be}{\begin{equation}}
\newcommand{\Ees}{\end{equation*}}
\newcommand{\Bes}{\begin{equation*}}
\newcommand{\Ee}{\end{equation}}

\newcommand{\R}{\mathbb{R}}
\newcommand{\E}{\mathbb{E}}
\newcommand{\e}{\varepsilon}

\newcommand{\N}{\mathbb{N}}
\newcommand{\mcl}{\mathcal}

\newcommand{\dif}{\mathrm{d}}

 \newcommand{\Bey}{\begin{eqnarray}}
 \newcommand{\Eey}{\end{eqnarray}}
   \newcommand{\Beys}{\begin{eqnarray*}}
 \newcommand{\Eeys}{\end{eqnarray*}}

\usepackage{marginnote}
\usepackage[notref,notcite]{showkeys}

\begin{document}

\title[Stochastic delay equations and Stochastic Variance Reduced Gradient]
{Approximation to stochastic variance reduced gradient Langevin dynamics by stochastic delay differential equations}
\maketitle

\begin{center}
	{\sc Peng Chen}\\
	{\it 1. Department of Mathematics, Faculty of Science and Technology} \\
	{\it University of Macau Av. Padre Tom\'{a}s Pereira, Taipa Macau, China} \\
	{\it 2. UMacau Zhuhai Research Institute, Zhuhai, China} \\
	{\it yb77430@connect.um.edu.mo} \\
	{\sc Jianya Lu} \\
	{\it 1. Department of Mathematics, Faculty of Science and Technology} \\
	{\it University of Macau Av. Padre Tom\'{a}s Pereira, Taipa Macau, China} \\
	{\it 2. UMacau Zhuhai Research Institute, Zhuhai, China} \\
	{\it yb87454@connect.um.edu.mo} \\
	{\sc Lihu Xu \footnote{* Corresponding author}} \\
	{\it 1. Department of Mathematics, Faculty of Science and Technology} \\
	{\it University of Macau Av. Padre Tom\'{a}s Pereira, Taipa Macau, China} \\
	{\it 2. UMacau Zhuhai Research Institute, Zhuhai, China} \\
	{\it lihuxu@um.edu.mo}
\end{center}
\begin{abstract}
	We study in this paper weak approximations in Wasserstein-1 distance to { stochastic variance} reduced gradient Langevin dynamics by stochastic delay differential equations, and obtain uniform error bounds. Our approach is via Malliavin calculus and a refined Lindeberg principle. 
	
	
\end{abstract}

\maketitle

\section{Introduction}\label{first}
A large body of optimization problems can be formulated as follows:  Let $\psi_{1},\cdots,\psi_{n}$ be a sequence of functions from $\mathbb{R}^{d}$ to $\mathbb{R}$, the goal is to find an approximate solution of the following optimization problem
\begin{align}\label{e:1.1}
{ \omega^*={\rm argmin}_{\omega\in\mathbb{R}^{d}}P(\omega)},\qquad P(\omega):=\frac{1}{n}\sum_{i=1}^{n}\psi_{i}(\omega).
\end{align}
{ In order to fast search the minimizer $\omega^*$},
\cite[Algorithm 2]{DK16} puts forward the stochastic variance reduced gradient Langevin dynamics (SVRG-LD) algorithm as the following: for an $\eta>0$ as the step size, an $m \in \N$ as an epoch length, a $B \in \N$ as the size of mini-batch, and an inverse temperature parameter $\delta>0$,
\begin{itemize}
\item[(1)] let the initial value be $\tl \omega_{0} \in \R^{d}$,
\item[(2)] for $l=0,1,...,$
\begin{itemize}
\item[(i)]  we update the internal iteration as follows:
$$\omega_{l,0}=\tl \omega_{l}, \ \ \ \ \ \tl \mu_{l}=\frac{1}{n}\sum_{i=1}^{n}\nabla\psi_{i}(\tilde{\omega}_{l}), $$
for $t=1,...,m$,
 $$\ \ \ \ \ \ \ \ \ \ \ \ \ \ \omega_{l,t}=\omega_{l,t-1}-\eta\big(\frac{1}{B}\sum_{i_t\in\mathbf{I}_{l}}\left(\nabla\psi_{i_{t}}(\omega_{l,t-1})-\nabla\psi_{i_{t}}(\tilde{\omega}_{l})\right)+\tilde{\mu}_l\big)+\sqrt{\eta\delta}W_{l,t},$$
where $\mathbf{I}_{l,t}$ is a subset of size $B$ randomly drawn from $\{1,\cdots,n\}$, $W_{l,t}\sim \mathcal{N}(0,I_{d\times d})$ and $\mathcal{N}(0,I_{d\times d})$ is the $d$-dimensional normal distribution with mean 0 and covariance matrix $I_{d\times d}$, where $I_{d \times d}$ is a $d\times d$ identity matrix,
\item[(ii)] $\tl \omega_{l+1}=\omega_{l,m}$,
\end{itemize}
\item[(3)] the outcomes are $\tl \omega_{0}, \tl \omega_{1}, ...,....$.
\end{itemize}
In particular, as $\delta=0$ and $B=1$,  SVRG-LD turns to be the stochastic variance reduction gradient descent (SVRG) algorithm in \cite{J-Z}. { The step size $\eta$ is often called learning rate. We call the procedures (i) and (ii) in (2) as internal iteration, and the those of obtaining $(\tilde \omega_k)_{k \ge 0}$ as external iteration.}

 {  We study in this paper the approximation of SVRG-LD by stochastic differential delay equations (SDDEs) in Wasserstein-1 distance, and understand the algorithm from the point of view of continuous stochastic dynamics.}

\vskip 3mm

\subsection{Literature review} Stochastic gradient Langevin dynamics (SGLD) model was put forward in \cite{WeTe11}, and has been extensively applied to study optimization problems arising in lots of research problems, see  \cite{CDT20,zhu19,Rotskoff18,LTE17}.  However, using SGLD often comes at a cost of high variance and thus leads to a slow convergence. Borrowing the variance reduction techniques developed for stochastic gradient descent (SGD) algorithms, Dubey et al. \cite{DK16} proposed an SVRG-LD algorithm.


The weak approximations in most of the aforementioned literatures are defined in
the following way: given a family $\mcl H$ of test functions with high order derivatives and a certain growth condition, the approximation error is defined as
\Be  \label{e:AppHError}
\sup_{h \in \mcl H} |\E h(\omega_{n \eta})-\E h(X_{n\eta})|,
\Ee
where $\omega_{n\eta}$ is the stochastic algorithm at the $n$-th iteration, $X_{n \eta}$ is the continuous dynamics at the time $n \eta$, and $\eta$ is the step size of the algorithm. { As $n\to\infty$ and $\eta\to0$, the bounds obtained in these literatures tends to $0$}. It is not easily seen that the error defined in \eqref{e:AppHError}
implies an error in a Wasserstein type metric. Note that Wasserstein-1 (W1) and Wasserstein-2 (W2) distances often play a crucial role in approximations arising in machine learning and probability measure samplings, and are connected to many important research fields \cite{Lam20,NJ19,ZPP20,BDF20}. Although
\cite{L-T-E} gave a bound for the error between SGD and stochastic differential equations (SDEs) in a so-called smoothed Wasserstein distance, it seems to us that smoothed Wasserstein distances have not been widely applied in research.

The works \cite{BDMS1,DM1,DM2,F-S-X}, arising from the Langevin dynamics samplings or diffusion approximations, have a nature similar to ours.
Due to the distribution sampling motivation, the SDEs in \cite{BDMS1,DM1,DM2} are a reversible gradient system, i.e. the drift is the gradient of a potential $U$, thus the equilibrium measure is known in advance and has a density proportion to $e^{-U}$.
Under certain conditions on the drift, they proved non-asymptotic bounds for total variation or W2 distance between the equilibrium measure and the ergodic measure of the EM scheme of Langevin dynamics, their analysis heavily depends on the gradient form of the drift, and is not easily seen to be extended to a non-gradient system. Our SDDE is not a gradient system in that the system has a delay and that the diffusion term depends on the state rather than a constant.
\cite{F-S-X} is concerning estimating the error between the ergodic measures of an SDE and its EM scheme, the approach does not work for our problem since it is especially designed for bounding the errors between ergodic measures. { We refer the reader to \cite{B-G,BYY16,Buc00} for the study of SDDEs.}

\subsection{Our {motivations}, contributions and methods}
{ It was recently discovered in \cite{L-T-E,LTE17,FDD20,FGL20,Simsekli19,hu19,Fan18,Hu2021,feng18,CSX20}  that SGD algorithms can be (weakly) approximated by continuous time SDEs. These SDEs often offer much needed insight to the algorithms under considerations, for instance, the continuous time treatment allows applications of stochastic control theory to develop novel adaptive algorithms \cite{ZYL20,ZXG20}. It is natural to consider the continuous time approximation problem of SVRG-LD, and understand the algorithm from the point of view of continuous stochastic dynamics.}

Our approximation result { establishes} a connection between SVRG-LD and SDDE, we know from the theory of SDDE if the diffusion term has a delay, a dissipation will be produced to reduce the fluctuation from Brownian motion, see \cite{Sch13} for instance. This provides an intuitive interpretation why SVRG and SVRG-LD can reduce variances.
As is shown in Appendix \ref{invariant measure}, when $K$ in Assumption \ref{assum2} and $\delta$ in the SDDE (\ref{sfde}) below both are small enough, the distribution of the solution of the SDDE (\ref{sfde}) below is close to the minimizer $\omega^*$. This gives another motivation to understand the SVRG-LD and SVRG from SDDE.

Under the assumptions which hold true for a large family of nonconvex minimization problems, our main result Theorem \ref{main} provides an error bound for the approximation which is uniform with respect to external iteration length $s$ and the internal iteration length $m$, from which we can see that the approximation error tends to $0$ as long as the step size $\eta$ tends to $0$, see Corollary \ref{Coro2} and \ref{Coro}. Combining this with the exponential convergence result in \cite{DK16}, we know that our approximating SDDE will be close to the minimizer $\omega^*$  exponentially fast.


{ Our approach to proving our main results is via a refined Lindeberg principle developed in \cite{CSX20} and Malliavin calculus, it seems that this method has not been reported in the study of stochastic algorithms. The classical Lindeberg principle provides an elegant way to prove the normal CLT, and has been applied to many other research topics, see \cite{C-X19, C-S,T-V,CSZ17,C-D,MePe16,KoMo11}. Although Malliavin calculus has been extensively applied to statistical simulation problems such as Langevin dynamics samplings, see \cite{Muzellec20,Fournie99,Privault08,brehier14,F-S-X,GM05,DJ06}. It seems to us that our paper is the first one to apply Malliavin calculus to study stochastic algorithms approximation problems.}

{ In order to use the refined Lindeberg principle and Bismut's version of Malliavin calculus mentioned above}, we need to overcome several difficulties due to the memory property of SVRG-LD and SDDE. Whereas the iteration sequence of SGD is a Markov chain,
that of SVRG-LD is not Markovian due to the memory in each internal iteration (see (2).(i) in SVRG-LD algorithm above).  The approximation methods developed for SGD is therefore not applicable. Based on the special structure of SVRG-LD, we divide the algorithm into local patches which are internal iterations, each patch is a Markov chain, while all of these patches can be viewed as a function valued Markov chain. It is well known that an SDDE system is not Markovian, whence many stochastic analysis methods such as Fokker-Planck equation are not available. Similar as analyzing SVRG-LD, we split the approximating  SDDE into local patches, each patch being a Markov process and all of them forming a function-valued Markov process. Malliavin calculus is applied to the Markov dynamics in patches.



\subsection{Our Assumptions}

For the convenience of analysis, we will consider the case of $B=1$ from now on.

Let $I$ be a uniformly distributed random variable valued on $[n]=\{1,2,...,n\}$, i.e.,
$$\mathbb{P}(I=i)=\frac 1n, \quad \ i \in [n].$$
The Euclidean norm of $x\in\mathbb{R}^{d}$ and the inner product of $x,y\in\mathbb{R}^{d}$ are denoted by $|x|$ and $\langle x,y\rangle$, respectively. For any $x,y \in \R^{d}$, $|x-y|$ is naturally the Euclidean distance between $x$ and $y$. Under the assumptions \ref{assum1'} and \ref{assum2} below, \cite{XCZ18} gave the analysis of convergence of SVRG-LD.

\begin{assumption}\label{assum1'}({\rm Smoothness})
The function $\psi_{i}(x)$ is $L$-smooth for $L>0$, $i=1,\cdots,n,$ i.e.,
\begin{align*}
|\nabla\psi_{i}(x)-\nabla\psi_{i}(y)|\leq L|x-y|, \quad \forall x,y\in\mathbb{R}^{d}.
\end{align*}
\end{assumption}

The assumption above can be relaxed to the following assumption:

\begin{assumption}\label{assum1}
Let $I$ be a uniformly distributed random variable valued on $[n]=\{1,2,...,n\}$, there exists a constant $L>0$ such that for any $x,y\in\mathbb{R}^{d}$, we have
\begin{align}\label{Lip}
\mathbb{E}\left[|\nabla\psi_{I}(x)-\nabla\psi_{I}(y)|^{4}\right]\leq L^{4}|x-y|^{4}.
\end{align}
\end{assumption}

\begin{remark}
By the Cauchy-Schwarz inequality, Assumption \ref{assum1} immediately implies that $P(x)=\frac{1}{n}\sum_{i=1}^{n}\psi_{i}(x)$ is also $L$-smooth, i.e.,
\begin{align}\label{Plip}
|\nabla P(x)-\nabla P(y)|\leq L|x-y|, \quad \forall x,y\in\mathbb{R}^{d},
\end{align}
which further implies
\begin{align}\label{PL}
|\nabla P(x)|\leq|\nabla P(0)|+L|x|, \quad \forall x\in\mathbb{R}^{d}.
\end{align}
\end{remark}

\begin{assumption}\label{assum2}({\rm Dissipative})
There exist constants $\gamma,K{\ge}0,$ such that for any $x\in\mathbb{R}^{d}$ we have
\begin{align*}
\langle\nabla P(x),x\rangle\geq\gamma|x|^{2}-K.
\end{align*}
\end{assumption}

{Assumption \ref{assum2} is usually replaced by the following condition,} which is slightly stronger than Assumption \ref{assum2}.

\begin{assumption}\label{assum3}
There exist constants ${ \gamma>0}$ and $K{\ge}0,$ such that for any $x,y\in\mathbb{R}^{d}$ we have
\begin{align}\label{integration1}
\langle x-y,\nabla P(x)-\nabla P(y)\rangle\geq\gamma|x-y|^{2}-K, \quad \forall x,y\in\mathbb{R}^{d}.
\end{align}
\end{assumption}


\subsection{Notations}
Each time we speak about Lipschitz functions on $\R^d$, we use the Euclidean norm.
$\mcl C(\R^d,\R)$ denotes the collection of all continuous functions $f: \R^d \rightarrow \R$ and $\mcl C^k(\R^d,\R)$ with $k \ge 1$ denotes the collection of all $k$-th order continuously differentiable functions.
$\mcl C^\infty_0(\R^d,\R)$ denotes the set of smooth functions { that vanish at infinity}. For $f \in \mcl C^3(\R^d,\R)$ and $v, v_1, v_2,v_{3},x \in \R^d$, the directional derivative $\nabla_v f(x)$, $\nabla_{v_2}\nabla_{v_1} f(x)$ and $\nabla_{v_{3}}\nabla_{v_2}\nabla_{v_1} f(x)$ are  defined by
\ \ \
\Bes
\begin{split}
& \nabla_v f(x)\ = \ \lim_{\e \rightarrow 0} \frac{f(x+\e v)-f(x)}{\e}, \\
\end{split}
\Ees
\Bes
\begin{split}
&\nabla_{v_2} \nabla_{v_1} f(x)\ = \ \lim_{\e \rightarrow 0} \frac{\nabla_{v_1}f(x+\e v_2)-\nabla_{v_1}f(x)}{\e},
\end{split}
\Ees
\Bes
\begin{split}
&\nabla_{v_3}\nabla_{v_2} \nabla_{v_1} f(x)\ = \ \lim_{\e \rightarrow 0} \frac{\nabla_{v_2}\nabla_{v_1}f(x+\e v_3)-\nabla_{v_2}\nabla_{v_1}f(x)}{\e},
\end{split}
\Ees
respectively. Let $\nabla f(x)\in \R^d$ and $\nabla^2 f(x)\in \R^{d \times d}$ denote the gradient and the Hessian of $f$, respectively.
It is known that
$\nabla_v f(x) = \Ll \nabla f(x), v\Rr$ and
$\nabla_{v_2} \nabla_{v_1} f(x)=\Ll \nabla^2 f(x), v_1 v^{\rm T}_2\Rr_{{\rm HS}}$, where ${\rm T}$ is the transpose operator and $\Ll A, B\Rr_{{\rm HS}}:=\sum_{i,j=1}^d A_{ij} B_{ij}$ for $A, B \in \R^{d \times d}$. We define the operator norm of $\nabla^{2} f(x)$ by
\Beys
\|\nabla^{2} f(x)\|_{\rm op}&=& \sup_{|v_{1}|,|v_{2}|=1} |\nabla_{v_{2}} \nabla_{v_{1}} f(x)|
\Eeys
and
\Beys
\|\nabla^{2} f\|_{{\rm op}, \infty}&=& \sup_{x \in \R^{d}} \|\nabla^{2} f(x)\|_{\rm op}.
\Eeys
We often drop the subscript "{\rm op}" in the definitions above and simply write $\|\nabla^{2} f(x)\|=\|\nabla^{2} f(x)\|_{\rm op}$ and  $\|\nabla^{2} f(x)\|_{\infty}=\|\nabla^{2} f(x)\|_{{\rm op}, \infty}$ if no confusions arise. Similarly we define
\Beys
\|\nabla^{3} f(x)\|_{\rm op}&=& \sup_{|v_{i}|=1, i=1,2,3} |\nabla_{v_{3}} \nabla_{v_{2}} \nabla_{v_{1}} f(x)|
\Eeys
and $\|\nabla^{3} f\|_{{\rm op}, \infty}$ and the short-hand notations $\|\nabla^{3} f(x)\|$ and $\|\nabla^{3} f\|_{\infty}$.

Given a matrix $A \in \R^{d \times d}$, its Hilbert-Schmidt norm is
$\|A\|_{{\rm HS}} \ = \ \sqrt{\sum_{i,j=1}^{d} A^{2}_{ij}} \ = \ \sqrt{{\rm Tr} (A^{\rm T} A)}$ and its operator norm is
$\|A\|_{\rm op} \ = \ \sup_{|v|=1}|A v|$. We have the following relations:
\Be  \label{e:RelEigHS}
\|A\|_{\rm op} \ = \ \sup_{|v_{1}|,|v_{2}|=1} |\Ll A, v_{1} v^{\rm T}_{2}\Rr_{\text{HS}}|, \ \ \ \ \ \ \ \ \ \|A\|_{\rm op} \ \le \ \|A\|_{{\rm HS}} \ \le \ \sqrt d \|A\|_{\rm op} .
\Ee

We can also define $\nabla_v f(x)$, $\nabla_{v_2} \nabla_{v_1} f(x)$ and $\nabla_{v_{3}}\nabla_{v_2} \nabla_{v_1} f(x)$ for a third-order differentiable function $f=(f_1,\dots, f_d)^{\rm T}: \R^d \rightarrow \R^d$ in the same way as above. Define $\nabla f(x)=(\nabla f_1(x), \dots, \nabla f_d(x)) \in \R^{d \times d}$, $\nabla^2 f(x)=\{\nabla^2 f_i(x)\}_{i=1}^d \in \R^{d \times d \times d}$ and $\nabla^3 f(x)=\{\nabla^3 f_i(x)\}_{i=1}^d \in \R^{d \times d \times d\times d}$.
In this case, we have $\nabla_v f(x)=[\nabla f(x)]^{\rm T} v$,
$$
\nabla_{v_2} \nabla_{v_1} f(x)=\{ \Ll \nabla^2 f_1(x), v_1 v^{\rm T}_2\Rr_{{\rm HS}},\dots, \Ll \nabla^2 f_d(x), v_1 v^{\rm T}_2\Rr_{{\rm HS}}\}^{\rm T},
$$
$$\nabla_{v_{3}}\nabla_{v_{2}}\nabla_{v_{1}}f(x)=\{\big\langle\langle\nabla^{3}f_{1}(x),v_{1}v_{2}^{T}\rangle_{{\rm {HS}}},v_{3}\big\rangle,\cdots,\big\langle\langle\nabla^{3}f_{d}(x),v_{1}v_{2}^{T}\rangle_{\rm{HS}},v_{3}\big\rangle\}^{T},$$
here, for any tensor $A\in\mathbb{R}^{d\times d\times d},$ $\big\langle \langle A,v_{1}v_{2}^{T}\rangle_{{\rm {HS}}},v_{3}\big\rangle=\sum_{i,j,k=1}^{d}A_{ijk}v_{1}^{(i)}v_{2}^{(j)}v_{3}^{(k)}$ with $v_{l}^{(i)}$ is the $i$-th component of the vector $v_{l},$ $l=1,2,3.$

Inductively, for $M\in\mathcal{C}^{3}(\mathbb{R}^{d},\mathbb{R}^{d\times d})$ and $v,v_{1},v_{2},v_{3},x\in\mathbb{R}^{d},$ the directional derivative $\nabla_{v}M(x),$ $\nabla_{v_{2}}\nabla_{v_{1}}M(x)$ and $\nabla_{v_{3}}\nabla_{v_{2}}\nabla_{v_{1}}M(x)$ are defined by
\begin{align*}
\nabla_{v}M(x)=\lim_{\epsilon\rightarrow0}\frac{M(x+\epsilon v)-M(x)}{\epsilon},
\end{align*}
\begin{align*}
\nabla_{v_{2}}\nabla_{v_{1}}M(x)=\lim_{\epsilon\rightarrow0}\frac{\nabla_{v_{1}}M(x+\epsilon v_{2})-\nabla_{v_{1}}M(x)}{\epsilon}
\end{align*}
and
\begin{align*}
\nabla_{v_{3}}\nabla_{v_{2}}\nabla_{v_{1}}M(x)=\lim_{\epsilon\rightarrow0}\frac{\nabla_{v_{2}}\nabla_{v_{1}}M(x+\epsilon v_{3})-\nabla_{v_{2}}\nabla_{v_{1}}M(x)}{\epsilon},
\end{align*}
respectively. For a more thorough discussion on the norms and the derivatives, we refer the reader to \cite[Chapters V and VIII]{D69}, \cite[Chapter 5]{D-N-N} and the references therein.

For any $x\geq0$, let $\lfloor x\rfloor$ denote the largest integer which is less than $x$. The symbols $C_{\cdot}$ denote positive numbers depending on subscripts $\cdot$ and their values may vary from line to line.  In addition, we have the convention that $\frac{1}{0}=\infty.$

\vskip 3mm

The paper is organized as following. Our main results and applications are stated in Section \ref{second}. In Section \ref{forth}, we provide some preliminary lemmas and the proof of our main result. We conclude the paper in Section \ref{conclusion}. The proofs of example in Section \ref{second} are given in Appendix \ref{Asecond}, and the proof of the crucial lemma in Section \ref{forth} is deferred to Appendix \ref{moment estimate}. In Appendix \ref{Aind}, we give the proof of the key preliminary lemma using Malliavian calculus that can be read independently.

\section{Motivation and main result}\label{second}
{ On the basis of Assumptions \ref{assum1}, \ref{assum3} and Assumption \ref{assum4} below, we establish our main results, Theorem \ref{main}, Corollary \ref{Coro2} and Corollary \ref{Coro}, about approximating the SVRG-LD by SDDE  which are uniform with respect to $s$ and the internal iteration length $m$.  From Corollary \ref{Coro2}, we can see that the approximation error tends to $0$ as long as the temperature parameter $\delta$ takes $\delta= C|\ln\eta|^{-1}$ and $\eta$ tends to $0$. From Corollary \ref{Coro}, the approximation error tends to $0$ when $K\le\delta$. Combining this with the exponential convergence result in \cite{DK16}, we know that our approximating SDDE will be close to the minimizer $\omega^*$  exponentially fast. Furthermore, we can clearly see the variance reduction effect of the SVRG-LD from the analysis of the SDDE below.}

In the above SVRG-LD algorithm, if we take $k=ml+t$ and denote $\omega_{k}=\omega_{l,t}$, it is easy to check that the iteration in the algorithm can be represented as the following difference equation:
\begin{align}\label{representation}
 \omega_{0}=&\tilde{\omega}_{0}, \nonumber \\
 \omega_{k}=&\omega_{k-1}-\eta\Big[\nabla\psi_{i_{k}}(\omega_{k-1})-\nabla\psi_{i_{k}}(\omega_{\lceil\frac{k}{m}\rceil m})+\nabla P(\omega_{\lceil\frac{k}{m}\rceil m})\Big]+\sqrt{\eta\delta}W_{k},\quad k\geq1,
\end{align}
where $x \in \R$, $\lceil x\rceil$ is the largest integer which is strictly less than $x$. Its output is
\begin{eqnarray}  \label{d:TlOmega}
\tilde{\omega}_s&=&\omega_{ms}, \quad \quad \ s=0,1,2,....
\end{eqnarray}
In order to compare  (\ref{representation}) with an SDDE, we rewrite it as
\begin{eqnarray}  \label{e:OmeDef}
\omega_{k}&=&\omega_{k-1}-\eta\nabla P(\omega_{k-1})+\sqrt{\eta}V_{\eta,\delta}(\omega_{k-1}, \omega_{\lceil\frac{k}{m}\rceil m}, i_k,W_{k}),
\end{eqnarray}
where
\begin{eqnarray*}
&&V_{\eta,\delta}(\omega_{k-1}, \omega_{\lceil\frac{k}{m}\rceil m}, i_k,W_{k})\\
&=&-\sqrt{\eta}\left[\nabla\psi_{i_{k}}(\omega_{k-1})-\nabla\psi_{i_{k}}(\omega_{\lceil\frac{k}{m}\rceil m})-\nabla P(\omega_{k-1})+\nabla P(\omega_{\lceil\frac{k}{m}\rceil m})\right]+\sqrt{\delta}W_{k}.
\end{eqnarray*}
Note that $i_k$ is randomly chosen from $[n]$ and independent of each other. Denote by $I$ a uniformly distributed random variable valued on the set $[n]$. Moreover, $W_{k}$ is also independent of each other, and it is independent of $i_{k}$. A straightforward calculation yields
\begin{eqnarray}  \label{e:ConV}
\mathbb{E}\big[V_{\eta,\delta}(\omega_{k-1}, \omega_{\lceil\frac{k}{m}\rceil m}, i_k,W_{k})\big|\omega_{k-1},\omega_{\lceil\frac{k}{m}\rceil m}\big]&=&0,
\end{eqnarray}
\begin{eqnarray}  \label{e:ConCov}
{\rm Cov}\big[V_{\eta,\delta}(\omega_{k-1}, \omega_{\lceil\frac{k}{m}\rceil m}, i_k,W_{k})\big|\omega_{k-1},\omega_{\lceil\frac{k}{m}\rceil m}\big]&=&\eta\Sigma(\omega_{k-1},\omega_{\lceil\frac{k}{m}\rceil m})+\delta I_{d},
\end{eqnarray}
where
\begin{eqnarray*}
&&\Sigma\left(\omega_{k-1},\omega_{\lceil\frac{k}{m}\rceil m}\right)\\
&=&\E_{I} \Big[\big(\nabla\psi_{I}(\omega_{k-1})-\nabla\psi_{I}(\omega_{\lceil\frac{k}{m}\rceil m})\big) \big(\nabla\psi_{I}(\omega_{k-1})-\nabla\psi_{I}(\omega_{\lceil\frac{k}{m}\rceil m})\big)^{\rm T} \Big]\\
& \ \ & -\E_{I} \left[\nabla\psi_{I}(\omega_{k-1})-\nabla\psi_{I}(\omega_{\lceil\frac{k}{m}\rceil m})\right] \left(\E_{I} \left[\nabla\psi_{I}(\omega_{k-1})-\nabla\psi_{I}(\omega_{\lceil\frac{k}{m}\rceil m})\right]\right)^{\rm T}.
\end{eqnarray*}
By Assumption \ref{assum1}, we further get
\begin{eqnarray}  \label{e:Tr=L2Norm}
	&&\eta {\rm tr}\left(\Sigma(\omega_{k-1},\omega_{\lceil\frac{k}{m}\rceil m})\right)\\
	&=& \eta\mathbb{E}\big[|\nabla\psi_{I}(\omega_{k-1})-\nabla\psi_{I}(\omega_{\lceil\frac{k}{m}\rceil m})|^2\big|\omega_{k-1},\omega_{\lceil\frac{k}{m}\rceil m}\big]-\eta |\nabla P(\omega_{k-1})-\nabla P(\omega_{\lceil\frac{k}{m}\rceil m})|^2  \nonumber \\
	&\le &\eta \E_{I}\left(|\nabla\psi_{I}(\omega_{k-1})-\nabla\psi_{I}(\omega_{\lceil\frac{k}{m}\rceil m})|^2\right)\leq L^{2} \eta |\omega_{k-1}-\omega_{\lceil\frac{k}{m}\rceil m}|^2. \nonumber
\end{eqnarray}	
\vskip 3mm

Because the renewal of $\omega_{k}$ not only depends on $\omega_{k-1}$ but also on $\omega_{\lceil\frac{k}{m}\rceil m}$, it is clear that $\{\omega_{k}\}_{k \ge 0}$ is not a Markov chain. Alternatively, if we take the $m$-tuple
$\overline{\omega_{i}}=\{\omega_{im},\omega_{im+1},...,\omega_{im+m-1}\}$ with $i \ge 0$ as one element, then $\{\overline{\omega_{k}}\}_{k \ge 0}$ is a Markov chain. This has a nature that is very similar to an SDDE, and inspires us to consider the following SDDE:
\begin{align}\label{sfde}
\dif X_{t}=-\nabla P(X_{t})\dif t+\sqrt{\eta}Q_{\eta,\delta}(X_{t},X_{\lceil\frac{t}{m\eta}\rceil m\eta})\dif B_{t}, 
\end{align}
where
$$Q_{\eta,\delta}(x,y)=\big(\Sigma(x,y)+\frac{\delta}{\eta} I_{d}\big)^{\frac{1}{2}}\in \R^{d \times d}$$
is a positive definite matrix for all $x, y \in \R^{d}$ and $B_{t}$ is a standard $d$-dimensional Brownian motion. The Euler-Maruyama discretization with a step-size $\eta$ to (\ref{sfde}) reads as
\begin{align*}
\hat{X}_{k}=\hat{X}_{k-1}-\eta\nabla P(\hat{X}_{k-1})+\eta Q_{\eta,\delta}(\hat{X}_{k-1},{{\hat{X}_{\lceil\frac{k}{m}\rceil m}}})Z_{k-1},
\end{align*}
where $Z_{k-1}$ is a sequence of $d$-dimensional i.i.d. standard normally distributed random vectors. We denote
\Bey  \label{d:TlXs}
\tl X_s&=&X_{sm\eta}, \ \ \ \ \ \ s=0,1,2,....
\Eey

We know from Assumption \ref{assum1} that $\nabla P$ is Lipschitz, it is natural to assume that the higher order derivatives of $P$ is uniformly bounded.
On the other hand, (\ref{e:Tr=L2Norm}) implies that
\begin{align}\label{QL}
\|Q_{\eta,\delta}(\omega_{k-1},\omega_{\lceil\frac{k}{m}\rceil m})\|_{\rm HS}\ \le \ L|\omega_{k-1}-\omega_{\lceil\frac{k}{m}\rceil m}|+\sqrt{\frac{\delta d}{\eta}}.
\end{align}
We see that $Q$ has linear growth. Based on the analysis above, in order to bound the distributions between the SDDE and the SVRG-LD, we further assume that

\begin{assumption}\label{assum4}
There exist constants $A_{i}\geq0$ with $i=1,2,...,5,$ such that for any $x \in \R^{d}$ and unit vectors $v_{i}\in\mathbb{R}^{d}$, i.e., $|v_{i}|=1,$ $i=1,2,3,$ $\nabla P$ satisfies
\begin{align*}
|\nabla_{v_2}\nabla_{v_1} \nabla P(x)| \ \le \ A_{1},\quad
|\nabla_{v_3}\nabla_{v_2}\nabla_{v_1} \nabla P(x)| \ \le \ A_{2};
\end{align*}
and that any $x, y \in \R^{d}$, $Q_{\eta,\delta}$ satisfies
\begin{align}\label{QQ1}
\|\nabla_{1,v_{1}} Q_{\eta,\delta}(x,y)\|_{\rm HS}^{2} \ \le \ A_3,\ \ \ \ \ \|\nabla_{2,v_{1}} Q_{\eta,\delta}(x,y)\|_{\rm HS}^{2} \ \le \ A_3,
\end{align}
$$\|\nabla_{1,v_{1}} \nabla_{1,v_{2}}Q_{\eta,\delta}(x,y)\|_{\rm HS}^{2} \ \le \ A_4, \quad \|\nabla_{1,v_{1}} \nabla_{1,v_{2}} \nabla_{1,v_{3}}Q_{\eta,\delta}(x,y)\|_{\rm HS}^{2} \ \le \ A_5,$$
where $\nabla_{1,v}$ is the partial derivative acting on $x$ along the direction $v \in \R^{d}$, i.e.,
\begin{eqnarray*}
\nabla_{1,v} Q_{\eta,\delta}(x,y)&=&\lim_{\e \rightarrow 0} \frac{Q_{\eta,\delta}(x+\e v,y)-Q_{\eta,\delta}(x,y)}{\e}
\end{eqnarray*}
and $\nabla_{2,v}$ is the partial derivative acting on $y$ along the direction $v \in \R^{d}$, similarly defined as above.
\end{assumption}

In addition, from Assumption \ref{assum1}, it is easy to verify that
\begin{align*}
|\nabla_{v_1} \nabla P(x)| \ \le L|v_{1}|.
\end{align*}

Under the Assumptions \ref{assum1}, \ref{assum3} and \ref{assum4}, there exists a unique solution to the SDDE (\ref{sfde}) and (\ref{sfde}) has a unique invariant measure when $K=0$ (see, e.g., \cite{M,Buc00,B-G,BYY16,BRW20}). From now on, we simply write a number $C_{A_1,...,A_5}$, depending on $A_1,...,A_5$,  by $C_A$ in shorthand.

In the following two examples, inspired from the law of large numbers, we will show that the Assumptions \ref{assum1}, \ref{assum3} and \ref{assum4} holds with high probability in Appendix \ref{Asecond}. Note that in practice one only verifies that the designs of statistical models holds with a high probability, see \cite{Z-Y06} for lasso models and \cite{CSX20} for logistic regression models.

\begin{example}\label{ex2}
(Alternate Model in \cite[Section 5.1]{L-T-E}). Let $H\in\mathbb{R}^{d\times d}$ be a symmetric, positive definite matrix, we diagonalize it in the form $H=\tilde{Q}D\tilde{Q}^{T},$ where $\tilde{Q}$ is an orthogonal matrix and $D$ is a diagonal matrix of eigenvalues. Given a sequence of $n$ samples $a_{1},\cdots,a_{n}$, where $a_{i}\in\mathbb{R}^{d}$ is a sequence of random vectors independently drawn from a standard normally distributed random vector, that is, $a_{i}\sim \mathcal{N}(0,I_{d})$ for $i=1,\cdots,n.$ We then define the loss function
\begin{align*}
\psi_{i}(\omega):=\frac{1}{2}(\tilde{Q}^{T}\omega)^{T}[D+{\rm diag}(a_{i})](\tilde{Q}^{T}\omega),
\end{align*}
where ${\rm diag}(a_{i})$ is a diagonal matrix, whose diagonal elements are each component of the vector $a_{i}.$ Therefore, the SVRG-LD iterates in (\ref{representation}) become,
\begin{align*}
 \omega_{k}=\omega_{k-1}-&\eta\Big[\tilde{Q}[D+{\rm diag}(a_{i_{k}})]\tilde{Q}^{T}(\omega_{k-1}-\omega_{\lceil\frac{k}{m}\rceil m})+\frac{1}{n}\sum_{i=1}^{n}\tilde{Q}[D+{\rm diag}(a_{i})]\tilde{Q}^{T}\omega_{\lceil\frac{k}{m}\rceil m})\Big]\\
 &+\sqrt{\eta\delta}W_{k},
\end{align*}
which implies
\begin{align*}
\nabla P(x)=\frac{1}{n}\sum_{i=1}^{n}\tilde{Q}[D+{\rm diag}(a_{i})]\tilde{Q}^{T}x=Hx+\frac{1}{n}\sum_{i=1}^{n}\tilde{Q}{\rm diag}(a_{i})\tilde{Q}^{T}x
\end{align*}
and
\begin{align*}
\Sigma(x,y)=&\frac{1}{n}\sum_{i=1}^{n}\Big(\tilde{Q}[D+{\rm diag}(a_{i})]\tilde{Q}^{T}(x-y)\Big)\Big(\tilde{Q}[D+{\rm diag}(a_{i})]\tilde{Q}^{T}(x-y)\Big)^{T}\\
&-\Big(\frac{1}{n}\sum_{i=1}^{n}\tilde{Q}[D+{\rm diag}(a_{i})]\tilde{Q}^{T}(x-y)\Big)\Big(\frac{1}{n}\sum_{i=1}^{n}\tilde{Q}[D+{\rm diag}(a_{i})]\tilde{Q}^{T}(x-y)\Big)^{T}.
\end{align*}
\end{example}

\begin{example}\label{ex1}
(Penalized Logistic regression). In machine learning, one of the most popular generalized linear model is the penalized logistic regression for binary classification problem. Given a sequence of $n$ samples $(a_{1},b_{1}),\cdots,(a_{n},b_{n}),$ where $a_{i}$ is a sequence of i.i.d. random vectors and the binary response $b_{i}\in\{0,1\}$ is generated by the following probabilistic model,
\begin{align*}
\mathbb{P}(b_{i}=1|a_{i})=\frac{1}{1+e^{-a_{i}^{T}\omega^{*}}},\quad
\mathbb{P}(b_{i}=0|a_{i})=\frac{1}{1+e^{a_{i}^{T}\omega^{*}}}.
\end{align*}
The corresponding loss function at $\omega$ is
\begin{align*}
\psi_{i}(\omega):=-\big[b_{i}a_{i}^{T}\omega-\ln(1+e^{a_{i}^{T}\omega})\big]+\frac{\lambda}{2} \omega^{T}\omega,
\end{align*}
where $\lambda>0$ is the tuning parameter (see, e.g., \cite[(4.20) and (18.11)]{HTF09}). Therefore, the SVRG-LD iterates in (\ref{representation}) become,
\begin{align*}
 \omega_{k}=\omega_{k-1}-\eta\Big[&a_{i_{k}}(\frac{1}{1+e^{-a_{i_{k}}^{T}\omega_{k-1}}}-\frac{1}{1+e^{-a_{i_{k}}^{T}\omega_{\lceil\frac{k}{m}\rceil m}}})\\
 &+\frac{1}{n}\sum_{i=1}^{n}a_{i}(\frac{1}{1+e^{-a_{i}^{T}\omega_{\lceil\frac{k}{m}\rceil m}}}-b_{i})+\lambda\omega_{k-1}\Big]+\sqrt{\eta\delta}W_{k},
\end{align*}
which implies
\begin{align*}
\nabla P(x)=\frac{1}{n}\sum_{i=1}^{n}a_{i}(\frac{1}{1+e^{-a_{i}^{T}x}}-b_{i})+\lambda x
\end{align*}
and
\begin{align*}
\Sigma(x,y)=&\frac{1}{n}\sum_{i=1}^{n}\Big[\Big(a_{i}(\frac{1}{1+e^{-a_{i}^{T}x}}-\frac{1}{1+e^{-a_{i}^{T}y}})\Big)\Big(a_{i}(\frac{1}{1+e^{-a_{i}^{T}x}}-\frac{1}{1+e^{-a_{i}^{T}y}})\Big)^{T}\Big]\\
&-\Big(\frac{1}{n}\sum_{i=1}^{n}a_{i}(\frac{1}{1+e^{-a_{i}^{T}x}}-\frac{1}{1+e^{-a_{i}^{T}y}})\Big)\Big(\frac{1}{n}\sum_{i=1}^{n}a_{i}(\frac{1}{1+e^{-a_{i}^{T}x}}-\frac{1}{1+e^{-a_{i}^{T}y}})\Big)^{T}.
\end{align*}
\end{example}

\begin{remark}
In Example \ref{ex2}, when $n$ is large enough, with high probability, we can calculate the matrix $Q_{\eta,\delta}(x,y)$ directly. However, to the best of our knowledge, for any positive definite matrix $\hat{\Sigma}(x)$, it is very difficult to calculate the derivative of $\hat{\Sigma}(x)^{\frac{1}{2}}$. With the help of \cite{D-N}, we give a way to estimate the constants $A_{3}$, $A_{4}$ and $A_{5}$ in Assumption \ref{assum4}, which depend on $\eta$ and $\delta$. Moreover, when $\eta\leq\delta$, $A_{3}$, $A_{4}$ and $A_{5}$ are always bounded. For simplicity, we assume that $0<\delta\leq1$ and $\eta\leq\delta$ throughout the paper.
\end{remark}


Recall that $W_1$ distance between two probability measures $\mu_{1}$ and $\mu_{2}$ is defined as
\begin{align}\label{deW1}
W_{1}(\mu_{1},\mu_{2})\ =\ \inf_{(X,Y)\in\mathcal{C}(\mu_{1},\mu_{2})}\mathbb{E}|X-Y|,
\end{align}
where $\mathcal{C}(\mu_{1},\mu_{2})$ is the set of all the coupling realizations of $\mu_{1},\mu_{2}$. By a duality,
$$
W_{1}(\mu_{1},\mu_{2})\ =\ \sup_{h \in {\rm Lip}(1)} |\mu_{1}(h)-\mu_{2}(h)|,
$$
where ${\rm Lip}(1)=\{h: \R^{d} \rightarrow \R; \ |h(y)-h(x)| \le |y-x|\}$ and
$$\mu_{i}(h)=\int_{\R} h(x) \mu_{i}(\dif x), \ \ \ \ i=1,2.$$
The main result of this paper is the following theorem, which provides an approximation error between the distributions of $\tl \omega_s$ and $\tilde X_s$.
\begin{theorem}\label{main}
		Assume that the Assumptions \ref{assum1}, \ref{assum3} and \ref{assum4} hold. Choosing $0<\delta\leq1$ and $\eta\leq\min\{\delta,$ $\big(\frac{\gamma}{432L^{4}}\big)^{\frac{1}{3}},\frac{\gamma}{96L^{2}},(\frac{\gamma}{576L^{3}})^{\frac{1}{2}},\frac{\gamma}{8L^{2}},\frac{\gamma\delta}{2A_{3}}\}$. Then, for any $s\in\N$, there exists a constant $\lambda:=\frac\gamma4\exp\{-\frac{2K(L+3\gamma/4)}{\gamma\delta}\}>0$ such that
\begin{align*}
&W_{1}\left(\mathcal{L}\left(\tilde X_s\right),\mathcal{L}\left(\tl \omega_{s}\right)\right)\\
\le&C_{A,\gamma,d,|\nabla P(0)|,L}\frac{1}{1-a}\big(\mathbb{E}|\tilde{\omega}_{0}|^{4}+\frac{1+K^{2}}{1-\rho}\big)^{\frac{7}{4}}
\left(1+|\ln\eta|+\frac{\delta}{\eta^{\frac{1}{4}}}+\lambda^{-2}e^{\lambda}\right)(\eta\delta)^{\frac{1}{2}},
\end{align*}
where $a=e^{\frac{2K(L+3\gamma/4)}{\gamma\delta}-\lambda m\eta}<1$, $\rho=(1-\gamma\eta)^{m}+\frac{1}{2}<1$ (when $m$ is large enough).
\end{theorem}
\begin{corollary}\label{Coro2}
Keep the same assumptions and notations as Theorem \ref{main}. In addition, suppose $\delta=\frac{8K(L+3\gamma/4)}{\gamma }|\ln\eta|^{-1}$ and $m=\eta^{\frac54+\e}$ for any positive $\e$, then we have
\begin{align*}
W_{1}\left(\mathcal{L}\left(\tilde X_s\right),\mathcal{L}\left(\tl \omega_{s}\right)\right)\to0,\quad \mathrm{as~}\eta\to0.
\end{align*}
\end{corollary}

\begin{corollary}\label{Coro}
Keep the same assumptions and notations as Theorem \ref{main}. In addition, suppose $K\le\delta$ , then we have
\begin{align*}
&W_{1}\left(\mathcal{L}\left(\tilde X_s\right),\mathcal{L}\left(\tl \omega_{s}\right)\right)\\
\le& C_{A,\gamma,d,|\nabla P(0)|,L}(1-e^{-\frac{\gamma}{4}m\eta})^{-1}\big(\mathbb{E}|\tilde{\omega}_{0}|^{4}+\frac{1}{1-\rho}\big)^{\frac{7}{4}}
\left(1+|\ln\eta|+\frac{\delta}{\eta^{\frac{1}{4}}}\right)(\eta\delta)^{\frac{1}{2}},
\end{align*}
the error bound tends to $0$ as $\eta\to0$ and $m\eta=O(1)$.
\end{corollary}
\begin{remark}
When $K$ in Assumption \ref{assum2} and $\delta$ in the SDDE (\ref{sfde}) both are small enough, the distribution of the solution $\tilde{X}_{s}$ of the SDDE (\ref{sfde}) is close to the minimizer $\omega^*$ as $s\rightarrow\infty$, see Appendix \ref{invariant measure} below, this gives another motivation to understand the SVRG-LD and SVRG by SDDE.
\end{remark}
\section{The proof of Theorem \ref{main}}\label{forth}

{ The strategy of proving Theorem \ref{main} is divided into two steps. The first step uses a refined Lindeberg principle and Malliavin calculus to prove an approximation error bound for the internal Markov chains $\{\omega_{k}\}_{ms\leq k\leq m(s+1)}$ and $\{X_{t}\}_{ms\eta\leq t\leq m(s+1)\eta}$ in Subsection \ref{internal sub}, whereas the second step only uses the refined Lindeberg principle to approximate the external Markov chain $\{\tilde{\omega}_{s}\}_{s\geq0}$ by $\{\tilde{X}_{s\eta}\}_{s\geq0}$ in Subsection \ref{external sub}.}

\subsection{Approximation of internal Markov chain}\label{internal sub}
\begin{lemma} \label{l:TwoMC}
	Both $(\tilde \omega_s)_{s \in \Z^{+}}$ and $(\tilde X_s)_{s \in \Z^{+}}$ are Markov chains.
\end{lemma}

\begin{proof}
	The proof is standard.
	For an $s \in \Z^{+}$, given $\tilde \omega_{s}$ (i.e. $\omega_{sm}$), by \eqref{representation} we know that the distribution of $\tilde \omega_{s+1}$ (i.e. $\omega_{(s+1)m}$) is uniquely determined by $\tilde \omega_{s}$ and the i.i.d. random variables $i_{sm+1}, ..., i_{s(m+1)}$, whence
	\begin{eqnarray*}
		\mathbb{P}\left(\tilde \omega_{s+1} \in A \big|\tilde \omega_{s}, ...,\tilde \omega_{0}\right)&=&\mathbb{P}\left(\tilde \omega_{s+1} \in A \big|\tilde \omega_{s}\right), \ \ \ \ \ \quad A \in \mcl B(\R^{d}).
	\end{eqnarray*}
	So $(\tilde \omega_s)_{s \ge 0}$ is a Markov chain.
	
	Similarly, given $\tilde X_s$ (i.e. $X_{sm\eta}$), by \eqref{sfde}, the distribution of  $\tilde X_{s+1}$  is determined by
	$\tilde X_s$ and $(B_{t})_{sm \eta \le t \le (s+1) m \eta}$, from which we know
	\begin{eqnarray*}
		\mathbb{P}\left(\tilde X_{s+1}\eta \in A \big|\tilde X_{s}, ...,\tilde X_1, \tilde X_0\right)&=&\mathbb{P}\left(\tilde X_{s+1} \in A \big|\tilde X_s\right) \ \ \ \ \ \quad A \in \mcl B(\R^{d}),
	\end{eqnarray*}
	so $(\tilde X_s)_{s \ge 0}$ is a Markov chain.
\end{proof}

It is easy to see that an internal iteration of SVRG-LD $\{\omega_k\}_{0 \le k \le m}$ is a time homogeneous Markov chain with states on $\R^d$, while the SDDE \eqref{sfde} restricted on the time period $[0, m\eta]$ reads as
\begin{align}\label{sfde-0}
\dif X_{t}=-\nabla P(X_{t})\dif t+\sqrt{\eta}Q_{\eta,\delta}(X_{t},X_0)\dif B_{t} \ \ \ {\rm for} \ t \in [0,m\eta].
\end{align}
When $X_{0}=\omega_{0}$ is fixed, the above SDDE is equivalent to the following SDE:
\begin{align}\label{sfde-0-0}
\dif X_{t}=-\nabla P(X_{t})\dif t+\sqrt{\eta}Q_{\eta,\delta}(X_{t},\omega_0)\dif B_{t} \ \ \ {\rm for} \ t \in [0,m\eta],
\end{align}
thus the solution $(X_t)_{t \in [0,m\eta]}$ is a time homogeneous Markov process with states on $\R^d$. We denote $X_{s,t}^x$ with $s\le t\in [0,\eta]$ to stress the dependence of process on the value $X_s=x$. For the simplicity of notations, we denote $X_{s,t}^x$ by $X_{t-s}^x$ according to time homogeneous property. $\omega^{x}_{k}$ is denoted by same way.

Let $W\sim \mathcal{N}(0,I_{d}),$ which is independent of $I$. The infinitesimal generators of $\{\omega_k\}_{0 \le k \le m}$ and $(X_t)_{t \in [0,m\eta]}$ are respectively
\begin{eqnarray}\label{e:GenW-1}
\mcl A^{\omega}_j f(x)&=&\E[f(\omega_{j+1})|\omega_j=x]-f(x) \\
&=&\E\left[f\big(x-\eta \nabla \psi_I(x)+\eta \nabla \psi_I(\omega_0)-\eta\nabla P(\omega_0)+\sqrt{\eta\delta}W\big)\right]-f(x)   \nonumber
\end{eqnarray}
for $j=0,1,2,...,m-1$, and
\begin{eqnarray}\label{cong}
	\mcl A^{X}_t f(x)&=&\lim_{\Delta t \rightarrow 0+}\frac{\E[f(X_{t+\Delta t})|X_t=x]-f(x)}{\Delta t}\nonumber\\
	&=& \frac{1}2 \eta\Ll Q_{\eta,\delta}(x,\omega_0)^{2}, \nabla^{2} f(x)\Rr_{{\rm HS}}-\Ll \nabla P(x),\nabla f(x)\Rr\nonumber\\
&=& \frac{1}2 \Ll \eta\Sigma(x,\omega_0)+\delta I_{d}, \nabla^{2} f(x)\Rr_{{\rm HS}}-\Ll \nabla P(x),\nabla f(x)\Rr
\end{eqnarray}
for $t \in [0,m\eta)$. Due to time homogeneous property of the two processes, their generators do not depend on the time, as we have seen from the above, so from now on we shall simply write
\Be  \label{e:GenCom}
\mcl A^X\ = \ \mcl A^{X}_t, \quad \quad \mcl A^{\omega}\ = \ \mcl A^{\omega}_j.
\Ee

Under the above assumptions, we have the following exponential convergence for the SDE (\ref{sfde-0-0}), which will be proved in Appendix \ref{moment estimate}.

\begin{lemma}\label{exponential}
Under Assumption \ref{assum3}, (\ref{Plip}) and (\ref{QQ1}), for any $t>0$ and $x,y\in\mathbb{R}^{d}$, as $\eta\leq\min\{1,\frac{\gamma\delta}{2A_{3}}\}$, there exists a constant $\lambda=\frac\gamma4\exp\{-\frac{2K(L+3\gamma/4)}{\gamma\delta}\}>0$, we have
\begin{align}\label{exponential1}
W_{1}\left(\mathcal{L}(X_{t}^{x}),\mathcal{L}(X_{t}^{y})\right)\leq C_\gamma \lambda^{-1}e^{-\lambda t}|x-y|.
\end{align}

\end{lemma}

Note that the diffusion coefficient of SDE (\ref{sfde-0-0}) is positive definite, we have the following estimates, which will be proved is Appendix \ref{Aind}.
\begin{lemma}\label{mainlem1}
Let $X_{t}$ be the solution to the equation (\ref{sfde-0-0}) and denote $P_{t}h(x)=\mathbb{E}[h(X_{t}^{x})]$ for $h\in Lip(1).$ Then, for any $x\in\mathbb{R}^{d}$ and unit vectors $v,v_{1},v_{2},v_{3}\in\mathbb{R}^{d},$ as $\eta\in(0,\delta]$ and $t\in(0,1]$, we have
\begin{align}\label{gradient2}
|\nabla_{v}(P_{t}h)(x)|\leq e^{L+4},
\end{align}
\begin{align}\label{sec}
|\nabla_{v_{2}}\nabla_{v_{1}}(P_{t}h)(x)|\leq C_{A,L,d}\frac{1}{\sqrt{\delta t}}
\end{align}
and
\begin{align}\label{thirdmain}
|\nabla_{v_{3}}\nabla_{v_{2}}\nabla_{v_{1}}P_{t}h(x)|\leq C_{A,L,d}\left(1+\frac{1}{\delta t}+\frac{1}{t^{\frac{5}{4}}}\right).
\end{align}
\end{lemma}

Now, we give some moment estimates of SDDE and SVRG-LD, which will be proved in Appendix \ref{moment estimate}.

\begin{lemma}\label{fourth}
Let $X_{t}$ be the solution to the equation (\ref{sfde-0-0}) and $\eta<\frac{\gamma}{8L^{2}}$. Then, we have
\begin{align}\label{moment}
\mathbb{E}|X_{t}^{x}|^{2}\leq C_{\gamma,d,|\nabla P(0)|,L}(1+|x|^{2}+\mathbb{E}|\omega_{0}|^{2}+K)
\end{align}
and
\begin{eqnarray}\label{moment-1}
\mathbb{E}|X_{t}^{x}-x|^{2}
&\le&C_{\gamma,d,|\nabla P(0)|,L}(1+|x|^{2}+\mathbb{E}|\omega_{0}|^{2}+K)t(t+\eta+\delta).
\end{eqnarray}
\end{lemma}

\begin{lemma}\label{coupling2}
Let $\omega_{k}^{x}$ be defined in (\ref{representation}) and $\eta\leq\min\{1,\big( \frac{\gamma}{432L^{4}}\big)^{\frac{1}{3}},\frac{\gamma}{96L^{2}},(\frac{\gamma}{576L^{3}})^{\frac{1}{2}}\}$. Then, for any $0\leq k\leq m$, we have
\begin{align}\label{SGDforth}
\mathbb{E}|\omega_{k}^{x}|^{4}\leq C_{\gamma,d,|\nabla P(0)|,L}\big(1+|x|^{4}+\mathbb{E}|\omega_{0}|^{4}+K^{2}\big).
\end{align}
\end{lemma}

Moreover, we also need the following lemma, which will be proved in Appendix \ref{moment estimate}.
\begin{lemma}\label{SGDtaylor}
Let $Z_{t}=X_{\eta t}$, $\mathcal{A}^{Z}$ be the infinitesimal generator. Let $\mcl A^{\omega}$ be defined by (\ref{e:GenCom}) and $u_{t}(x)=\E h(X^x_{t})$ for $0\leq k\leq m.$ Then, as $\eta\leq\min\{\delta,\frac{\gamma}{8L^{2}}\}$ and $t\in(0,1]$, we have
\begin{align*}
&\big|\mathbb{E}\int_{0}^{1}\big[\mathcal{A}^{Z}u_{t}(Z_{s}^{x})-\mathcal{A}^{\omega} u_{t}(x)\big]\dif s\big|\\
\leq&C_{A,\gamma,d,|\nabla P(0)|,L}\big(1+\frac{1}{t}+\frac{\delta}{t^{\frac{5}{4}}}\big)(1+\mathbb{E}|\omega_{0}|^{4}+K)(1+|x|^{3})\eta^{\frac{3}{2}}\delta^{\frac{1}{2}}.
\end{align*}
\end{lemma}

\begin{proposition}\label{regularity}
	Assume that the Assumptions \ref{assum1}, \ref{assum3} and \ref{assum4} hold. Choosing $\eta\leq\min\{\delta,$ $\big(\frac{\gamma}{432L^{4}}\big)^{\frac{1}{3}},$ $\frac{\gamma}{96L^{2}},(\frac{\gamma}{576L^{3}})^{\frac{1}{2}},\frac{\gamma}{8L^{2}},\frac{\gamma}{2A_{3}}\}$, for any $0\leq k\leq m$, we have
\begin{align*}
&W_1(\mcl L(X_{k \eta}), \mcl L(\omega_k))\\
 \leq&C_{A,\gamma,d,|\nabla P(0)|,L}(1+\mathbb{E}|\omega_{0}|^{4}+K^{2})^{\frac{7}{4}}(\eta\delta)^{\frac{1}{2}}\left(1+|\ln\eta|+\frac{\delta}{\eta^{\frac{1}{4}}}+\lambda^{-2}e^{\lambda}\right),
\end{align*}
where $\lambda$ is a positive constant defined in Lemma \ref{exponential}.

\end{proposition}

\begin{proof}
When $k=0,1$, the above inequalities hold obviously.

When $k\geq2$, let $X_{0}=Y_{0}=\omega_{0}$, denote $u_{t}(x)=\mathbb{E}[h(X_{t}^{x})]$ and $Z_{t}=X_{\eta t}$ for $0\leq l\leq k$ and $h\in {\rm Lip}(1)$. For the ease of further use, for any $z\in\mathbb{R}^{d}$, and any $r,t\in\mathbb{Z}^{+}$ with $t\geq r$, we denote by $Z_{t}(t,z)$ the random variable $Z_{t}$ given $Z_{r}=z$, and similarly $\omega_{t}(r,z)$ is similarly defined, it is easily seen that
\begin{align}\label{definition}
Z_{t}=Z_{t}(r,Z_{r}), \quad \omega_{t}=\omega_{t}(r,\omega_{r}).
\end{align}
Then, we have
\begin{align*}
\mathbb{E}h(Z_{k})=\mathbb{E}h\left(Z_{k}(1,Z_{1})\right)-\mathbb{E}h\left(Z_{k}(1,\omega_{1})\right)+\mathbb{E}h\left(Z_{k}(1,\omega_{1})\right).
\end{align*}
By (\ref{definition}) again, we know $Z_{k}(1,\omega_{1})=Z_{k}\left(2,Z_{2}(1,\omega_{1})\right)$ and thus
\begin{align*}
\mathbb{E}h\left(Z_{k}(1,\omega_{1})\right)=\mathbb{E}h\left(Z_{k}\left(2,Z_{2}(1,\omega_{1})\right)\right)-\mathbb{E}h\left(Z_{k}(2,\omega_{2})\right)+\mathbb{E}h\left(Z_{k}(2,\omega_{2})\right).
\end{align*}
Continue this process with repeatedly using (\ref{definition}), we finally obtain
\begin{align*}
\mathbb{E}h(Z_{k})=\sum_{j=1}^{k}\left[\mathbb{E}h\left(Z_{k}(j,Z_{j}(j-1,\omega_{j-1}))\right)-\mathbb{E}h\left(Z_{k}(j,\omega_{j})\right)\right]+\mathbb{E}h(\omega_{k}),
\end{align*}
and thus
\begin{align*}
\mathbb{E}h(Z_{k})-\mathbb{E}h(\omega_{k})=\sum_{j=1}^{k}\left[\mathbb{E}h\left(Z_{k}(j,Z_{j}(j-1,\omega_{j-1}))\right)-\mathbb{E}h\left(Z_{k}(j,\omega_{j})\right)\right].
\end{align*}
Let us now bound each term on the right hand side. Because $Z_{t}$ is a time homogeneous Markov chain, we have
\begin{align*}
u_{\eta(k-j)}(z)=\mathbb{E}\left[h(X_{\eta k})|X_{\eta j}=z\right]=\mathbb{E}\left[h(Z_{k})|Z_{j}=z\right].
\end{align*}
Now, we have
\begin{align*}
&\mathbb{E}h\left(Z_{k}(j,Z_{j}(j-1,\omega_{j-1}))\right)-\mathbb{E}h\left(Z_{k}(j,\omega_{j})\right)\\
=&\mathbb{E}u_{\eta(k-j)}\left(Z_{j}(j-1,\omega_{j-1})\right)-\mathbb{E}u_{\eta(k-j)}(\omega_{j})\\
=&\mathbb{E}u_{\eta(k-j)}\left(Z_{j}(j-1,\omega_{j-1})\right)-\mathbb{E}u_{\eta(k-j)}\left(\omega_{j}(j-1,\omega_{j-1})\right)\\
=&\mathbb{E}u_{\eta(k-j)}\left(Z_{1}^{\omega_{j-1}}\right)-\mathbb{E}u_{\eta(k-j)}\left(\omega_{1}^{\omega_{j-1}}\right),
\end{align*}
where the second equality is by (\ref{definition}) and the last one is by the relation $Z_{1}^{\omega_{j-1}}\stackrel{\rm d}{=}Z_{j}(j-1,\omega_{j-1})$ and $\omega_{1}^{\omega_{j-1}}\stackrel{\rm d}{=}\omega_{j}(j-1,\omega_{j-1})$. Hence, we have
\begin{align}\label{definition2}
\mathbb{E}h(Z_{k})-\mathbb{E}h(\omega_{k})=\sum_{j=1}^{k}\left[\mathbb{E}u_{\eta(k-j)}\left(Z_{1}^{\omega_{j-1}}\right)-\mathbb{E}u_{\eta(k-j)}\left(\omega_{1}^{\omega_{j-1}}\right)\right],
\end{align}
which further implies
\begin{align}\label{definition2}
W_{1}\left(\mathcal{L}(Z_{k}),\mathcal{L}(\omega_{k})\right)\leq&\sum_{j=1}^{k-1}\sup_{h\in{\rm Lip}(1)}\left|\mathbb{E}u_{\eta(k-j)}\left(Z_{1}^{\omega_{j-1}}\right)-\mathbb{E}u_{\eta(k-j)}\left(\omega_{1}^{\omega_{j-1}}\right)\right|\nonumber\\
&+\sup_{h\in{\rm Lip}(1)}\left|\mathbb{E}h\left(Z_{1}^{\omega_{k-1}}\right)-\mathbb{E}h\left(\omega_{1}^{\omega_{k-1}}\right)\right|.
\end{align}

For the first term, if $k\leq\eta^{-1}+1$, denote the generator of the process $Z_{t}$ by $\mathcal{A}^{Z}$. Then, by It$\hat{o}$'s formula and the definition of $\mathcal{A}^{\omega}$, for any $1\leq j\leq k-1$, we have
\begin{align}\label{NEED}
&\mathbb{E}u_{\eta(k-j)}\left(Z_{1}^{\omega_{j-1}}\right)-\mathbb{E}u_{\eta(k-j)}\left(\omega_{1}^{\omega_{j-1}}\right)\nonumber\\
=&\mathbb{E}\left[u_{\eta(k-j)}\left(Z_{1}^{\omega_{j-1}}\right)-u_{\eta(k-j)}(\omega_{j-1})\right]-\mathbb{E}\left[u_{\eta(k-j)}\left(\omega_{1}^{\omega_{j-1}}\right)-u_{\eta(k-j)}(\omega_{j-1})\right]\nonumber\\
=& \mathbb{E}\int_{0}^{1}\big[\mathcal{A}^{Z}u_{\eta(k-j)}(Z_{s}^{\omega_{j-1}})-\mathcal{A}^{\omega} u_{\eta(k-j)}(\omega_{j-1})\big]\dif s.
\end{align}
Since $\eta(k-j)\in(0,1]$, one can derive from Lemma \ref{SGDtaylor}, the H\"{o}lder inequality and Lemma \ref{coupling2} that
\begin{align*}
&\sum_{j=1}^{k-1}\sup_{h\in{\rm Lip}(1)}\left|\mathbb{E}u_{\eta(k-j)}\left(Z_{1}^{\omega_{j-1}}\right)-\mathbb{E}u_{\eta(k-j)}\left(\omega_{1}^{\omega_{j-1}}\right)\right|\\
\leq&C_{A,\gamma,d,|\nabla P(0)|,L}\sum_{j=1}^{k-1}\big(1+\frac{1}{\eta(k-j)}+\frac{\delta}{[\eta(k-j)]^{\frac{5}{4}}}\big)(1+\mathbb{E}|\omega_{0}|^{4}+K)(1+\mathbb{E}|\omega_{j-1}|^{3})\eta^{\frac{3}{2}}\delta^{\frac{1}{2}}\\
\leq&C_{A,\gamma,d,|\nabla P(0)|,L}(1+\mathbb{E}|\omega_{0}|^{4}+K^{2})^{\frac{7}{4}}\sum_{j=1}^{k-1}\big(1+\frac{1}{\eta(k-j)}+\frac{\delta}{[\eta(k-j)]^{\frac{5}{4}}}\big)\eta^{\frac{3}{2}}\delta^{\frac{1}{2}}\\
\leq&C_{A,\gamma,d,|\nabla P(0)|,L}(1+\mathbb{E}|\omega_{0}|^{4}+K^{2})^{\frac{7}{4}}(\eta\delta)^{\frac{1}{2}}\left(1+|\ln\eta|+\frac{\delta}{\eta^{\frac{1}{4}}}\right).
\end{align*}
If $k>\eta^{-1}+1$, we consider the following two cases. When $1\leq j< k-\eta^{-1}$, notice that $u_{\eta(k-j)}(x)=\mathbb{E}\left[h(X_{\eta(k-j)}^{x})\right]=\mathbb{E}\left[h(X_{\eta(k-j)-1}^{X_{1}^{x}})\right]$, by (\ref{exponential1}) with $t=(k-j)\eta-1$, we have
\begin{align*}
&\sup_{h\in{\rm Lip}(1)}\left|\mathbb{E}u_{\eta(k-j)}\left(X_{\eta}^{\omega_{j-1}}\right)-\mathbb{E}u_{\eta(k-j)}\left(\omega_{1}^{\omega_{j-1}}\right)\right|\\
=&\sup_{h\in{\rm Lip}(1)}\left|\mathbb{E}h\left(X_{\eta(k-j)}^{X_{\eta}^{\omega_{j-1}}}\right)-\mathbb{E}h\left(X_{\eta(k-j)}^{\omega_{1}^{\omega_{j-1}}}\right)\right|\\
=&\sup_{h\in{\rm Lip}(1)}\left|\mathbb{E}h\left(X_{\eta(k-j)-1}^{\bar{X}}\right)-\mathbb{E}h\left(X_{\eta(k-j)-1}^{\bar{Y}}\right)\right|\\
=&W_{1}\left(\mathcal{L}\left(X_{\eta(k-j)-1}^{\bar{X}}\right),\mathcal{L}\left(X_{\eta(k-j)-1}^{\bar{Y}}\right)\right)\leq e^{\frac{2K(L+3\gamma/4)}{\gamma\delta}}e^{-\lambda [\eta(k-j)-1]}\mathbb{E}\left|\bar{X}-\bar{Y}\right|,
\end{align*}
where $\bar{X}$ and $\bar{Y}$ are any random vectors such that $\bar{X}\stackrel{\rm d}{=}X_{1}^{X_{\eta}^{\omega_{j-1}}}$ and $\bar{Y}\stackrel{\rm d}{=}X_{1}^{\omega_{1}^{\omega_{j-1}}}$. Thus, noting that $\bar{X}$ and $\bar{Y}$ are arbitrary, it follows from the definition of the Wasserstein-1 distance (see (\ref{deW1}), taking the infimum about $\bar{X}$ and $\bar{Y}$ on the right side of the above inequality) that
\begin{align*}
&\sup_{h\in{\rm Lip}(1)}\left|\mathbb{E}u_{\eta(k-j)}\left(X_{\eta}^{\omega_{j-1}}\right)-\mathbb{E}u_{\eta(k-j)}\left(\omega_{1}^{\omega_{j-1}}\right)\right|\\
\leq& C_\gamma\lambda^{-1} e^{-\lambda [\eta(k-j)-1]}W_{1}\left(\mathcal{L}\left(X_{1}^{X_{\eta}^{\omega_{j-1}}}\right),\mathcal{L}\left(X_{1}^{\omega_{1}^{\omega_{j-1}}}\right)\right)\\
=&C_\gamma\lambda^{-1}e^{-\lambda [\eta(k-j)-1]}\sup_{h\in{\rm Lip}(1)}\left|\mathbb{E}h\left(X_{1}^{X_{\eta}^{\omega_{j-1}}}\right)-\mathbb{E}h\left(X_{1}^{\omega_{1}^{\omega_{j-1}}}\right)\right|\\
=&C_\gamma\lambda^{-1}e^{-\lambda [\eta(k-j)-1]}\sup_{h\in{\rm Lip}(1)}\left|\mathbb{E}u_{1}\left(X_{\eta}^{\omega_{j-1}}\right)-\mathbb{E}u_{1}\left(\omega_{1}^{\omega_{j-1}}\right)\right|.
\end{align*}
Then, by (\ref{NEED}), Lemmas \ref{SGDtaylor} and \ref{coupling2}, we have
\begin{align*}
&\sum_{j=1}^{\lfloor k-\eta^{-1}\rfloor}\sup_{h\in{\rm Lip}(1)}\left|\mathbb{E}[u_{\eta(k-j)}(X_{\eta}^{\omega_{j-1}})-u_{\eta(k-j)}(\omega_{j})]\right|\\
\leq&C_{A,\gamma,d,|\nabla P(0)|,L}\lambda^{-1}(1+\mathbb{E}|\omega_{0}|^{4}+K)\eta^{\frac{3}{2}}\delta^{\frac{1}{2}}\sum_{j=1}^{\lfloor k-\eta^{-1}\rfloor}e^{-\lambda [\eta(k-j)-1]}(1+\mathbb{E}|\omega_{j-1}|^{3})\\
\leq&C_{A,\gamma,d,|\nabla P(0)|,L}\lambda^{-1}(1+\mathbb{E}|\omega_{0}|^{4}+K)\eta^{\frac{3}{2}}\delta^{\frac{1}{2}}\sum_{j=1}^{\lfloor k-\eta^{-1}\rfloor}e^{-\lambda [\eta(k-j)-1]}\left(1+(\mathbb{E}|\omega_{0}|^{4})^{\frac{3}{4}}+K^{\frac{3}{2}}\right)\\
\leq&C_{A,\gamma,d,|\nabla P(0)|,L}\lambda^{-1}e^\lambda(1+\mathbb{E}|\omega_{0}|^{4}+K^{2})^{\frac{7}{4}}\eta^{\frac{3}{2}}\delta^{\frac{1}{2}}\sum_{j=1}^{\lfloor k-\eta^{-1}\rfloor}e^{-\lambda\eta(k-j)}
\end{align*}
which further implies
\begin{align*}
&\sum_{j=1}^{\lfloor k-\eta^{-1}\rfloor}\sup_{h\in{\rm Lip}(1)}\left|\mathbb{E}[u_{\eta(k-j)}(X_{\eta}^{\omega_{j-1}})-u_{\eta(k-j)}(\omega_{j})]\right|\\
\leq&C_{A,\gamma,d,|\nabla P(0)|,L}\lambda^{-1}e^{\lambda}(1+\mathbb{E}|\omega_{0}|^{4}+K^{2})^{\frac{7}{4}}\eta^{\frac{3}{2}}\delta^{\frac{1}{2}}\sum_{j=k-\lfloor k-\eta^{-1}\rfloor}^{k-1}e^{-\lambda\eta j}\\
\leq&C_{A,\gamma,d,|\nabla P(0)|,L}\lambda^{-1}e^{\lambda}(1+\mathbb{E}|\omega_{0}|^{4}+K^{2})^{\frac{7}{4}}\eta^{\frac{3}{2}}\delta^{\frac{1}{2}}\int_{\eta^{-1}-1}^{k}e^{-\lambda\eta y}\dif y\\
\leq&C_{A,\gamma,d,|\nabla P(0)|,L}\lambda^{-2}e^{\lambda}(1+\mathbb{E}|\omega_{0}|^{4}+K^{2})^{\frac{7}{4}}(\eta\delta)^{\frac{1}{2}}.
\end{align*}
When $k-\eta^{-1}\leq j\leq k-1$, by (\ref{NEED}), Lemmas \ref{SGDtaylor} and \ref{coupling2}, we have
\begin{align*}
&\sum_{j=\lfloor k-\eta^{-1}\rfloor+1}^{k-1}\sup_{h\in{\rm Lip}(1)}\left|\mathbb{E}[u_{\eta(k-j)}(X_{\eta}^{\omega_{j-1}})-u_{\eta(\omega-j)}(\omega_{j})]\right|\\
\leq&C_{A,\gamma,d,|\nabla P(0)|,L}(1+\mathbb{E}|\omega_{0}|^{4}+K^{2})^{\frac{7}{4}}\eta^{\frac{3}{2}}\delta^{\frac{1}{2}}\sum_{j=\lfloor k-\eta^{-1}\rfloor+1}^{k-1}\big(1+\frac{1}{\eta(k-j)}+\frac{\delta}{[\eta(k-j)]^{\frac{5}{4}}}\big)\\
\leq&C_{A,\gamma,d,|\nabla P(0)|,L}(1+\mathbb{E}|\omega_{0}|^{4}+K^{2})^{\frac{7}{4}}(\eta\delta)^{\frac{1}{2}}\left(1+|\ln\eta|+\frac{\delta}{\eta^{\frac{1}{4}}}\right).
\end{align*}
Therefore, we have
\begin{align*}
&\sum_{j=1}^{k-1}\sup_{h\in{\rm Lip}(1)}\left|\mathbb{E}u_{\eta(k-j)}\left(X_{\eta}^{\omega_{j-1}}\right)-\mathbb{E}u_{\eta(k-j)}\left(\omega_{1}^{\omega_{j-1}}\right)\right|\\
\leq&C_{A,\gamma,d,|\nabla P(0)|,L}(1+\mathbb{E}|\omega_{0}|^{4}+K^{2})^{\frac{7}{4}}(\eta\delta)^{\frac{1}{2}}\left(1+|\ln\eta|+\frac{\delta}{\eta^{\frac{1}{4}}}+\lambda^{-2}e^{\lambda}\right).
\end{align*}

For the second term,by (\ref{representation}), (\ref{Lip}), Cauchy-Schwarz inequality and (\ref{PL}), we have
\begin{align*}
\mathbb{E}\big|\omega_{k}-\omega_{k-1}\big|\leq&\eta\Big[\mathbb{E}\big|\nabla\psi_{i_{k}}(\omega_{k-1})-\nabla\psi_{i_{k}}(\omega_{0})+\nabla P(\omega_{0})\big|+(\frac{\delta}{\eta})^{\frac{1}{2}}\mathbb{E}|W_{k}|\Big]\\
\leq&C_{L}\eta\big(\mathbb{E}|\omega_{k-1}|+\mathbb{E}|\omega_{0}|+|\nabla P(0)|+(\frac{\delta d}{\eta})^{\frac{1}{2}}\big).
\end{align*}
Then, by the Cauchy-Schwarz inequality, (\ref{moment-1}) and (\ref{SGDforth}), we further have
\begin{align*}
&\left|\mathbb{E}h\left(Z_{1}^{\omega_{k-1}}\right)-\mathbb{E}h\left(\omega_{1}^{\omega_{k-1}}\right)\right|\\
\leq&\mathbb{E}\big|h\big(X_{\eta}^{\omega_{k-1}}\big)-h(\omega_{k-1})\big|+\mathbb{E}\big|h(\omega_{k})-h(\omega_{k-1})\big|\\
\leq&\mathbb{E}\big|X_{\eta}^{\omega_{k-1}}-\omega_{k-1}\big|+\mathbb{E}\big|\omega_{k}-\omega_{k-1}\big|\\
\leq&C_{\gamma,L}\eta\big(\sqrt{\mathbb{E}|\omega_{k-1}|^{2}}+K^{\frac{1}{2}}+(\mathbb{E}|\omega_{0}|^{4})^{\frac{1}{4}}+|\nabla P(0)|+(\frac{\delta d}{\eta})^{\frac{1}{2}}\big)\\
\leq&C_{\gamma,d,|\nabla P(0)|,L}\big[1+(\mathbb{E}|\omega_{0}|^{4})^{\frac{1}{4}}+K^{\frac{1}{2}}\big](\eta\delta)^{\frac{1}{2}}.
\end{align*}

Combining all of above, we have
\begin{align*}
&W_{1}\left(\mathcal{L}(X_{\eta N}),\mathcal{L}(Y_{N})\right)\\
\leq&C_{A,\gamma,d,|\nabla P(0)|,L}(1+\mathbb{E}|\omega_{0}|^{4}+K^{2})^{\frac{7}{4}}(\eta\delta)^{\frac{1}{2}}\left(1+|\ln\eta|+\frac{\delta}{\eta^{\frac{1}{4}}}+\lambda^{-2}e^{\lambda}\right).
\end{align*}

\end{proof}

\subsection{Approximation of external Markov chain}\label{external sub}

Recall $\tl X_s$ from \eqref{d:TlXs}. Let $h \in \mcl C^1(\R^d,\R)$ be Lipschitz, define
\Bey  \label{e:Uh}
U_h(s,x)&=&\E \left[h(\tl X^x_s)\right], \ \ \ \ s=0,1,2,....
\Eey
where $\tl X^x_s$ stresses that the initial value of $\tl X_s$ is $x$.

{ In this subsection, we first give the regularity of $(\tilde X_s)_{s \in \Z^{+}}$ in Lemma \ref{external}. Combining with Proposition \ref{regularity}, we prove Theorem \ref{main} by the Lindeberg principle.}

\begin{lemma}\label{external}
	Let $h \in {\rm Lip}(1)$. Choosing $\eta\leq\min\{1,\frac{\gamma\delta}{2A_{3}}\}$. Then, for any $x,y\in\mathbb{R}^{d}$, we have
	\Beys
	|U_h(s,x)-U_h(s,y)| &\le&a^{s}|x-y|,
	\Eeys
	where $a=e^{\frac{2K(L+3\gamma/4)}{\gamma\delta}-\lambda m\eta}<1$ (when $m$ is large enough).

In particular, when $K=0$, further assume $\eta\leq\gamma$, we have
\Beys
	|U_h(s,x)-U_h(s,y)| &\le&e^{-\frac{\gamma}{2}sm\eta}|x-y|.
	\Eeys
\end{lemma}
\begin{proof}
For any $x,y\in\mathbb{R}^{d}$, by (\ref{exponential1}), we have
\begin{align*}
\left|\mathbb{E}h(\tl X^{x}_s)-\mathbb{E}h(\tl X^{y}_s)\right|=&\left|\mathbb{E}h\left(X^{x}_{sm\eta}\right)-\mathbb{E}h\left(\tl X^{y}_{sm\eta}\right)\right|\\
=&\left|\mathbb{E}h\left(X^{X^{x}_{(s-1)m\eta}}_{m\eta}\right)-\mathbb{E}h\left(X^{X^{y}_{(s-1)m\eta}}_{m\eta}\right)\right|\\
=&\left|\mathbb{E}h\left(X^{\bar{X}}_{m\eta}\right)-\mathbb{E}h\left(X^{\bar{Y}}_{m\eta}\right)\right|\\
\leq&W_{1}\left(\mathcal{L}\left(X^{\bar{X}}_{m\eta}\right),\mathcal{L}\left(X^{\bar{Y}}_{m\eta}\right)\right)\leq e^{\frac{2K(L+3\gamma/4)}{\gamma\delta}} e^{-\lambda m\eta}\mathbb{E}|\bar{X}-\bar{Y}|,
\end{align*}
where $\bar{X}$ and $\bar{Y}$ are many random vectors such that $\bar{X}\stackrel{\rm d}{=}X^{x}_{(s-1)m\eta}$ and $\bar{Y}\stackrel{\rm d}{=}X^{y}_{(s-1)m\eta}$. Thus, noting that $\bar{X}$ and $\bar{Y}$ are arbitrary, it follows from the definition of the W1 distance that
\begin{align*}
\left|\mathbb{E}h(\tl X^{x}_s)-\mathbb{E}h(\tl X^{y}_s)\right|\leq e^{\frac{2K(L+3\gamma/4)}{\gamma\delta}}e^{-\lambda m\eta}\sup_{h\in{\rm Lip}(1)}|\mathbb{E}h\left(X^{x}_{(s-1)m\eta}\right)-\mathbb{E}h\left(X^{y}_{(s-1)m\eta}\right)|.
\end{align*}
Continue this process with repeatedly using (\ref{exponential1}), we finally obtain
\begin{align*}
\left|\mathbb{E}h(\tl X^{x}_s)-\mathbb{E}h(\tl X^{y}_s)\right|\leq e^{s\left(\frac{2K(L+3\gamma/4)}{\gamma\delta}-\lambda m\eta\right)}|x-y|.
\end{align*}

\end{proof}

\begin{lemma}\label{exSGD}
Let $\tilde{\omega}_{s}=\omega_{sm}$ and $\eta\leq\min\{1,\big( \frac{\gamma}{432L^{4}}\big)^{\frac{1}{3}},\frac{\gamma}{96L^{2}},(\frac{\gamma}{576L^{3}})^{\frac{1}{2}},\frac{\gamma}{\sqrt{6(1+\gamma)}10^2L^2}\}$. Then, for any $0\leq k\leq m$, we have
\begin{align}\label{exSGDforth}
\mathbb{E}|\tilde{\omega}_{s}|^{4}\leq \mathbb{E}|\tilde{\omega}_{0}|^{4}+\frac{C_{\gamma,d,|\nabla P(0)|,L}(1+K^{2})}{1-\rho}.
\end{align}
where $\rho=(1-\gamma\eta)^{m}+\frac{1}{2}<1$ (when $m$ is large enough).
\end{lemma}

\begin{proof}
From the proof of Lemma \ref{coupling2} below, it is easy to verify that
\begin{align*}
\mathbb{E}|\tilde{\omega}_{1}|^{4}\leq \big[(1-\gamma\eta)^{m}+3\times10^{4}(\frac{1}{\gamma}+\frac{1}{\gamma^2})L^{4}\eta^{2}\big]\mathbb{E}|\tilde{\omega}_{0}|^{4}+C_{\gamma,d,|\nabla P(0)|,L}(1+K^{2}),
\end{align*}
noticing that $\eta<\frac{\gamma}{\sqrt{6(1+\gamma)}10^2L^2}$, we have
\begin{align*}
\mathbb{E}|\tilde{\omega}_{1}|^{4}\leq\rho\mathbb{E}|\tilde{\omega}_{0}|^{4}+C_{\gamma,d,|\nabla P(0)|,L}(1+K^{2}).
\end{align*}
Inductively, we have
\begin{align*}
\mathbb{E}|\tilde{\omega}_{s}|^{4}\leq\rho^{s}\mathbb{E}|\tilde{\omega}_{0}|^{4}+C_{\gamma,d,|\nabla P(0)|,L}(1+K^{2})\sum_{k=0}^{s-1}\rho^{k}\leq\mathbb{E}|\tilde{\omega}_{0}|^{4}+\frac{C_{\gamma,d,|\nabla P(0)|,L}(1+K^{2})}{1-\rho}.
\end{align*}
\end{proof}






\subsection{Proofs of Theorem \ref{main} and Corollary \ref{Coro}}
\begin{proof}[{\bf Proof of Theorem \ref{main}}]
By the same argument as the proof of (\ref{definition2}), we have
\begin{align*}
\mathbb{E}h(\tilde X_s)-\mathbb{E}h(\tl \omega_{s})=\sum_{i=1}^{s}\left[\mathbb{E}U_{h}\big(s-i,\tl X^{\tl \omega_{i-1}}_{1}\big)-\mathbb{E}U_{h}\big(s-i,\tl \omega^{\tl \omega_{i-1}}_{1}\big)\right].
\end{align*}
Then, according to Proposition \ref{regularity} and Lemma \ref{external} with ${\rm Lip}(1)$ function replaced by ${\rm Lip}\big(a^{s-i}\big)$ function for $i=1,2,\cdots,s$, we have
\begin{align*}
	&|\mathbb{E}h(\tilde X_s)-\mathbb{E}h(\tl \omega_{s})|\\
\le&\sum_{i=1}^{s}\Big|\mathbb{E}U_{h}\big(s-i,\tl X^{\tl \omega_{i-1}}_{1}\big)-\mathbb{E}U_{h}\big(s-i,\tl \omega^{\tl \omega_{i-1}}_{1}\big)\Big| \\
	\le&C_{A,\gamma,d,|\nabla P(0)|,L}(\eta\delta)^{\frac{1}{2}}\left(1+|\ln\eta|+\frac{\delta}{\eta^{\frac{1}{4}}}+\lambda^{-2}e^{\lambda}\right)\sum_{i=1}^{s}a^{s-i}\big(1+\mathbb{E}|\tilde{\omega}_{i-1}|^{4}+K^{2}\big)^{\frac{7}{4}}.
\end{align*}
Thanks to Lemma \ref{exSGD}, we have
\begin{align*}
&|\mathbb{E}h(\tilde X_s)-\mathbb{E}h(\tl \omega_{s})|\\
\le&C_{A,\gamma,d,|\nabla P(0)|,L}(\eta\delta)^{\frac{1}{2}}\left(1+|\ln\eta|+\frac{\delta}{\eta^{\frac{1}{4}}}+\lambda^{-2}e^{\lambda}\right)
\big(\mathbb{E}|\tilde{\omega}_{0}|^{4}+\frac{1+K^{2}}{1-\rho}\big)^{\frac{7}{4}}\sum_{i=1}^{s}a^{s-i}\\
\leq&C_{A,\gamma,d,|\nabla P(0)|,L}(\eta\delta)^{\frac{1}{2}}\left(1+|\ln\eta|+\frac{\delta}{\eta^{\frac{1}{4}}}+\lambda^{-2}e^{\lambda}\right)
\big(\mathbb{E}|\tilde{\omega}_{0}|^{4}+\frac{1+K^{2}}{1-\rho}\big)^{\frac{7}{4}}\frac{1}{1-a}.
\end{align*}
The proof is complete.
\end{proof}

\begin{proof}[{\bf Proof of Corollary \ref{Coro2}}]
Following the assumption, a straight calculation implies $\lambda=\frac\gamma4\eta^{\frac14}$, thus $\lambda^{-2}e^\lambda(\eta\delta)^{1/2}$ converges to $0$. The condition of $m$ guarantees $a<1$ and we get the result.
\end{proof}
\begin{proof}[{\bf Proof of Corollary \ref{Coro}}]
The fact $\frac K\delta\le1$ implies $\frac4\gamma \exp\{-\frac{2(L+3\gamma/4)}{\gamma}\}\le\lambda\le \frac4\gamma$ and thus $\lambda^{-2}e^\lambda$ is bounded. We can get the result following the Theorem \ref{main}.
\end{proof}

\section{Conclusion}\label{conclusion}
{
In this paper, we establish a weak approximation error bound (in W1 distance) between SVRG-LD and SDDE under the assumptions which are true for a large family of nonconvex minimization problems, from this bound we can see that the approximation error can tends to $0$ by carefully choosing the inverse temperature parameter $\delta$ according to the learning rate $\eta$. When the optimization problem is convex, combining our result with the exponential convergence result in \cite{DK16}, we know that our approximating SDDE will be close to the minimizer $\omega^*$  exponentially fast. Furthermore, we can clearly see the variance reduction effect of the SVRG-LD from the analysis of the SDDE.

Our approach to proving the main results is by a refined Lindeberg principle and Malliavin calculus. This method provides a mathematical framework for weak approximations, and we hope it can be applied to study other stochastic algorithms. As is shown in Theorem \ref{main}, the error bound will explode if the inverse temperature parameter $\delta$ is not specifically chosen according to the learning rate $\eta$. We will attempt to solve this problem in the future research, one possible way is to use the frame of Stein's method developed in \cite{F-S-X}. Another interesting research direction is borrowing the well established properties of SDDEs to develop stochastic algorithms with variance reduction effect.}

\begin{appendix}

\section{proofs of Example \ref{ex1}}\label{Asecond}

In this section, we will verify Assumptions \ref{assum1}, \ref{assum3} and \ref{assum4} for Examples \ref{ex2} and \ref{ex1}, respectively.

\subsection{Example \ref{ex2}}

Recall $a_{i}\sim\mathcal{N}(0,I_{d})$ for $i=1,\cdots,n$,
\begin{align*}
\mathbb{E}|\nabla\psi_{I}(x)-\nabla\psi_{I}(y)|^{4}=\frac{1}{n}\sum_{i=1}^{n}\left|H(x-y)+\tilde{Q}{\rm diag}(a_{i})\tilde{Q}^{T}(x-y)\right|^{4},
\end{align*}
\begin{align*}
\nabla P(x)=Hx+\frac{1}{n}\sum_{i=1}^{n}\tilde{Q}{\rm diag}(a_{i})\tilde{Q}^{T}x,
\end{align*}
\begin{align*}
\Sigma(x,y)=&\frac{1}{n}\sum_{i=1}^{n}\Big(\tilde{Q}[D+{\rm diag}(a_{i})]\tilde{Q}^{T}(x-y)\Big)\Big(\tilde{Q}[D+{\rm diag}(a_{i})]\tilde{Q}^{T}(x-y)\Big)^{T}\\
&-\Big(\frac{1}{n}\sum_{i=1}^{n}\tilde{Q}[D+{\rm diag}(a_{i})]\tilde{Q}^{T}(x-y)\Big)\Big(\frac{1}{n}\sum_{i=1}^{n}\tilde{Q}[D+{\rm diag}(a_{i})]\tilde{Q}^{T}(x-y)\Big)^{T}.
\end{align*}
and $Q_{\eta,\delta}(x,y)=\big(\Sigma(x,y)+\frac{\delta}{\eta}I_{d}\big)^{\frac{1}{2}}$.

Inspired from the law of large numbers, we first give the following lemma.

\begin{lemma}\label{highp1}
For any $\epsilon>0$ and unit vectors $v,v_{1},v_{2}\in\mathbb{R}^{d}$, we have
\begin{align}\label{ineq21}
\mathbb{P}\left(\left|\frac{1}{n}\sum_{i=1}^{n}\tilde{Q}{\rm diag}(a_{i})\tilde{Q}^{T}v\right|>\epsilon\right)\leq\frac{d}{n\epsilon^{2}},
\end{align}
\begin{align}\label{ineq22}
\mathbb{P}\left(\left|\frac{1}{n}\sum_{i=1}^{n}|a_{i}|^{4}-\mathbb{E}|a_{1}|^{4}\right|>\epsilon\right)\leq\frac{\mathbb{E}|a_{1}|^{8}}{n\epsilon^{2}}
\end{align}
and
\begin{align}\label{ineq23}
\mathbb{P}\left(\left|\frac{1}{n}\sum_{i=1}^{n}\tilde{Q}{\rm diag}(a_{i})\tilde{Q}^{T}v_{1}\left(\tilde{Q}{\rm diag}(a_{i})\tilde{Q}^{T}v_{2}\right)^{T}v-\tilde{Q}{\rm diag}(\tilde{Q}v_{1}){\rm diag}(\tilde{Q}v_{2})\tilde{Q}^{T}v\right|>\epsilon\right)\leq\frac{6d^{6}}{n\epsilon^{2}}.
\end{align}
\end{lemma}
\begin{proof}
By Chebyshev's inequality, we have
\begin{align*}
\mathbb{P}\left(\left|\frac{1}{n}\sum_{i=1}^{n}\tilde{Q}{\rm diag}(a_{i})\tilde{Q}^{T}v\right|>\epsilon\right)
\leq\frac{\sum_{i=1}^{n}\mathbb{E}|a_{i}|^{2}}{n^{2}\epsilon^{2}}=\frac{d}{n\epsilon^{2}}
\end{align*}
and
\begin{align*}
\mathbb{P}\left(\left|\frac{1}{n}\sum_{i=1}^{n}|a_{i}|^{4}-\mathbb{E}|a_{1}|^{4}\right|>\epsilon\right)
\leq&\frac{\sum_{i=1}^{n}\mathbb{E}\left||a_{i}|^{4}-\mathbb{E}|a_{1}|^{4}\right|^{2}}{n^{2}\epsilon^{2}}\leq\frac{2\mathbb{E}|a_{1}|^{8}}{n\epsilon^{2}}.
\end{align*}
Notice that $\mathbb{E}\left[\tilde{Q}{\rm diag}(a_{i})\tilde{Q}^{T}v_{1}\left(\tilde{Q}{\rm diag}(a_{i})\tilde{Q}^{T}v_{2}\right)^{T}v\right]=\tilde{Q}{\rm diag}(\tilde{Q}v_{1}){\rm diag}(\tilde{Q}v_{2})\tilde{Q}^{T}v$, $\|\tilde{Q}\|_{{\rm{HS}}}=\sqrt{d}$ and $\mathbb{E}|a_{1}|^{4}=3d^{2}$, by Chebyshev's inequality, we have
\begin{align*}
&\mathbb{P}\left(\left|\frac{1}{n}\sum_{i=1}^{n}\tilde{Q}{\rm diag}(a_{i})\tilde{Q}^{T}v_{1}\left(\tilde{Q}{\rm diag}(a_{i})\tilde{Q}^{T}v_{2}\right)^{T}v-\tilde{Q}{\rm diag}(\tilde{Q}v_{1}){\rm diag}(\tilde{Q}v_{2})\tilde{Q}^{T}v\right|>\epsilon\right)\\
\leq&\sum_{i=1}^{n}\mathbb{E}\left|\tilde{Q}{\rm diag}(a_{i})\tilde{Q}^{T}v_{1}\left(\tilde{Q}{\rm diag}(a_{i})\tilde{Q}^{T}v_{2}\right)^{T}v-\tilde{Q}{\rm diag}(\tilde{Q}v_{1}){\rm diag}(\tilde{Q}v_{2})\tilde{Q}^{T}v\right|^{2}\Big/(n\epsilon)^{2}\\
\leq&\frac{2\|\tilde{Q}\|_{{\rm HS}}^{8}\mathbb{E}|a_{1}|^{4}}{n\epsilon^{2}}=\frac{6d^{6}}{n\epsilon^{2}}.
\end{align*}
\end{proof}

Then, we first verify the Assumptions \ref{assum1} and \ref{assum3}.
\begin{lemma}
Denote the smallest eigenvalue of $H$ by $\lambda_{\min}$. When $n$ is large enough, for any $x,y \in \R^{d}$ and unit vectors $v\in\mathbb{R}^{d}$, the following inequalities
\begin{align*}
\langle\nabla^{2} P(x)v,v\rangle\geq\lambda_{\min}|v|^{2}
\end{align*}
and
\begin{align*}
\mathbb{E}|\nabla\psi_{I}(x)-\nabla\psi_{I}(y)|^{4}\leq8(\|H\|_{{\rm HS}}^{4}+3d^{6})|x-y|^{4}
\end{align*}
holds with high probability.
\end{lemma}

\begin{proof}
For any $x\in\mathbb{R}^{d}$, by (\ref{ineq21}), it is easy to verify that with high probability, we have
\begin{align*}
\nabla P(x)=Hx,
\end{align*}
which implies
\begin{align*}
\langle\nabla^{2} P(x)v,v\rangle\geq\lambda_{\min}|v|^{2}.
\end{align*}
For any $x,y\in\mathbb{R}^{d}$, we have
\begin{align*}
\mathbb{E}|\nabla\psi_{I}(x)-\nabla\psi_{I}(y)|^{4}=&\frac{1}{n}\sum_{i=1}^{n}\left|H(x-y)+\tilde{Q}{\rm diag}(a_{i})\tilde{Q}^{T}(x-y)\right|^{4}\\
\leq&8\|H\|_{{\rm HS}}^{4}|x-y|^{4}+\frac{8}{n}\|\tilde{Q}\|_{{\rm HS}}^{8}\sum_{i=1}^{n}|a_{i}|^{4}|x-y|^{4},
\end{align*}
by (\ref{ineq22}), with high probability, we further have
\begin{align*}
\mathbb{E}|\nabla\psi_{I}(x)-\nabla\psi_{I}(y)|^{4}\leq8(\|H\|_{{\rm HS}}^{4}+3d^{6})|x-y|^{4}.
\end{align*}
\end{proof}

Next, with the help of Lemma \ref{highp1}, we will calculate the matrix $Q_{\eta,\delta}(x,y)$ directly.

\begin{lemma}\label{direct}
When $n$ is large enough, for any $x,y\in\mathbb{R}^{d}$, the following equality
\begin{align*}
Q_{\eta,\delta}(x,y)=\tilde{Q}\left[{\rm diag}(\tilde{Q}(x-y))^{2}+\frac{\delta}{\eta}I_{d}\right]^{\frac{1}{2}}\tilde{Q}^{T}
\end{align*}
holds with high probability.
\end{lemma}

\begin{proof}
Since $\frac{1}{n}\sum_{i=1}^{n}\Big(\tilde{Q}[D+{\rm diag}(a_{i})]\tilde{Q}^{T}0\Big)\Big(\tilde{Q}[D+{\rm diag}(a_{i})]\tilde{Q}^{T}0\Big)^{T}0=0$, by integration and (\ref{ineq23}), the following equality
\begin{align*}
\frac{1}{n}\sum_{i=1}^{n}\tilde{Q}{\rm diag}(a_{i})\tilde{Q}^{T}(x-y)\left(\tilde{Q}{\rm diag}(a_{i})\tilde{Q}^{T}(x-y)\right)^{T}=\tilde{Q}{\rm diag}(\tilde{Q}(x-y))^{2}\tilde{Q}^{T}.
\end{align*}
By (\ref{ineq21}), with high probability, we have $\frac{1}{n}\sum_{i=1}^{n}\tilde{Q}{\rm diag}(a_{i})\tilde{Q}^{T}x=0.$ These imply that with high probability,
\begin{align*}
\Sigma(x,y)=&\frac{1}{n}\sum_{i=1}^{n}\Big(\tilde{Q}[D+{\rm diag}(a_{i})]\tilde{Q}^{T}(x-y)\Big)\Big(\tilde{Q}[D+{\rm diag}(a_{i})]\tilde{Q}^{T}(x-y)\Big)^{T}\\
&-\Big(\frac{1}{n}\sum_{i=1}^{n}\tilde{Q}[D+{\rm diag}(a_{i})]\tilde{Q}^{T}(x-y)\Big)\Big(\frac{1}{n}\sum_{i=1}^{n}\tilde{Q}[D+{\rm diag}(a_{i})]\tilde{Q}^{T}(x-y)\Big)^{T}\\
=&\tilde{Q}{\rm diag}(\tilde{Q}(x-y))^{2}\tilde{Q}^{T}.
\end{align*}
Hence, we have
\begin{align*}
Q_{\eta,\delta}(x,y)=\left(\Sigma(x,y)+\frac{\delta}{\eta}I_{d}\right)^{\frac{1}{2}}=\tilde{Q}\left[{\rm diag}(\tilde{Q}(x-y))^{2}+\frac{\delta}{\eta}I_{d}\right]^{\frac{1}{2}}\tilde{Q}^{T}.
\end{align*}
\end{proof}

With the above results, we will verify Assumption \ref{assum4}.

\begin{lemma}
When $n$ is large enough, for any $x,y \in \R^{d}$ and unit vectors $v_{i}\in\mathbb{R}^{d}$, $i=1,2,3$, the following inequalities
\begin{align*}
|\nabla_{v_2}\nabla_{v_1} \nabla P(x)|=0,\quad
|\nabla_{v_3}\nabla_{v_2}\nabla_{v_1} \nabla P(x)|=0,
\end{align*}
\begin{align*}
\|\nabla_{1,v_{1}} Q_{\eta,\delta}(x,y)\|_{\rm HS} \ \le d^{2},\quad \|\nabla_{2,v_{1}} Q_{\eta,\delta}(x,y)\|_{\rm HS}\ \le d^{2}
\end{align*}
and
\begin{align*}
\|\nabla_{1,v_{1}} \nabla_{1,v_{2}}Q_{\eta,\delta}(x,y)\|_{\rm HS} \ \le(\frac{\eta}{\delta})^{\frac{1}{2}}d^{2}, \quad \|\nabla_{1,v_{1}} \nabla_{1,v_{2}} \nabla_{1,v_{3}}Q_{\eta,\delta}(x,y)\|_{\rm HS} \ \le 3\frac{\eta}{\delta}d^{3}.
\end{align*}
hold with high probability.
\end{lemma}

\begin{proof}
For any $x\in\mathbb{R}^{d}$, by (\ref{ineq21}), the equality
\begin{align*}
\nabla P(x)=Hx
\end{align*}
holds with high probability, which implies
\begin{align*}
\nabla_{v_2}\nabla_{v_1} \nabla P(x)=0, \quad \nabla_{v_3}\nabla_{v_2}\nabla_{v_1} \nabla P(x)=0.
\end{align*}
By Lemma \ref{direct}, it is straightforward to calculate that with high probability, we have
\begin{align*}
\|\nabla_{1,v_{1}} Q_{\eta,\delta}(x,y)\|_{\rm HS}\leq\|\tilde{Q}\|_{{\rm HS}}^{3}\sqrt{d}=d^{2}, \quad \|\nabla_{2,v_{1}} Q_{\eta,\delta}(x,y)\|_{\rm HS}\leq d^{2},
\end{align*}
\begin{align*}
\|\nabla_{1,v_{1}} \nabla_{1,v_{2}}Q_{\eta,\delta}(x,y)\|_{\rm HS} \ \le\|\tilde{Q}\|_{{\rm HS}}^{4}(\frac{\eta}{\delta})^{\frac{1}{2}}\sqrt{d}=(\frac{\eta}{\delta})^{\frac{1}{2}}d^{\frac{5}{2}}
\end{align*}
and
\begin{align*}
\|\nabla_{1,v_{1}} \nabla_{1,v_{2}} \nabla_{1,v_{3}}Q_{\eta,\delta}(x,y)\|_{\rm HS} \ \le3\|\tilde{Q}\|_{{\rm HS}}^{5}\frac{\eta}{\delta}\sqrt{d}=3\frac{\eta}{\delta}d^{3}.
\end{align*}
\end{proof}

\subsection{Example \ref{ex1}}

Recall
\begin{align*}
\mathbb{E}|\nabla\psi_{I}(x)-\nabla\psi_{I}(y)|^{4}=\frac{1}{n}\sum_{i=1}^{n}\left|a_{i}(\frac{1}{1+e^{-a_{i}^{T}x}}-\frac{1}{1+e^{-a_{i}^{T}y}})+\lambda(x-y)\right|^{4},
\end{align*}
\begin{align*}
\nabla P(x)=\frac{1}{n}\sum_{i=1}^{n}a_{i}(\frac{1}{1+e^{-a_{i}^{T}x}}-b_{i})+\lambda x,
\end{align*}
\begin{align*}
\Sigma(x,y)=&\frac{1}{n}\sum_{i=1}^{n}\Big(a_{i}(\frac{1}{1+e^{-a_{i}^{T}x}}-\frac{1}{1+e^{-a_{i}^{T}y}})\Big)\Big(a_{i}(\frac{1}{1+e^{-a_{i}^{T}x}}-\frac{1}{1+e^{-a_{i}^{T}y}})\Big)^{T}\\
&-\Big(\frac{1}{n}\sum_{i=1}^{n}a_{i}(\frac{1}{1+e^{-a_{i}^{T}x}}-\frac{1}{1+e^{-a_{i}^{T}y}})\Big)\Big(\frac{1}{n}\sum_{i=1}^{n}a_{i}(\frac{1}{1+e^{-a_{i}^{T}x}}-\frac{1}{1+e^{-a_{i}^{T}y}})\Big)^{T}.
\end{align*}
and $Q_{\eta,\delta}(x,y)=\big(\Sigma(x,y)+\frac{\delta}{\eta}I_{d}\big)^{\frac{1}{2}}$.

Inspired from the law of large numbers, we immediately have the following lemma.

\begin{lemma}
In Example \ref{ex1}, suppose $\mathbb{E}|a_{i}|^{16}<\infty$, $i=1,\cdots,n$. Then, for any $\epsilon>0$ and $x,y\in \R^{d}$, we have
\begin{align}\label{ineq1}
\mathbb{P}\Big(\Big|\frac{1}{n}\sum_{i=1}^{n}\frac{a_{i}}{1+e^{-a_{i}^{T}x}}
-\mathbb{E}\big[\frac{a_{1}}{1+e^{-a_{1}^{T}x}}\big]\Big|>\epsilon\Big)\leq\frac{2\mathbb{E}|a_{1}|^{2}}{n\epsilon^{2}},
\end{align}
\begin{align}\label{ineq2}
\mathbb{P}\Big(\big|\frac{1}{n}\sum_{i=1}^{n}a_{i}b_{i}-\mathbb{E}[a_{i}b_{i}]\big|>\epsilon\Big)\leq\frac{2\mathbb{E}|a_{1}|^{2}}{n\epsilon^{2}},
\end{align}
\begin{align}\label{ineq3}
\mathbb{P}\left(\left|\frac{1}{n}\sum_{i=1}^{n}|a_{i}|^{8}-\mathbb{E}|a_{1}|^{8}\right|>\epsilon\right)\leq\frac{2\mathbb{E}|a_{1}|^{16}}{n\epsilon^{2}},
\end{align}
and
\begin{align}\label{ineq4}
&\mathbb{P}\Big(\Big\|\frac{1}{n}\sum_{i=1}^{n}\big(a_{i}(\frac{1}{1+e^{-a_{i}^{T}x}}-\frac{1}{1+e^{-a_{i}^{T}y}})\big)
\big(a_{i}(\frac{1}{1+e^{-a_{i}^{T}x}}-\frac{1}{1+e^{-a_{i}^{T}y}})\big)^{T}\nonumber\\
&\qquad-\mathbb{E}\big[\big(a_{1}(\frac{1}{1+e^{-a_{1}^{T}x}}-\frac{1}{1+e^{-a_{1}^{T}y}})\big)
\big(a_{1}(\frac{1}{1+e^{-a_{1}^{T}x}}-\frac{1}{1+e^{-a_{1}^{T}y}})\big)^{T}\big]\Big\|_{{\rm HS}}>\epsilon\Big)\nonumber\\
\leq&\frac{2\mathbb{E}|a_{1}|^{4}}{n\epsilon^{2}}.
\end{align}
\end{lemma}

\begin{proof}

We only give the proof of the first inequality, the other three are similar. By Chebyshev's inequality, we have
\begin{align*}
&\mathbb{P}\Big(\Big|\frac{1}{n}\sum_{i=1}^{n}\frac{a_{i}}{1+e^{-a_{i}^{T}x}}
-\mathbb{E}\big[\frac{a_{1}}{1+e^{-a_{1}^{T}x}}\big]\Big|>\epsilon\Big)\\
=&\mathbb{P}\left(\left|\frac{1}{n}\sum_{i=1}^{n}\left[\frac{a_{i}}{1+e^{-a_{i}^{T}x}}
-\mathbb{E}\big[\frac{a_{1}}{1+e^{-a_{1}^{T}x}}\big]\right]\right|>\epsilon\right)\\
\leq&\frac{\sum_{i=1}^{n}\mathbb{E}\left|\frac{a_{i}}{1+e^{-a_{i}^{T}x}}
-\mathbb{E}\big[\frac{a_{1}}{1+e^{-a_{1}^{T}x}}\big]\right|^{2}}{n^{2}\epsilon^{2}}
\leq\frac{2\mathbb{E}|a_{1}|^{2}}{n\epsilon^{2}}.
\end{align*}
\end{proof}

From above lemma, it is easily seen that with high probability, we have
\begin{align*}
\nabla P(x)=\mathbb{E}\left[a_{1}(\frac{1}{1+e^{-a_{1}^{T}x}}-b_{1})\right]+\lambda x.
\end{align*}

Then, we first verify the Assumptions \ref{assum1} and \ref{assum3}.
\begin{lemma}
In Example \ref{ex1}, suppose $\mathbb{E}|a_{i}|^{16}<\infty$, $i=1,\cdots,n$. When $n$ is large enough, for any $x,y \in \R^{d}$ and unit vectors $v\in\mathbb{R}^{d}$, we have
\begin{align*}
\langle\nabla^{2} P(x)v,v\rangle\geq\lambda|v|^{2},
\end{align*}
and the following inequality
\begin{align*}
\mathbb{E}|\nabla\psi_{I}(x)-\nabla\psi_{I}(y)|^{4}\leq8(\mathbb{E}|a_{1}|^{8}+\lambda^{4})|x-y|^{4}
\end{align*}
holds with high probability.
\end{lemma}

\begin{proof}
For any $x\in\mathbb{R}^{d}$, notice
\begin{align*}
\nabla^{2}P(x)=\frac{1}{n}\sum_{i=1}^{n}\frac{a_{i}a_{i}^{T}e^{-a_{i}^{T}x}}{(1+e^{-a_{i}^{T}x})^{2}}+\lambda I_{d},
\end{align*}
since $\frac{1}{n}\sum_{i=1}^{n}\frac{a_{i}a_{i}^{T}}{(1+e^{-a_{i}^{T}x})^{2}}$ is a positive semi-definite matrix, we have
\begin{align*}
\langle\nabla^{2} P(x)v,v\rangle\geq\lambda|v|^{2}.
\end{align*}
For any $x,y\in\mathbb{R}^{d}$, we have
\begin{align*}
\mathbb{E}|\nabla\psi_{I}(x)-\nabla\psi_{I}(y)|^{4}=&\frac{1}{n}\sum_{i=1}^{n}\left|a_{i}(\frac{1}{1+e^{-a_{i}^{T}x}}-\frac{1}{1+e^{-a_{i}^{T}y}})+\lambda(x-y)\right|^{4}\\
=&\frac{8}{n}\sum_{i=1}^{n}\left|a_{i}(\frac{1}{1+e^{-a_{i}^{T}x}}-\frac{1}{1+e^{-a_{i}^{T}y}})\right|^{4}+8\lambda^{4}|x-y|^{4}\\
\leq&\frac{8}{n}\sum_{i=1}^{n}|a_{i}|^{8}|x-y|+8\lambda^{4}|x-y|^{4},
\end{align*}
by (\ref{ineq3}), with high probability, we further have
\begin{align*}
\mathbb{E}|\nabla\psi_{I}(x)-\nabla\psi_{I}(y)|^{4}\leq8(\mathbb{E}|a_{1}|^{8}+\lambda^{4})|x-y|^{4},
\end{align*}
the desired result follows.
\end{proof}

Before verifying Assumption \ref{assum4}, we first give the following lemma, which plays a crucial role in proving Assumption \ref{assum4}. We first state the following lemma, which gives the upper bounds of the derivatives of $Q_{\eta,\delta}(x,y)$. Since $x,y$ in the function $Q_{\eta,\delta}(x,y)$ is symmetric, for simplicity, we only consider the univariate function in the following.

\begin{lemma}\label{taylor}
Let $\hat{\Sigma}(x)$ be a symmetric positive function matrix and $\hat{\Sigma}(x)\in\mathcal{C}^{3}(\mathbb{R}^{d},\mathbb{R}^{d\times d}).$ Then, for any $x,v,v_{1},v_{2},v_{3}\in\mathbb{R}^{d},$ we have
\begin{align*}
\|\nabla_{v}\hat{\Sigma}(x)^{\frac{1}{2}}\|_{{\rm HS}}\leq\frac{1}{2}\lambda_{\min}\big(\hat{\Sigma}(x)\big)^{-\frac{1}{2}}\|\nabla_{v}\hat{\Sigma}(x)\|_{{\rm HS}},
\end{align*}
\begin{align*}
&\|\nabla_{v_{1}}\nabla_{v_{2}}\hat{\Sigma}(x)^{\frac{1}{2}}\|_{{\rm HS}}\nonumber\\
\leq&\frac{1}{4}d^{\frac{1}{2}}\lambda_{\min}\big(\hat{\Sigma}(x)\big)^{-\frac{3}{2}}\|\nabla_{v_{1}}\hat{\Sigma}(x)\|_{{\rm HS}}\|\nabla_{v_{2}}\hat{\Sigma}(x)\|_{{\rm HS}}+\frac{1}{2}\lambda_{\min}\big(\hat{\Sigma}(x)\big)^{-\frac{1}{2}}\|\nabla_{v_{1}}\nabla_{v_{2}}\hat{\Sigma}(x)\|_{{\rm HS}}
\end{align*}
and
\begin{align*}
&\big\|\nabla_{v_{1}}\nabla_{v_{2}}\nabla_{v_{3}}\big(\hat{\Sigma}(x)^{\frac{1}{2}}\big)\big\|_{{\rm HS}}\nonumber\\
\leq&\frac{1}{4}d^{\frac{1}{2}}\lambda_{\min}\big(\hat{\Sigma}(x)\big)^{-\frac{3}{2}}\big(\|\nabla_{v_{2}}\hat{\Sigma}(x)\|_{{\rm HS}}\|\nabla_{v_{1}}\nabla_{v_{3}}\hat{\Sigma}(x)\|_{{\rm HS}}+\|\nabla_{v_{1}}\nabla_{v_{2}}\hat{\Sigma}(x)\|_{{\rm HS}}\|\nabla_{v_{3}}\hat{\Sigma}(x)\|_{{\rm HS}}\big)\nonumber\\
&+\frac{3}{8}d\lambda_{\min}\big(\hat{\Sigma}(x)\big)^{-\frac{5}{2}}\|\nabla_{v_{1}}\hat{\Sigma}(x)\|_{{\rm HS}}\|\nabla_{v_{2}}\hat{\Sigma}(x)\|_{{\rm HS}}\|\nabla_{v_{3}}\hat{\Sigma}(x)\|_{{\rm HS}}\nonumber\\
&+\frac{1}{4}d^{\frac{1}{2}}\lambda_{\min}\big(\hat{\Sigma}(x)\big)^{-\frac{3}{2}}\|\nabla_{v_{1}}\hat{\Sigma}(x)\|_{{\rm HS}}\|\nabla_{v_{2}}\nabla_{v_{3}}\hat{\Sigma}(x)\|_{{\rm HS}}\nonumber\\
&+\frac{1}{2}\lambda_{\min}\big(\hat{\Sigma}(x)\big)^{-\frac{1}{2}}\|\nabla_{v_{1}}\nabla_{v_{2}}\nabla_{v_{3}}\hat{\Sigma}(x)\|_{{\rm HS}},
\end{align*}
where $\lambda_{\min}\big(\hat{\Sigma}(x)\big)$ is the smallest eigenvalue of $\hat{\Sigma}(x).$
\end{lemma}

\begin{proof}
For simplicity, denote the square root function $\varphi(\hat{\Sigma})=\hat{\Sigma}^{\frac{1}{2}}$ for any symmetric positive matrix $\hat{\Sigma}.$ Then, \cite[Theorem 1.1]{D-N} shows that for $n\geq0$
\begin{align}\label{derivative}
\||\nabla^{n+1}\varphi(\hat{\Sigma})|\|\leq d^{\frac{n}{2}}\frac{(2n)!}{n!}2^{-(2n+1)}\lambda_{\min}(\hat{\Sigma})^{-(n+\frac{1}{2})},
\end{align}
where the multi-linear operator norm $\||\cdot|\|$ is defined by
$$
\||\nabla^{n+1}\varphi(\hat{\Sigma})|\|=\sup_{\|H\|_{{\rm HS}}=1}\|\nabla^{n+1}\varphi(\hat{\Sigma})H\|_{{\rm HS}}.
$$
Therefore, by Chain rule (see, e.g., \cite[Proposition 3.2]{A-B-D}), we have
\begin{align*}
\big\|\nabla_{v}\big(\hat{\Sigma}(x)^{\frac{1}{2}}\big)\big\|_{{\rm HS}}=\big\|\nabla_{v}\varphi\big(\hat{\Sigma}(x)\big)\big\|_{{\rm HS}}=&\big\|\nabla\varphi\big(\hat{\Sigma}(x)\big)\nabla_{v}\hat{\Sigma}(x)\big\|_{{\rm HS}}\\
=&\big\|\big|\nabla\varphi\big(\hat{\Sigma}(x)\big)\big|\big\|\|\nabla_{v}\hat{\Sigma}(x)\|_{{\rm HS}}\\
\leq&\frac{1}{2}\lambda_{\min}\big(\hat{\Sigma}(x)\big)^{-\frac{1}{2}}\|\nabla_{v}\hat{\Sigma}(x)\|_{{\rm HS}}.
\end{align*}
By Product rule (see, e.g., \cite[Proposition 3.3]{A-B-D}) and Chain rule, we have
\begin{align*}
\big\|\nabla_{v_{1}}\nabla_{v_{2}}\big(\hat{\Sigma}(x)^{\frac{1}{2}}\big)\big\|_{{\rm HS}}=&\Big\|\nabla_{v_{1}}\Big(\nabla\varphi\big(\hat{\Sigma}(x)\big)\nabla_{v_{2}}\hat{\Sigma}(x)\Big)\Big\|_{{\rm HS}}\\
=&\big\|\nabla^{2}\varphi\big(\hat{\Sigma}(x)\big)\nabla_{v_{1}}\hat{\Sigma}(x)\nabla_{v_{2}}\hat{\Sigma}(x)+\nabla\varphi\big(\hat{\Sigma}(x)\big)\nabla_{v_{1}}\nabla_{v_{2}}\hat{\Sigma}(x)\big\|_{{\rm HS}}\\
\leq&\big\|\nabla^{2}\varphi\big(\hat{\Sigma}(x)\big)\nabla_{v_{1}}\hat{\Sigma}(x)\nabla_{v_{2}}\hat{\Sigma}(x)\big\|_{{ \rm{HS}}}+\big\|\nabla\varphi\big(\hat{\Sigma}(x)\big)\nabla_{v_{1}}\nabla_{v_{2}}\hat{\Sigma}(x)\big\|_{{\rm HS}},
\end{align*}
then, by (\ref{derivative}) and Cauchy-Schwarz inequality, we have
\begin{align*}
&\|\nabla_{v_{1}}\nabla_{v_{2}}\hat{\Sigma}(x)^{\frac{1}{2}}\|_{{\rm HS}}\\
\leq&\big\|\big|\nabla^{2}\varphi\big(\hat{\Sigma}(x)\big)\big|\big\|\|\nabla_{v_{1}}\hat{\Sigma}(x)\nabla_{v_{2}}\hat{\Sigma}(x)\|_{{\rm HS}}
+\big\|\big|\nabla\varphi\big(\hat{\Sigma}(x)\big)\big|\big\|\|\nabla_{v_{1}}\nabla_{v_{2}}\hat{\Sigma}(x)\|_{{\rm HS}}\\
\leq&\frac{1}{4}d^{\frac{1}{2}}\lambda_{\min}\big(\hat{\Sigma}(x)\big)^{-\frac{3}{2}}\|\nabla_{v_{1}}\hat{\Sigma}(x)\|_{{\rm HS}}\|\nabla_{v_{2}}\hat{\Sigma}(x)\|_{{\rm HS}}+\frac{1}{2}\lambda_{\min}\big(\hat{\Sigma}(x)\big)^{-\frac{1}{2}}\|\nabla_{v_{1}}\nabla_{v_{2}}\hat{\Sigma}(x)\|_{{\rm HS}}.
\end{align*}
Furthermore, notice
\begin{align*}
\nabla_{v_{1}}\nabla_{v_{2}}\nabla_{v_{3}}\big(\hat{\Sigma}(x)^{\frac{1}{2}}\big)
=\nabla_{v_{1}}\Big(\nabla^{2}\varphi\big(\hat{\Sigma}(x)\big)\nabla_{v_{2}}\hat{\Sigma}(x)\nabla_{v_{3}}\hat{\Sigma}(x)+\nabla\varphi\big(\hat{\Sigma}(x)\big)\nabla_{v_{2}}\nabla_{v_{3}}\hat{\Sigma}(x)\Big),
\end{align*}
by Linearity (see, e.g., \cite[Proposition 3.1]{A-B-D}), Product rule and Chain rule, we have
\begin{align*}
&\nabla_{v_{1}}\nabla_{v_{2}}\nabla_{v_{3}}\big(\hat{\Sigma}(x)^{\frac{1}{2}}\big)\\
=&\nabla_{v_{1}}\Big(\nabla^{2}\varphi\big(\hat{\Sigma}(x)\big)\nabla_{v_{2}}\hat{\Sigma}(x)\nabla_{v_{3}}\hat{\Sigma}(x)\Big)+\nabla_{v_{1}}\Big(\nabla\varphi\big(\hat{\Sigma}(x)\big)\nabla_{v_{2}}\nabla_{v_{3}}\hat{\Sigma}(x)\Big)\\
=&\nabla^{3}\varphi\big(\hat{\Sigma}(x)\big)\nabla_{v_{1}}\hat{\Sigma}(x)\nabla_{v_{2}}\hat{\Sigma}(x)\nabla_{v_{3}}\hat{\Sigma}(x)+\nabla^{2}\varphi\big(\hat{\Sigma}(x)\big)\nabla_{v_{1}}\nabla_{v_{2}}\hat{\Sigma}(x)\nabla_{v_{3}}\hat{\Sigma}(x)\\
&+\nabla^{2}\varphi\big(\hat{\Sigma}(x)\big)\nabla_{v_{2}}\hat{\Sigma}(x)\nabla_{v_{1}}\nabla_{v_{3}}\hat{\Sigma}(x)+\nabla^{2}\varphi\big(\hat{\Sigma}(x)\big)\nabla_{v_{1}}\hat{\Sigma}(x)\nabla_{v_{2}}\nabla_{v_{3}}\hat{\Sigma}(x)\\
&+\nabla\varphi\big(\hat{\Sigma}(x)\big)\nabla_{v_{1}}\nabla_{v_{2}}\nabla_{v_{3}}\hat{\Sigma}(x),
\end{align*}
which implies that
\begin{align*}
&\big\|\nabla_{v_{1}}\nabla_{v_{2}}\nabla_{v_{3}}\big(\hat{\Sigma}(x)^{\frac{1}{2}}\big)\big\|_{{\rm HS}}\\
\leq&\big\|\big|\nabla^{3}\varphi\big(\hat{\Sigma}(x)\big)\big|\big\|\|\nabla_{v_{1}}\hat{\Sigma}(x)\|_{{\rm HS}}\|\nabla_{v_{2}}\hat{\Sigma}(x)\|_{{\rm HS}}\|\nabla_{v_{3}}\hat{\Sigma}(x)\|_{{\rm HS}}\\
&+\big\|\big|\nabla^{2}\varphi\big(\hat{\Sigma}(x)\big)\big|\big\|\big(\|\nabla_{v_{2}}\hat{\Sigma}(x)\|_{{\rm HS}}\|\nabla_{v_{1}}\nabla_{v_{3}}\hat{\Sigma}(x)\big\|_{{\rm HS}}+\|\nabla_{v_{1}}\nabla_{v_{2}}\hat{\Sigma}(x)\|_{{\rm HS}}\|\nabla_{v_{3}}\hat{\Sigma}(x)\|_{{\rm HS}}\big)\\
&+\big\|\big|\nabla^{2}\varphi\big(\hat{\Sigma}(x)\big)\big|\big\|\|\nabla_{v_{1}}\hat{\Sigma}(x)\|_{{\rm HS}}\|\nabla_{v_{2}}\nabla_{v_{3}}\hat{\Sigma}(x)\|_{{\rm HS}}\\
&+\big\|\big|\nabla\varphi\big(\hat{\Sigma}(x)\big)\big|\big\|\|\nabla_{v_{1}}\nabla_{v_{2}}\nabla_{v_{3}}\hat{\Sigma}(x)\|_{{\rm HS}}.
\end{align*}
Hence, by (\ref{derivative}), we have
\begin{align*}
&\big\|\nabla_{v_{1}}\nabla_{v_{2}}\nabla_{v_{3}}\big(\hat{\Sigma}(x)^{\frac{1}{2}}\big)\big\|_{{\rm HS}}\\
\leq&\frac{1}{4}d^{\frac{1}{2}}\lambda_{\min}\big(\hat{\Sigma}(x)\big)^{-\frac{3}{2}}\big(\|\nabla_{v_{2}}\hat{\Sigma}(x)\|_{{\rm HS}}\|\nabla_{v_{1}}\nabla_{v_{3}}\hat{\Sigma}(x)\|_{{\rm HS}}+\|\nabla_{v_{1}}\nabla_{v_{2}}\hat{\Sigma}(x)\|_{{\rm HS}}\|\nabla_{v_{3}}\hat{\Sigma}(x)\|_{{\rm HS}}\big)\\
&+\frac{3}{8}d\lambda_{\min}\big(\hat{\Sigma}(x)\big)^{-\frac{5}{2}}\|\nabla_{v_{1}}\hat{\Sigma}(x)\|_{{\rm HS}}\|\nabla_{v_{2}}\hat{\Sigma}(x)\|_{{\rm HS}}\|\nabla_{v_{3}}\hat{\Sigma}(x)\|_{{\rm HS}}\\
&+\frac{1}{4}d^{\frac{1}{2}}\lambda_{\min}\big(\hat{\Sigma}(x)\big)^{-\frac{3}{2}}\|\nabla_{v_{1}}\hat{\Sigma}(x)\|_{{\rm HS}}\|\nabla_{v_{2}}\nabla_{v_{3}}\hat{\Sigma}(x)\|_{{\rm HS}}\\
&+\frac{1}{2}\lambda_{\min}\big(\hat{\Sigma}(x)\big)^{-\frac{1}{2}}\|\nabla_{v_{1}}\nabla_{v_{2}}\nabla_{v_{3}}\hat{\Sigma}(x)\|_{{\rm HS}}.
\end{align*}
\end{proof}

Now, we can verify Assumption \ref{assum4}.

\begin{lemma}\label{high}
In Example \ref{ex1}, suppose $\mathbb{E}|a_{i}|^{9}<\infty$, $i=1,\cdots,n$. When $n$ is large enough, for any $x,y \in \R^{d}$ and unit vectors $v_{i}\in\mathbb{R}^{d}$, $i=1,2,3$, the following inequalities
\begin{align*}
|\nabla_{v_2}\nabla_{v_1} \nabla P(x)| \ \le 3\mathbb{E}|a_{1}|^{3},\quad
|\nabla_{v_3}\nabla_{v_2}\nabla_{v_1} \nabla P(x)| \ \le 13\mathbb{E}|a_{1}|^{4},
\end{align*}
\begin{align*}
\|\nabla_{1,v_{1}} Q_{\eta,\delta}(x,y)\|_{\rm HS} \ \le 2\mathbb{E}|a_{1}|^{3}(\frac{\eta}{\delta})^{\frac{1}{2}},\ \ \ \ \ \|\nabla_{2,v_{1}} Q_{\eta,\delta}(x,y)\|_{\rm HS}\ \le 2\mathbb{E}|a_{1}|^{3}(\frac{\eta}{\delta})^{\frac{1}{2}},
\end{align*}
\begin{align*}
\|\nabla_{1,v_{1}} \nabla_{1,v_{2}}Q_{\eta,\delta}(x,y)\|_{\rm HS} \ \le 4(\frac{\eta}{\delta})^{\frac{1}{2}}\big(2\mathbb{E}|a_{1}|^{4}+\mathbb{E}|a_{1}|^{6}d^{\frac{1}{2}}\frac{\eta}{\delta}\big),
\end{align*}
\begin{align*}
\|\nabla_{1,v_{1}} \nabla_{1,v_{2}} \nabla_{1,v_{3}}Q_{\eta,\delta}(x,y)\|_{\rm HS} \ \le 4(\frac{\eta}{\delta})^{\frac{1}{2}}\big[12d^{\frac{1}{2}}\mathbb{E}|a_{1}|^{7}\frac{\eta}{\delta}+6d\mathbb{E}|a_{1}|^{9}(\frac{\eta}{\delta})^{2}+11\mathbb{E}|a_{1}|^{5}\big].
\end{align*}
hold with high probability.
\end{lemma}

\begin{proof}
By (\ref{ineq1}), with high probability, we have
\begin{align*}
\nabla^{2} P(x)=\mathbb{E}\big[\frac{a_{1}a_{1}^{T}e^{-a_{i}^{T}x}}{(1+e^{-a_{i}^{T}x})^{2}}\big]-\lambda I_{d},
\end{align*}
which implies
\begin{align*}
|\nabla_{v_2}\nabla_{v_1} \nabla P(x)|\leq \mathbb{E}\big[\frac{e^{-a_{1}^{T}x}}{(1+e^{-a_{1}^{T}x})^{2}}|a_{1}|^{3}\big]+2\mathbb{E}\big[\frac{e^{-2a_{1}^{T}x}}{(1+e^{-a_{1}^{T}x})^{3}}|a_{1}|^{3}\big]
\leq3\mathbb{E}|a_{1}|^{3},
\end{align*}
\begin{align*}
&|\nabla_{v_3}\nabla_{v_2}\nabla_{v_1} \nabla P(x)|\\
\leq&\mathbb{E}\big[\frac{e^{-a_{1}^{T}x}}{(1+e^{-a_{1}^{T}x})^{2}}|a_{1}|^{4}\big]+6\mathbb{E}\big[\frac{e^{-2a_{1}^{T}x}}{(1+e^{-a_{1}^{T}x})^{3}}|a_{1}|^{4}\big]
+6\mathbb{E}\big[\frac{e^{-3a_{1}^{T}x}}{(1+e^{-a_{1}^{T}x})^{4}}|a_{1}|^{4}\big]\leq13\mathbb{E}|a_{1}|^{4}.
\end{align*}
Since $Q_{\eta,\delta}(x,y)^{2}=\Sigma(x,y)+\frac{\delta}{\eta}I_{d}$ and $\Sigma(x,y)$ is a positive semi-definite matrix, we have $\inf_{x,y\in\mathbb{R}^{d}}\lambda_{\min}\big(Q_{\eta,\delta}(x,y)^{2}\big)\geq\frac{\delta}{\eta}$. Moreover, by (\ref{ineq4}), with high probability, we have
\begin{align*}
\|\nabla_{1,v_{1}}\Sigma(x,y)\|_{\rm HS}\leq&2\mathbb{E}\Big[\frac{e^{-a_{1}^{T}x}}{(1+e^{-a_{1}^{T}x})^{2}}\big|\frac{1}{1+e^{-a_{1}^{T}x}}-\frac{1}{1+e^{-a_{1}^{T}y}}\big||a_{1}|^{3}\Big]\\
&{+2\mathbb{E}\Big[\frac{|a_1|^2e^{-a_{1}^{T}x}}{(1+e^{-a_{1}^{T}x})^{2}}\Big]\Big|\mathbb{E}\Big[a_1\big(\frac{1}{1+e^{-a_{1}^{T}x}}-\frac{1}{1+e^{-a_{1}^{T}y}}\big)\Big]\Big|}\\
\leq&4\mathbb{E}|a_{1}|^{3},
\end{align*}
\begin{align*}
\|\nabla_{2,v_{1}} \Sigma(x,y)\|_{\rm HS}\leq4\mathbb{E}|a_{1}|^{3},
\end{align*}
\begin{align*}
&\|\nabla_{1,v_{1}} \nabla_{1,v_{2}}\Sigma(x,y)\|_{\rm HS}\\
\leq&4\mathbb{E}\Big[\Big(\big(\frac{e^{-a_{1}^{T}x}}{(1+e^{-a_{1}^{T}x})^{2}}+\frac{2e^{-2a_{1}^{T}x}}{(1+e^{-a_{1}^{T}x})^{3}}\big)
\big|\frac{1}{1+e^{-a_{1}^{T}x}}-\frac{1}{1+e^{-a_{1}^{T}y}}\big|+\frac{e^{-2a_{1}^{T}x}}{(1+e^{-a_{1}^{T}x})^{4}}\Big)|a_{1}|^{4}\Big]\\
\leq&16\mathbb{E}|a_{1}|^{4}
\end{align*}
and
\begin{align*}
\|\nabla_{1,v_{1}} \nabla_{1,v_{2}} \nabla_{1,v_{3}}\Sigma(x,y)\|_{\rm HS}\leq88\mathbb{E}|a_{1}|^{5}.
\end{align*}
Therefore, by Lemma \ref{taylor}, we have
\begin{align*}
\|\nabla_{1,v_{1}} Q_{\eta,\delta}(x,y)\|_{\rm HS}\leq&\frac{1}{2}\lambda_{\min}\big(Q_{\eta,\delta}(x,y)^{2}\big)^{-\frac{1}{2}}\|\nabla_{1,v}Q_{\eta,\delta}(x,y)^{2}\|_{{\rm HS}}\leq2\mathbb{E}|a_{1}|^{3}(\frac{\eta}{\delta})^{\frac{1}{2}},
\end{align*}
\begin{align*}
\|\nabla_{2,v_{1}} Q_{\eta,\delta}(x,y)\|_{\rm HS}\leq2\mathbb{E}|a_{1}|^{3}(\frac{\eta}{\delta})^{\frac{1}{2}},
\end{align*}
\begin{align*}
&\|\nabla_{1,v_{1}} \nabla_{1,v_{2}}Q_{\eta,\delta}(x,y)\|_{\rm HS}\\
\leq&\frac{1}{4}d^{\frac{1}{2}}\lambda_{\min}\big(Q_{\eta,\delta}(x,y)^{2}\big)^{-\frac{3}{2}}\|\nabla_{1,v_{1}}Q_{\eta,\delta}(x,y)^{2}\|_{{\rm HS}}\|\nabla_{1,v_{2}}Q_{\eta,\delta}(x,y)^{2}(x)\|_{{\rm HS}}\\
&+\frac{1}{2}\lambda_{\min}\big(Q_{\eta,\delta}(x,y)^{2}\big)^{-\frac{1}{2}}\|\nabla_{1,v_{1}}\nabla_{1,v_{2}}Q_{\eta,\delta}(x,y)^{2}\|_{{\rm HS}}\\
\leq&\frac{1}{4}d^{\frac{1}{2}}(\frac{\eta}{\delta})^{\frac{3}{2}}16\mathbb{E}|a_{1}|^{6}+\frac{1}{2}(\frac{\eta}{\delta})^{\frac{1}{2}}16\mathbb{E}|a_{1}|^{4}
=4(\frac{\eta}{\delta})^{\frac{1}{2}}\big(2\mathbb{E}|a_{1}|^{4}+\mathbb{E}|a_{1}|^{6}d^{\frac{1}{2}}\frac{\eta}{\delta}\big)
\end{align*}
and
\begin{align*}
&\|\nabla_{1,v_{1}} \nabla_{1,v_{2}} \nabla_{1,v_{3}}Q_{\eta,\delta}(x,y)\|_{\rm HS}\\
\leq&32d^{\frac{1}{2}}(\frac{\eta}{\delta})^{\frac{3}{2}}\mathbb{E}|a_{1}|^{7}+24d(\frac{\eta}{\delta})^{\frac{5}{2}}\mathbb{E}|a_{1}|^{9}+16d^{\frac{1}{2}}(\frac{\eta}{\delta})^{\frac{3}{2}}\mathbb{E}|a_{1}|^{7}
+44(\frac{\eta}{\delta})^{\frac{1}{2}}\mathbb{E}|a_{1}|^{5}\\
=&4(\frac{\eta}{\delta})^{\frac{1}{2}}\big[12d^{\frac{1}{2}}\mathbb{E}|a_{1}|^{7}\frac{\eta}{\delta}+6d\mathbb{E}|a_{1}|^{9}(\frac{\eta}{\delta})^{2}+11\mathbb{E}|a_{1}|^{5}\big].
\end{align*}
\end{proof}

\section{Proof of Lemmas in Section \ref{forth}}\label{moment estimate}

\subsection{Proof of Lemma \ref{exponential}}
Following \cite[Theorem 2.1]{ZSQ2019}, one can immediately finish the proof. To verify the condition of \cite[Theorem 2.1]{ZSQ2019}, we take
$$\sigma(x)=\sqrt{\eta}Q_{\eta,\delta}(x,\omega_0)=\left(\eta\Sigma(x,\omega_{0})+\delta I_{d}\right)^{\frac{1}{2}},
$$
$\sigma_0^2=\delta/2$, $\theta=0$, $q=1$ and $\tilde\sigma(x)=(\sigma^2(x)-\sigma_0^2I_d)^{\frac12}$ therein. Following \cite[Example 2.4]{ZSQ2019} and assumption \eqref{QQ1}, it is easy to see
\begin{eqnarray*}
\|\tilde\sigma(x)-\tilde\sigma(y)\|^2_{\mathrm{HS}}\le \frac{1}{2\sigma_0^2}|\sigma(x)-\sigma(y)|^2_{\mathrm{HS}}\le  \frac{\eta A_3}{2\sigma_0^2}|x-y|^2\le \frac{\gamma}{2}|x-y|^2.
\end{eqnarray*}
Combining equation above with assumptions \eqref{Plip} and \eqref{integration1}, one has
\begin{eqnarray*}
&&\Ll-\nabla P(x)+\nabla P(y),x-y \Rr+\frac12\|\tilde\sigma(x)-\tilde\sigma(y)\|^2_{\mathrm{HS}}+(q-3/2)\frac{|\tilde\sigma(x)-\tilde\sigma(y)(x-y)|^2}{|x-y|^2}\\
&\le& (L+\frac\gamma4+\frac\gamma2)|x-y|^21_{\{|x-y|^2\le \frac {4K}{\gamma}\}}-\frac\gamma2|x-y|^2.
\end{eqnarray*}
Taking $\bar K_1(x)=(L+\frac34\gamma)1_{\{|x|^2\le \frac {4K}{\gamma}\}}|x|^2$ and $\bar K_2=\frac{\gamma}{2}$, \cite[Condition (H)]{ZSQ2019} is satisfied. Let $\tilde K_2=\frac\gamma2$, a straight calculation implies that $C_1=1$, $c_0=C_2=\exp\{\frac{2K}{\delta\gamma}(L+\frac{3}{4}\gamma)\}$ and $\kappa=\frac{\gamma}{4}\exp\{-\frac{2K}{\delta\gamma}(L+\frac{3}{4}\gamma)\}$ therein, then we obtain
\begin{align*}
W_{1}\left(\mathcal{L}(X_{t}^{x}),\mathcal{L}(X_{t}^{y})\right)\leq e^{\frac{2K(L+\frac34\gamma)}{\gamma\delta}}e^{-\lambda t}|x-y|=\frac 4\gamma \lambda^{-1}e^{-\lambda t}|x-y|,
\end{align*}
where $\lambda=\frac\gamma4\exp\{-\frac{2K(L+\frac34\gamma)}{\gamma\delta}\}$.

\qed

\subsection{Proof of Lemma \ref{fourth}}
Recall (\ref{sfde-0-0}), by It\^{o}'s formula, (\ref{integration1}) and (\ref{QL}), we have
\begin{eqnarray*}
\frac{\dif}{\dif s}\E |X^{x}_{s}|^{2}&=&2 \E \Ll X_{s}, -\nabla P(X_{s})\Rr +\eta\E\|Q_{\eta,\delta}(X_{s},\omega_{0})\|^{2}_{\rm HS}\\
& \le &{ -2\gamma\E |X^{x}_{s}|^{2}+2K+2\E|X_s||\nabla P(0)|+2\eta\big(L^{2}|X_{s}-\omega_{0}|^{2}+\frac{\delta d}{\eta}\big)}\\
& \le &-(2\gamma-4L^{2}\eta-\frac{\gamma}2\big)\E |X^{x}_{s}|^{2}+2K+\frac2{\gamma}|\nabla P(0)|^2+4\eta { L^2}\mathbb{E}|\omega_{0}|^{2}+2\delta d\\
&\leq&-\gamma\E |X^{x}_{s}|^{2}+2K+\frac2{\gamma}|\nabla P(0)|^2+4\eta L^2\mathbb{E}|\omega_{0}|^{2}+2\delta d,
\end{eqnarray*}
where the last two lines following Young's inequality and the fact $\eta<\frac{\gamma}{8L^{2}}$.  Solving this differential inequality with initial data $X^{x}_{0}=x,$ gives
\begin{eqnarray}\label{ergodi}
\E |X^{x}_{t}|^{2}
&\leq&e^{-\gamma t}|x|^{2}+\frac{2(K+\gamma^{-1}|\nabla P(0)|^2+2\eta L^2\mathbb{E}|\omega_{0}|^{2}+\delta d)}{\gamma}\\
&\leq&C_{\gamma,d,|\nabla P(0)|,L}(1+|x|^{2}+\mathbb{E}|\omega_{0}|^{2}+K).\nonumber
\end{eqnarray}
Thus \eqref{moment} is proved.

By the Cauchy-Schwarz inequality, It\^{o}'s isometry, (\ref{PL}) and (\ref{QL}), we have
\begin{eqnarray*}
\mathbb{E}|X_{t}^{x}-x|^{2}&\leq&2\mathbb{E}|\int_{0}^{t}-\nabla P(X_{r}^x)\dif r|^{2}+2\mathbb{E}|\int_{0}^{t}\sqrt{\eta}Q_{\eta,\delta}(X_{r},\omega_0)\dif B_{r}|^{2}\nonumber\\
&\leq&2t\int_{0}^{t}\mathbb{E}|\nabla P(X_{r}^{x})|^{2}\dif r+2\eta\int_{0}^{t}\mathbb{E}\|Q_{\eta,\delta}(X_{r},\omega_0)\|^{2}_{{\rm {HS}}}\dif r\nonumber\\
&\leq&4t\int_{0}^{t}\big(|\nabla P(0)|^{2}+L^{2}\mathbb{E}|X_{r}|^{2}\big)\dif r+4\eta\int_{0}^{t}\big(L^{2}\mathbb{E}|X_{s}-\omega_{0}|^{2}+\frac{\delta d}{\eta}\big)\dif r\\
&\leq&4L^{2}(t+2\eta)\int_{0}^{t}\mathbb{E}|X_{r}|^{2}\dif r+4t\big(t|\nabla P(0)|^{2}+2L^{2}\eta\mathbb{E}|\omega_{0}|^{2}+\delta d\big),
\end{eqnarray*}
which, together with \eqref{moment}, implies \eqref{moment-1}.
\qed

\subsection{Proof of Lemma \ref{coupling2}}
By (\ref{representation}), it is easy to see
\begin{align*}
\mathbb{E}|\omega_{k}|^{4}
=&\mathbb{E}|\omega_{k-1}|^{4}+\mathbb{E}\big|\eta\big[\nabla\psi_{i_{k}}(\omega_{k-1})-\nabla\psi_{i_{k}}(\omega_{0})+\nabla P(\omega_{0})\big]-\sqrt{\eta\delta}W_{k}\big|^{4}\\
&-4\mathbb{E}\Big[|\omega_{k-1}|^{2}\big\langle\omega_{k-1},\eta\big[\nabla\psi_{i_{k}}(\omega_{k-1})-\nabla\psi_{i_{k}}(\omega_{0})+\nabla P(\omega_{0})\big]-\sqrt{\eta\delta}W_{k}\big\rangle\Big]\\
&+4\mathbb{E}\Big[\big\langle\omega_{k-1},\eta\big[\nabla\psi_{i_{k}}(\omega_{k-1})-\nabla\psi_{i_{k}}(\omega_{0})+\nabla P(\omega_{0})\big]-\sqrt{\eta\delta}W_{k}\big\rangle^{2}\Big]\\
&+2\mathbb{E}\Big[|\omega_{k-1}|^{2}\big|\eta\big[\nabla\psi_{i_{k}}(\omega_{k-1})-\nabla\psi_{i_{k}}(\omega_{0})+\nabla P(\omega_{0})\big]-\sqrt{\eta\delta}W_{k}\big|^{2}\Big]\\
&-4\mathbb{E}\Big[\big|\eta\big[\nabla\psi_{i_{k}}(\omega_{k-1})-\nabla\psi_{i_{k}}(\omega_{0})+\nabla P(\omega_{0})\big]-\sqrt{\eta\delta}W_{k}\big|^{2}\\
&\qquad\quad\big\langle\omega_{k-1},\eta\big[\nabla\psi_{i_{k}}(\omega_{k-1})-\nabla\psi_{i_{k}}(\omega_{0})+\nabla P(\omega_{0})\big]-\sqrt{\eta\delta}W_{k}\big\rangle\Big].
\end{align*}
Now we estimate each term on the right hand side. For the second term, the fact $\eta<\big(\frac{\gamma}{432L^{4}}\big)^{\frac{1}{3}}$, (\ref{Lip}) and (\ref{PL}) imply
\begin{align*}
&\mathbb{E}\big|\eta\big[\nabla\psi_{i_{k}}(\omega_{k-1})-\nabla\psi_{i_{k}}(\omega_{0})+\nabla P(\omega_{0})\big]-\sqrt{\eta\delta}W_{k}\big|^{4}\\
\leq&27\eta^{4}\big[\mathbb{E}|\nabla\psi_{i_{k}}(\omega_{k-1})-\nabla\psi_{i_{k}}(\omega_{0})|^{4}+\mathbb{E}|\nabla P(\omega_{0})|^4\big]+27(\eta\delta)^{2}\mathbb{E}|W_{1}|^{4}\\
\leq&27\eta^{4}\big[L^{4}\mathbb{E}|\omega_{k-1}-\omega_{0}|^{4}+\mathbb{E}|\nabla P(\omega_{0})|^4\big]+27(\eta\delta)^{2}\mathbb{E}|W|^{4}\\
\leq& \frac{\gamma}{2}\eta\mathbb{E}|\omega_{k-1}|^{4}+216\eta^{4}\big(2L^{4}\mathbb{E}|\omega_{0}|^{4}+|\nabla P(0)|^{4}\big)+81(\eta\delta d)^{2},
\end{align*}
where the last line following $\E|W|^4\le 3d^2$.

For the third term, since $i_{k}$ is independent of $\omega_{k-1}$ and uniformly distributed on $[n]$, \eqref{integration1} yields
\begin{eqnarray*}
&&\mathbb{E}\big[|\omega_{k-1}|^{2}\langle\omega_{k-1},\nabla\psi_{i_{k}}(\omega_{k-1})\rangle\big]\\
&=&\E\E_{I}\big[|\omega_{k-1}|^{2}\langle\nabla\psi_{I}(\omega_{k-1}),\omega_{k-1}\rangle\big] \\
&=& \E \big[|\omega_{k-1}|^{2}\Ll\nabla P(\omega_{k-1})-\nabla P(0),\omega_{k-1}\Rr \big]+\E \big[|\omega_{k-1}|^{2}\Ll\nabla P(0),\omega_{k-1}\Rr \big]\\
&\ge&\gamma\mathbb{E}|\omega_{k-1}|^{4}-K\mathbb{E}|\omega_{k-1}|^{2}+\E \big[|\omega_{k-1}|^{2}\Ll\nabla P(0),\omega_{k-1}\Rr \big],
\end{eqnarray*}
which implies
\begin{align*}
&-4\eta\mathbb{E}\big[|\omega_{k-1}|^{2}\langle\omega_{k-1},\nabla\psi_{i_{k}}(\omega_{k-1})\rangle\big]\\
\leq&-4\gamma\eta\mathbb{E}|\omega_{k-1}|^{4}+4K\eta\mathbb{E}|\omega_{k-1}|^{2}+4\eta\E \big[|\omega_{k-1}|^{3}|\nabla P(0)|\big]\\
\leq&-3\gamma\eta\mathbb{E}|\omega_{k-1}|^{4}+\frac{8K^{2}}{\gamma}\eta+\frac{216|\nabla P(0)|^{4}}{\gamma^{3}}\eta,
\end{align*}
where the last inequality is by Young's inequality. In addition, since $W_{k}$ is independent of $i_{k}$ and $\omega_{k-1},$ we have
\begin{align*}
\mathbb{E}\Big[|\omega_{k-1}|^{2}\big\langle\omega_{k-1},\eta\big[\nabla\psi_{i_{k}}(\omega_{0})-\nabla P(\omega_{0})\big]+\sqrt{\eta\delta}W_{k}\big\rangle\Big]=0.
\end{align*}

For the forth term, the fact $\eta<\frac{\gamma}{96L^{2}}$ implies
\begin{align*}
&4\mathbb{E}\left[\big\langle\omega_{k-1},\eta\big[\nabla\psi_{i_{k}}(\omega_{k-1})-\nabla\psi_{i_{k}}(\omega_{0})+\nabla P(\omega_{0})\big]-\sqrt{\eta\delta}W_{k}\big\rangle^{2}\right]\\
\leq&12\eta^{2}\mathbb{E}\Big[|\omega_{k-1}|^{2}\big(|\nabla\psi_{i_{k}}(\omega_{k-1})-\nabla\psi_{i_{k}}(\omega_{0})|^{2}+|\nabla P(\omega_{0})|^{2}\big)\Big]+12\eta\delta\mathbb{E}[|\omega_{k-1}|^{2}|W_{k}|^{2}]\\
\leq&12\eta^{2}\mathbb{E}\Big[|\omega_{k-1}|^{2}\big(2L^{2}|\omega_{k-1}|^{2}+4L^{2}|\omega_{0}|^{2}+2|\nabla P(0)|^{2}\big)\Big]+12\eta\delta d\mathbb{E}[|\omega_{k-1}|^{2}]\\
\leq&\frac{\gamma}{2}\eta\mathbb{E}|\omega_{k-1}|^{4}+\frac{144\eta}{\gamma}\mathbb{E}\big[4L^{2}\eta|\omega_{0}|^{2}+2 \eta|\nabla P(0)|^{2}+\delta d\big]^{2}\\
\leq&\frac{\gamma}{2}\eta\mathbb{E}|\omega_{k-1}|^{4}+\frac{432\eta}{\gamma}\big(16L^{4}\eta^{2}\mathbb{E}|\omega_{0}|^{4}+4 \eta^2|\nabla P(0)|^{4}+(\delta d)^{2}\big).
\end{align*}
The fifth term can be estimated by a similar calculation with the forth term, and we have
\begin{align*}
&2\mathbb{E}\Big[|\omega_{k-1}|^{2}\big|\eta\left(\nabla\psi_{i_{k}}(\omega_{k-1})-\nabla\psi_{i_{k}}(\omega_{0})+\nabla P(\omega_{0})\right)-\sqrt{\eta\delta}W_{k}\big|^{2}\Big]\\
\leq&\frac{\gamma}{4}\eta\mathbb{E}|\omega_{k-1}|^{4}+\frac{216\eta}{\gamma}\big[16L^{4}\eta^{2}\mathbb{E}|\omega_{0}|^{4}+4 \eta^2|\nabla P(0)|^{4}+(\delta d)^{2}\big].
\end{align*}
For the last term, noticing that $\eta<(\frac{\gamma}{576L^{3}})^{\frac{1}{2}}$, by (\ref{Lip}), the H\"{o}lder inequality, (\ref{PL}) and Young's inequality, we can get
\begin{align*}
&4\mathbb{E}\Big[\big|\eta\big[\nabla\psi_{i_{k}}(\omega_{k-1})-\nabla\psi_{i_{k}}(\omega_{0})+\nabla P(\omega_{0})\big]-\sqrt{\eta\delta}W_{k}\big|^{2}\\
&\qquad\big\langle\omega_{k-1},\eta\big[\nabla\psi_{i_{k}}(\omega_{k-1})-\nabla\psi_{i_{k}}(\omega_{0})+\nabla P(\omega_{0})\big]-\sqrt{\eta\delta}W_{k}\big\rangle\Big]\\
\leq&36\mathbb{E}\Big[\big[\eta^{3}|\nabla\psi_{i_{k}}(\omega_{k-1})-\nabla\psi_{i_{k}}(\omega_{0})|^{3}+\eta^{3}|\nabla P(\omega_{0})|^{3}+|\sqrt{\eta\delta}W_{k}|^{3}\big]|\omega_{k-1}|\Big]\\
\leq&144\eta^{3}\mathbb{E}\Big[|\omega_{k-1}|\big[L^{3}|\omega_{k-1}|^{3}+2L^{3}|\omega_{0}|^{3}+|\nabla P(0)|^{3})\big]\Big]+36(\eta\delta)^{\frac{3}{2}}\mathbb{E}\big[|\omega_{k-1}||W|^{3}]\\
\leq&\frac{3\gamma}{4}\eta\mathbb{E}|\omega_{k-1}|^{4}+\frac{ 200\eta^{\frac{5}{3}}}{\gamma^{\frac{1}{3}}}\big[16\eta^{2}(L^{4}\mathbb{E}|\omega_{0}|^{4}+|\nabla P(0)|^{4})+ (\delta d)^{2}\big].
\end{align*}
Since $\eta<1$, the inequalities above  imply
\begin{align*}
\mathbb{E}|\omega_{k}|^{4}\leq&\big(1-\gamma\eta\big)\mathbb{E}|\omega_{k-1}|^{4}+C_{\gamma,L}\big(K^{2}+|\nabla P(0)|^{4}+\mathbb{E}|\omega_{0}|^{4}+\delta^{2}d^2\big)\eta.
\end{align*}
Therefore,
\begin{align*}
\mathbb{E}|\omega_{k}^{x}|^{4}
\le & (1-\gamma\eta)^{k}|x|^{4}+C_{\gamma,L}\big(K^{2}+|\nabla P(0)|^{4}+\mathbb{E}|\omega_{0}|^{4}+\delta^{2}d^2\big)\eta\sum_{j=0}^{k-1}
(1-\gamma\eta)^{j} \\
\le & C_{\gamma,d,|\nabla P(0)|,L}\big(1+|x|^{4}+\mathbb{E}|\omega_{0}|^{4}+K^{2}\big).
\end{align*}
\qed

\subsection{Proof of Lemma \ref{SGDtaylor}}
Recall (\ref{cong}), for any $f\in\mathcal{C}_{b}^{2}(\mathbb{R}^{d}),$ it is easy to calculate that
\begin{align*}
\mathcal{A}^{Z}f(x)=&\lim_{t\rightarrow0}\frac{\mathbb{E}f(X_{\eta t}^{x})-f(x)}{\eta t}{\eta}\\
=&\frac{1}2\eta \Ll \eta\Sigma(x,\omega_0)+\delta I_{d}, \nabla^{2} f(x)\Rr_{{\rm HS}}-\eta\Ll \nabla P(x),\nabla f(x)\Rr.
\end{align*}
Then, for any $u_t(x)=\E h(X^x_{t})$ with $k\geq1,$ we have
\begin{align*}
&\mathbb{E}\int_{0}^{1}\mathcal{A}^{Z}u_{t}(X_{\eta s}^{x})\dif s\\
=&-\eta\mathbb{E}\int_{0}^{1}\langle\nabla u_{t}(X_{\eta s}^{x}),\nabla P(X_{\eta s}^{x})\rangle\dif s+\frac{1}{2}\eta\mathbb{E}\int_{0}^{1}\langle\nabla^{2}u_{t}(X_{\eta s}^{x}),\eta\Sigma(X_{\eta s}^{x},\omega_{0})+\delta I_{d}\rangle_{{\rm HS}} \dif s\\
=&-\mathbb{E}\int_{0}^{\eta}\langle\nabla u_{t}(X_{s}^{x}),\nabla P(X_{s}^{x})\rangle\dif s+\frac{1}{2}\mathbb{E}\int_{0}^{\eta}\langle\nabla^{2}u_{t}(X_{s}^{x}),\eta\Sigma(X_{s}^{x},\omega_{0})+\delta I_{d}\rangle_{{\rm HS}} \dif s.
\end{align*}
Recall (\ref{e:GenW-1}), we further have
\begin{align*}
\mcl A^{\omega} u_{t}(x)=&\E\left[u_{t}\big(x-\eta \nabla \psi_I(x)+\eta \nabla \psi_I(\omega_0)-\eta\nabla P(\omega_0)+\sqrt{\eta\delta}W\big)\right]-u_{t}(x)\\
=&\E\left[u_{t}\big(x+\sqrt{\eta}V_{\eta,\delta}(x,\omega_{0},I,W)-\eta\nabla P(x)\big)\right]-u_{t}(x)
\end{align*}
with $V_{\eta,\delta}(x,\omega_{0},I,W)=-\sqrt{\eta}\big[\nabla\psi_{I}(x)-\nabla\psi_{I}(\omega_{0})-\nabla P(x)+\nabla P(\omega_{0})\big]+\sqrt{\delta}W$. Then, by Taylor's expansion, we have
\begin{align*}
\mathcal{A}^{\omega}u_{t}(x)=&\mathbb{E}\Big[\big\langle\nabla u_{t}(x),\sqrt{\eta}V_{\eta,\delta}(x,\omega_{0},I,W)-\eta\nabla P(x)\big\rangle\Big]+\E[\mathcal{R}^{u_{t}}(x)]\\
&+\frac{1}{2}\mathbb{E}\big\langle\nabla^{2}u_{t}(x),\big[\sqrt{\eta}V_{\eta,\delta}(x,\omega_{0},I,W)-\eta\nabla P(x)\big]\\
&\qquad\qquad\qquad\quad\big[\sqrt{\eta}V_{\eta,\delta}(x,\omega_{0},I,W)-\eta\nabla P(x)\big]^{T}\big\rangle_{{\rm HS}}\\
=&\langle\nabla u_{t}(x),-\eta\nabla P(x)\rangle+\frac{1}{2}\eta^{2}\langle\nabla^{2}u_{t}(x),\mathbb{E}[Q_{\eta,\delta}(x,\omega_{0})]^{2}+\nabla P(x)\big(\nabla P(x)\big)^{T}\rangle_{{\rm HS}}\\
&+\E[\mathcal{R}^{u_{t}}(x)],
\end{align*}
where
\begin{align*}
&\mathcal{R}^{u_{t}}(x)\\
=&\int_{0}^{1}\int_{0}^{r}\langle\nabla^{2}u_{t}\big(x+s\big[\sqrt{\eta}V_{\eta,\delta}(x,\omega_{0},I,W)-\eta\nabla P(x)\big]\big)
-\nabla^{2}u_{t}(x),\\
&\qquad\qquad\quad\big[\sqrt{\eta}V_{\eta,\delta}(x,\omega_{0},I,W)-\eta\nabla P(x)\big]\big[\sqrt{\eta}V_{\eta,\delta}(x,\omega_{0},I,W)-\eta\nabla P(x)\big]^{T}\rangle \dif s\dif r.
\end{align*}
Therefore, we have
\begin{align*}
\big|\mathbb{E}\int_{0}^{1}\big[\mathcal{A}^{Z}u_{t}(Z_{s}^{x})-\mathcal{A}^{\omega} u_{t}(x)\big]\dif s\big|\leq\mathcal{J}_{1}+\mathcal{J}_{2}+\E|\mathcal{R}^{u_{t}}(x)|,
\end{align*}
where
\begin{align*}
\mathcal{J}_{1}:=\Big|\mathbb{E}\int_{0}^{\eta}\langle\nabla u_{t}(X_{s}^{x}),\nabla P(X_{s}^{x})\rangle\dif s&-\eta\langle\nabla u_{t}(x),\nabla P(x)\rangle\\
&+\frac{1}{2}\eta^{2}\langle\nabla^{2} u_{t}(x),\nabla P(x)\big(\nabla P(x)\big)^{T}\rangle_{{\rm HS}}\Big|
\end{align*}
and
\begin{align*}
\mathcal{J}_{2}:=\Big|\frac{1}{2}\mathbb{E}\int_{0}^{\eta}\langle\nabla^{2}u_{t}(X_{s}^{x}),\eta\Sigma(X_{s}^{x},\omega_{0})+\delta I_{d}\rangle_{{\rm HS}} \dif s-\frac{1}{2}\eta^{2}\langle\nabla^{2}u_{t}(x),\mathbb{E}[Q_{\eta,\delta}(x,\omega_{0})]^{2}\rangle_{{\rm HS}}\Big|.
\end{align*}

For $\mathcal{J}_{1},$ we have
\begin{align*}
\mathcal{J}_{1}\leq&\big|\mathbb{E}\int_{0}^{\eta}\langle\nabla u_{t}(X_{s}^{x}),\nabla P(X_{s}^{x})-\nabla P(x)\rangle\dif s\big|\\
&+\big|\mathbb{E}\int_{0}^{\eta}\langle\nabla u_{t}(X_{s}^{x})-\nabla u_{t}(x),\nabla P(x)\rangle\dif s+\frac{1}{2}\eta^{2}\langle\nabla^{2} u_{t}(x),\nabla P(x)\big(\nabla P(x)\big)^{T}\rangle_{{\rm HS}}\big|\\
:=&\mathcal{J}_{11}+\mathcal{J}_{12}.
\end{align*}
By (\ref{gradient2}), (\ref{Plip}), the Cauchy-Schwarz inequality and (\ref{moment-1}), one has
\begin{align*}
\mathcal{J}_{11}
\leq&C_{L}\int_{0}^{\eta}\mathbb{E}\big|X_{s}^{x}-x\big|\dif s \\
\leq&C_{\gamma,d,|\nabla P(0)|,L}\int_{0}^{\eta}(1+|x|+\sqrt{\mathbb{E}|\omega_{0}|^{2}}+K^{\frac{1}{2}})\sqrt{s(s+\eta+\delta)}\dif s \\
\leq&C_{\gamma,d,|\nabla P(0)|,L}(1+|x|+\sqrt{\mathbb{E}|\omega_{0}|^{2}}+K^{\frac{1}{2}})\eta^{\frac{3}{2}}(\eta^{\frac{1}{2}}+\delta^{\frac{1}{2}}).
\end{align*}
Notice that
\begin{align*}
&\mathbb{E}\langle\nabla u_{t}(X_{s}^{x})-\nabla u_{t}(x),\nabla P(x)\rangle\\
=&\mathbb{E}\langle\nabla^{2} u_{t}(x),(X_{s}^{x}-x)\big(\nabla P(x)\big)^{T}\rangle_{{\rm HS}}\\
&+\int_{0}^{1}\mathbb{E}\langle\nabla^{2} u_{t}\big(x+r(X_{s}^{x}-x)\big)-\nabla^{2} u_{t}(x),(X_{s}^{x}-x)\big(\nabla P(x)\big)^{T}\rangle_{{\rm HS}}\dif r\\
=&-\int_{0}^{s}\mathbb{E}\langle\nabla^{2} u_{t}(x),\nabla P(X_{v}^{x})\big(\nabla P(x)\big)^{T}\rangle_{{\rm HS}}\dif v\\
&+\int_{0}^{1}\mathbb{E}\langle\nabla^{2} u_{t}\big(x+r(X_{s}^{x}-x)\big)-\nabla^{2} u_{t}(x),(X_{s}^{x}-x)\big(\nabla P(x)\big)^{T}\rangle_{{\rm HS}}\dif r.
\end{align*}
By (\ref{sec}), (\ref{thirdmain}), (\ref{Plip}) and (\ref{PL}), we have
\begin{align*}
\mathcal{J}_{12}\leq&\big|\mathbb{E}\int_{0}^{\eta}\int_{0}^{s}\mathbb{E}\langle\nabla^{2} u_{t}(x),\big(\nabla P(X_{v}^{x})-\nabla P(x)\big)\big(\nabla P(x)\big)^{T}\rangle_{{\rm HS}}\dif v\dif s\big|\\
&+\big|\int_{0}^{\eta}\int_{0}^{1}\mathbb{E}\langle\nabla^{2} u_{t}\big(x+r(X_{s}^{x}-x)\big)-\nabla^{2} u_{t}(x),(X_{s}^{x}-x)\big(\nabla P(x)\big)^{T}\rangle_{{\rm HS}}\dif r\dif s\big|\\
\leq&C_{A,\gamma,d,|\nabla P(0)|,L}{ \big(1+\frac{1}{\sqrt{\delta t}}\big)}(1+|x|)\int_{0}^{\eta}\int_{0}^{s}\mathbb{E}\big|X_{v}^{x}-x\big|\dif v\dif s\\
&+C_{A,\gamma,d,|\nabla P(0)|,L}\big(1+\frac{1}{\delta t}+\frac{1}{t^{\frac{5}{4}}}\big)(1+|x|)\int_{0}^{\eta}\int_{0}^{1}r\mathbb{E}|X_{s}^{x}-x|^{2}\dif r\dif s.
\end{align*}
Then, by the Cauchy-Schwarz inequality and (\ref{moment-1}), we can get
\begin{align*}
\mathcal{J}_{12}\leq&C_{A,\gamma,d,|\nabla P(0)|,L}\big(1+\frac{1}{\sqrt{\delta t}}\big)(1+|x|^{2}+\mathbb{E}|\omega_{0}|^{2}+K)\eta^{\frac{5}{2}}(\eta^{\frac{1}{2}}+\delta^{\frac{1}{2}})\\
&+C_{A,\gamma,d,|\nabla P(0)|,L}\big(1+\frac{1}{\delta t}+\frac{1}{t^{\frac{5}{4}}}\big){(1+|x|)(1+|x|^{2}+\mathbb{E}|\omega_{0}|^{2}+K)\eta^{2}(\eta+\delta)}.
\end{align*}
The condition $\eta\leq\delta\leq1$ further implies
\begin{align*}
\mathcal{J}_{12}\leq&C_{A,\gamma,d,|\nabla P(0)|,L}\big(1+\frac{1}{t}+\frac{{\delta}}{t^{\frac{5}{4}}}\big){(1+|x|)}(1 +|x|^{2}+\mathbb{E}|\omega_{0}|^{2}+K)\eta^{2}.
\end{align*}
Hence,
\begin{align*}
\mathcal{J}_{1}\leq C_{A,\gamma,d,|\nabla P(0)|,L}(1+|x|)(1+|x|^{2}+\mathbb{E}|\omega_{0}|^{2}+K)\big[\big(\frac{1}{t}+\frac{{\delta}}{t^{\frac{5}{4}}}\big)\eta^{\frac{1}{2}}+\delta^{\frac{1}{2}}\big]\eta^{\frac{3}{2}}.
\end{align*}

For $\mathcal{J}_{2}$, notice that $\eta\mathbb{E}[Q_{\eta,\delta}(x,\omega_{0})]^{2}=\eta\mathbb{E}[\Sigma(x,\omega_{0})]+\delta I_{d}$, and for $x,y,z\in\mathbb{R}^{d}$, following the definition of $\Sigma\left(x,y\right)$, a straight calculation gives that
\begin{align*}
\Sigma(x,y)-\Sigma(z,y)=&\E\left[\left(\nabla\psi_{I}(x)-\nabla\psi_{I}(z)\right) \left(\nabla\psi_{I}(x)-\nabla\psi_{I}(y)\right)^{\rm T} \right]\\
& -\E\left[\nabla\psi_{I}(x)-\nabla\psi_{I}(z)\right] \left(\E_{I} \left[\nabla\psi_{I}(x)-\nabla\psi_{I}(y)\right]\right)^{\rm T}\\
&+\E\left[\left(\nabla\psi_{I}(z)-\nabla\psi_{I}(y)\right) \left(\nabla\psi_{I}(x)-\nabla\psi_{I}(z)\right)^{\rm T} \right]\\
& -\E\left[\nabla\psi_{I}(z)-\nabla\psi_{I}(y)\right] \left(\E_{I} \left[\nabla\psi_{I}(x)-\nabla\psi_{I}(z)\right]\right)^{\rm T}.
\end{align*}
By (\ref{Lip}), (\ref{Plip}) and the Cauchy-Schwarz inequality, we further have
\begin{align}\label{sigmal}
\|\Sigma(x,y)\|_{{\rm HS}}\leq2L^{2}|x-y|^{2}\leq4L^{2}(|x|^{2}+|y|^{2}),
\end{align}
and
\begin{align}\label{sigmaLip}
\|\Sigma(x,y)-\Sigma(z,y)\|_{{\rm HS}}\leq2L^{2}(|x-y|+|z-y|)|x-z|\leq2L^{2}(|x|+2|y|+|z|)|x-z|.
\end{align}
Then, the Cauchy-Schwarz inequality, (\ref{sec}) and (\ref{thirdmain}) imply
\begin{align*}
\mathcal{J}_{2}\leq&\frac{\eta}{2}\mathbb{E}\big|\int_{0}^{\eta}\langle\nabla^{2}u_{t}(X_{s}^{x}),\Sigma(X_{s}^{x},\omega_{0})-\Sigma(x,\omega_{0})\rangle_{{\rm HS}} \dif s\big|\\
&+\frac{1}{2}\mathbb{E}\big|\int_{0}^{\eta}\langle\nabla^{2}u_{t}(X_{s}^{x})-\nabla^{2}u_{t}(x),\eta\Sigma(x,\omega_{0})+\delta I_{d}\rangle_{{\rm HS}} \dif s\big|\\
\leq&\eta C_{A,\gamma,d,L}\big(1+\frac{1}{\sqrt{\delta t}}\big)\int_{0}^{\eta}\mathbb{E}\big[(|X_{s}^{x}|+|\omega_{0}|+|x|)|X_{s}^{x}-x|\big]\dif s\\
&+C_{A,\gamma,d,L}\big(1+\frac{1}{\delta t}+\frac{1}{t^{\frac{5}{4}}}\big)\int_{0}^{\eta}\mathbb{E}\big[|X_{s}^{x}-x|(\eta|x|^{2}+\eta|\omega_{0}|^{2}+\delta)\big]\dif s,
\end{align*}
Following the Cauchy-Schwarz inequality, (\ref{moment}) and (\ref{moment-1}), one has
\begin{align*}
\mathcal{J}_{2}\leq&C_{A,\gamma,d,|\nabla P(0)|,L}{ \big(1+\frac{1}{\sqrt{\delta t}}\big)}(1+|x|^{2}+\mathbb{E}|\omega_{0}|^{2}+K)\eta^{\frac{5}{2}}(\eta^{\frac{1}{2}}+\delta^{\frac{1}{2}})\\
&+C_{A,\gamma,d,L}\big(1+\frac{1}{\delta t}+\frac{1}{t^{\frac{5}{4}}}\big)(1+\sqrt{\mathbb{E}|\omega_{0}|^{2}}+K^{\frac{1}{2}})\\
&\qquad\qquad\quad(1+\sqrt{\mathbb{E}|\omega_{0}|^{4}})(1+|x|^{3})\eta^{\frac{3}{2}}(\eta^{\frac{1}{2}}+\delta^{\frac{1}{2}})(\eta+\delta).
\end{align*}
The condition $\eta\leq\delta\leq1$ further implies
\begin{align*}
\mathcal{J}_{2}\leq C_{A,\gamma,d,|\nabla P(0)|,L}&\big(1+\frac{1}{t}+\frac{\delta}{t^{\frac{5}{4}}}\big)(1+\mathbb{E}|\omega_{0}|^{4}+K)(1+|x|^{3})\eta^{\frac{3}{2}}\delta^{\frac{1}{2}}.
\end{align*}

In addition, by (\ref{thirdmain}), (\ref{Lip}), H\"{o}lder's inequality and (\ref{Plip}), we have
\begin{align*}
&\mathbb{E}|\mathcal{R}^{u_{t}}(x)|\\
\leq&C_{A,\gamma,d}\big(1+\frac{1}{\delta t}+\frac{1}{t^{\frac{5}{4}}}\big)\mathbb{E}\big|\sqrt{\eta}V_{\eta,\delta}(x,\omega_{0},I,W)-\eta\nabla P(x)\big|^{3}\\
\leq&C_{A,\gamma,d,|\nabla P(0)|,L}\big(1+\frac{1}{\delta t}+\frac{1}{t^{\frac{5}{4}}}\big)\big[\eta^{3}(1+|x|^{3}+\mathbb{E}|\omega_{0}|^{3})+(\eta\delta)^{\frac{3}{2}}\big]\\
\leq&C_{A,\gamma,d,|\nabla P(0)|,L}\big(1+\frac{1}{t}+\frac{\delta}{t^{\frac{5}{4}}}\big)(1+|x|^{3}+\mathbb{E}|\omega_{0}|^{3})\eta^{\frac{3}{2}}\delta^{\frac{1}{2}}.
\end{align*}

Combining all of above, we have
\begin{align*}
&\big|\mathbb{E}\int_{0}^{1}\big[\mathcal{A}^{Z}u_{t}(Z_{s}^{x})-\mathcal{A}^{\omega} u_{t}(x)\big]\dif s\big|\\
\leq&C_{A,\gamma,d,|\nabla P(0)|,L}\big(1+\frac{1}{t}+\frac{\delta}{t^{\frac{5}{4}}}\big)(1+\mathbb{E}|\omega_{0}|^{4}+K)(1+|x|^{3})\eta^{\frac{3}{2}}\delta^{\frac{1}{2}}.
\end{align*}
\qed

\section{Proof of Lemma \ref{mainlem1}}\label{Aind}

Under the Assumptions \ref{assum1}, \ref{assum3} and \ref{assum4} in Sections \ref{first} and \ref{second}, we recall some preliminary of Malliavin calculus and derive standard estimates related to Malliavin calculus and SDE, which will be applied to prove the Lemma \ref{mainlem1} in Section \ref{forth}.
\subsection{Malliavin calculus of SDE (\ref{sfde-0-0}) (\cite{F-S-X})}

For simplicity, denote $B(x):=-\nabla P(x)$ and $\sigma_{\eta,\delta,\omega_{0}}(x)=Q_{\eta,\delta}(x,\omega_{0}).$ If there is no ambiguity, we abbreviate it as $\sigma(x)=\sigma_{\eta,\delta,\omega_{0}}(x).$ Then, the SDE (\ref{sfde-0-0}) can be written as the following form:
\begin{align}\label{SDE}
\dif X_{t}=B(X_{t})\dif t+\sqrt{\eta}\sigma(X_{t})\dif B_{t},\quad { X_{0}}=x,
\end{align}
where $B_{t}$ is a standard $d-$dimensional Brownian motion. Moreover, the Assumptions \ref{assum1} and \ref{assum4} can be rewritten as the following form:

{\bf Assumption A} There exist constants $L\geq0,$ $A_{i}\geq0$ with $i=1,2,...,5,$ such that for any $x,y\in\mathbb{R}^{d}$ and unit vectors $v,v_{1},v_{2},v_{3} \in \R^{d},$ we have
\begin{align}\label{P1}
|\nabla_{v}B(x)|\leq L, \quad |\nabla_{v_2}\nabla_{v_1}B(x)| \ \le \ A_{1},
\end{align}
\begin{align}\label{Q1}
|\nabla_{v_3}\nabla_{v_2}\nabla_{v_1} B(x)| \ \le \ A_{2}, \quad \|\nabla_{v_{1}} \sigma(x)\|_{\rm HS}^{2} \ \le \ A_3
\end{align}
\begin{align}\label{Q2}
\|\nabla_{v_{1}} \nabla_{v_{2}}\sigma(x)\|_{\rm HS}^{2} \ \le \ A_4, \quad \|\nabla_{v_{1}} \nabla_{v_{2}} \nabla_{v_{3}}\sigma(x)\|_{\rm HS}^{2} \ \le \ A_5.
\end{align}

\begin{remark}
Since $S(x)=\sigma(x)\sigma(x)^{T}=\Sigma(x)+\frac{\delta}{\eta} I_{d}$ and $\Sigma(x)$ is semi-positive definite, for any $0\neq\xi\in\mathbb{R}^{d},$ we have
\begin{align}\label{A1}
\xi^{T}S(x)\xi\geq\frac{\delta}{\eta}\xi^{T}I_{d}\xi=\frac{\delta}{\eta}|\xi|^{2}.
\end{align}
\end{remark}

Under the {\bf Assumption A}, there exists a unique solution to the SDE (\ref{SDE}) and the SDE (\ref{SDE}) has a unique non-degenerate invariant measure (see, e.g., \cite{B-R,C1,E,K,P}).

Next, we briefly recall Bismut's approach to Malliavin calculus, which is crucial to prove Lemma \ref{mainlem1}.

We first consider the derivative of $X_{t}^{x}$ with respect to initial value $x$, which is called the Jacobian flow. Let $v\in\mathbb{R}^{d},$ the Jacobian flow $\nabla_{v}X_{t}^{x}$ along the direction $u$ is defined by
\begin{align*}
\nabla_{v}X_{t}^{x}=\lim_{\epsilon\rightarrow0}\frac{X_{t}^{x+\epsilon v}-X_{t}^{x}}{\epsilon},\quad t\geq0.
\end{align*}
The above limit exists and satisfies
\begin{align}\label{gradient SDE}
d\nabla_{v}X_{t}^{x}=\nabla B(X_{t}^{x})\nabla_{v}X_{t}^{x}\dif t+\sqrt{\eta}\nabla\sigma(X_{t}^{x})\nabla_{v}X_{t}^{x}\dif B_{t}, \quad \nabla_{v}X_{0}^{x}=v.
\end{align}
Then, we use the notations $J_{s,t}^{x}$ with $0\leq s\leq t<\infty$ for the stochastic flow between time $s$ and $t,$ that is,
\begin{align*}
\nabla_{v}X_{t}^{x}=J_{0,t}^{x}v.
\end{align*}
Note that we have the important cocycle property $J_{0,s}^{x}J_{s,t}^{x}=J_{0,t}^{x}$ for all $0\leq s\leq t<\infty$. For a more thorough discussion on stochastic flow, we refer the reader to \cite{K84,K97,B99,H-M} and the references therein.

For $v_{1},v_{2}\in\mathbb{R}^{d},$ we can define $\nabla_{v_{2}}\nabla_{v_{1}}X_{t}^{x},$ which satisfies
\begin{align}\label{secequation}
d\nabla_{v_{2}}\nabla_{v_{1}}X_{t}^{x}=&\nabla B(X_{t}^{x})\nabla_{v_{2}}\nabla_{v_{1}}X_{t}^{x}\dif t+\nabla^{2}B(X_{t}^{x})\nabla_{v_{2}}X_{t}^{x}\nabla_{v_{1}}X_{t}^{x}\dif t\nonumber\\
&+\sqrt{\eta}\nabla\sigma(X_{t}^{x})\nabla_{v_{2}}\nabla_{v_{1}}X_{t}^{x}\dif B_{t}+\sqrt{\eta}\nabla^{2}\sigma(X_{t}^{x})\nabla_{v_{2}}X_{t}^{x}\nabla_{v_{1}}X_{t}^{x}\dif B_{t},
\end{align}
with $\nabla_{v_{2}}\nabla_{v_{1}}X_{0}^{x}=0.$

For $v_{1},v_{2},v_{3}\in\mathbb{R}^{d},$ we can define $\nabla_{v_{3}}\nabla_{v_{2}}\nabla_{v_{1}}X_{t}^{x},$ which satisfies
\begin{align*}
d\nabla_{v_{3}}\nabla_{v_{2}}\nabla_{v_{1}}X_{t}^{x}=&\nabla^{2} B(X_{t}^{x})\nabla_{v_{3}}X_{t}^{x}\nabla_{v_{2}}\nabla_{v_{1}}X_{t}^{x}\dif t+\nabla B(X_{t}^{x})\nabla_{v_{3}}\nabla_{v_{2}}\nabla_{v_{1}}X_{t}^{x}\dif t\\
&+\nabla^{3}B(X_{t}^{x})\nabla_{v_{3}}X_{t}^{x}\nabla_{v_{2}}X_{t}^{x}\nabla_{v_{1}}X_{t}^{x}\dif t
+\nabla^{2}B(X_{t}^{x})\nabla_{v_{3}}\nabla_{v_{2}}X_{t}^{x}\nabla_{v_{1}}X_{t}^{x}\dif t\\
&+\nabla^{2}B(X_{t}^{x})\nabla_{v_{2}}X_{t}^{x}\nabla_{v_{3}}\nabla_{v_{1}}X_{t}^{x}\dif t
+\eta^{\frac{1}{2}}\nabla^{2}\sigma(X_{t}^{x})\nabla_{v_{3}}X_{t}^{x}\nabla_{v_{2}}\nabla_{v_{1}}X_{t}^{x}\dif B_{t}\\
&+\eta^{\frac{1}{2}}\big[\nabla\sigma(X_{t}^{x})\nabla_{v_{3}}\nabla_{v_{2}}\nabla_{v_{1}}X_{t}^{x}
+\nabla^{3}\sigma(X_{t}^{x})\nabla_{v_{3}}X_{t}^{x}\nabla_{v_{2}}X_{t}^{x}\nabla_{v_{1}}X_{t}^{x}\big]\dif B_{t}\\
&+\eta^{\frac{1}{2}}\big[\nabla^{2}\sigma(X_{t}^{x})\nabla_{v_{3}}\nabla_{v_{2}}X_{t}^{x}\nabla_{v_{1}}X_{t}^{x}
+\nabla^{2}\sigma(X_{t}^{x})\nabla_{v_{2}}X_{t}^{x}\nabla_{v_{3}}\nabla_{v_{1}}X_{t}^{x}\big]\dif B_{t},
\end{align*}
with $\nabla_{v_{3}}\nabla_{v_{2}}\nabla_{v_{1}}X_{0}^{x}=0.$

Then, we have the following estimate:

\begin{lemma}
For all  $x\in\mathbb{R}^{d}$ and $v,v_{1},v_{2},v_{3}\in\mathbb{R}^{d}$, as $\eta\in(0,1)$ and $t\in(0,1]$, we have
\begin{align}\label{grad3}
\mathbb{E}|\nabla_{v}X_{t}^{x}|^{8}\leq e^{8L+28A_3}|v|^{8},
\end{align}
\begin{align}\label{gradsec}
\mathbb{E}|\nabla_{v_{2}}\nabla_{v_{1}}X_{t}^{x}|^{4}
\leq 2(A_{1}+6A_{4})e^{12L+3A_{1}+12A_{3}+6A_{4}+28}|v_{1}|^{4}|v_{2}|^{4},
\end{align}
and
\begin{align}\label{gradthird}
\mathbb{E}|\nabla_{v_{3}}\nabla_{v_{2}}\nabla_{v_{1}}X_{t}^{x}|^{2}
\leq C_{A,L}|v_{1}|^{2}|v_{2}|^{2}|v_{3}|^{2}.
\end{align}
\end{lemma}
\begin{proof}
Recalling (\ref{gradient SDE}), by It\^{o}'s formula, (\ref{P1}), (\ref{Q1}) and (\ref{e:RelEigHS}), we have
\begin{align*}
\frac{\dif}{\dif s}\mathbb{E}|\nabla_{v}X_{s}^{x}|^{8}=&8\mathbb{E}[|\nabla_{v}X_{s}^{x}|^{6}\langle\nabla B(X_{s}^{x})\nabla_{v}X_{s}^{x},\nabla_{v}X_{s}^{x}\rangle]+4\eta\mathbb{E}[|\nabla_{v}X_{s}^{x}|^{6}\|\nabla\sigma(X_{s}^{x})\nabla_{v}X_{s}^{x}\|_{{\rm HS}}^{2}]\\
&+24\eta\mathbb{E}[|\nabla_{v}X_{s}^{x}|^{4}|\nabla\sigma(X_{s}^{x})\nabla_{v}X_{s}^{x}\nabla_{v}X_{s}^{x}|^{2}]\\
\leq&(8L+28\eta A_3)\mathbb{E}|\nabla_{v}X_{s}^{x}|^{8}.
\end{align*}
This inequality, together with $\nabla_{v}X_{0}^{x}=v$ and Gronwall inequality, implies
\begin{align*}
\mathbb{E}|\nabla_{v}X_{t}^{x}|^{8}\leq |v|^{8}e^{(8L+28\eta A_3)t}\leq e^{8L+28A_3}|v|^{8}.
\end{align*}

Using It\^{o}'s formula to $\varsigma_{1}(t)=\nabla_{v_{2}}\nabla_{v_{1}}X_{t}^{x}$, by \eqref{secequation}, the Cauchy-Schwarz inequality, and {\bf Assumption A}, we have
\begin{align*}
\frac{\dif}{\dif s}\mathbb{E}|\varsigma_{1}(s)|^{4}=&4\mathbb{E}\big[|\varsigma_{1}(s)|^{2}\langle\nabla B(X_{s}^{x})\varsigma_{1}(s)+\nabla^{2}B(X_{s}^{x})\nabla_{v_{2}}X_{s}^{x}\nabla_{v_{1}}X_{s}^{x},\varsigma_{1}(s)\rangle\big]\\
&+2\eta\mathbb{E}\big[|\varsigma_{1}(s)|^{2}\|\nabla\sigma(X_{s}^{x})\varsigma_{1}(s)+\nabla^{2}\sigma(X_{s}^{x})\nabla_{v_{2}}X_{s}^{x}\nabla_{v_{1}}X_{s}^{x}\|_{{\rm HS}}^{2}\big]\\
&+4\eta\mathbb{E}\big[|(\nabla\sigma(X_{s}^{x})\varsigma_{1}(s)+\nabla^{2}\sigma(X_{s}^{x})\nabla_{v_{2}}X_{s}^{x}\nabla_{v_{1}}X_{s}^{x})\varsigma_{1}(s)|^{2}\big]\\
\leq&4L\mathbb{E}|\varsigma_{1}(s)|^{4}+4A_{1}\mathbb{E}[|\nabla_{v_{2}}X_{s}^{x}||\nabla_{v_{1}}X_{s}^{x}||\varsigma_{1}(s)|^{3}]\\
&+12\eta\mathbb{E}\big[|\varsigma_{1}(s)|^{2}\big(\|\nabla\sigma(X_{s}^{x})\varsigma_{1}(s)\|_{{\rm HS}}^{2}+\|\nabla^{2}\sigma(X_{s}^{x})\nabla_{v_{2}}X_{s}^{x}\nabla_{v_{1}}X_{s}^{x}\|_{{\rm HS}}^{2}\big)\big]\\
\leq&4(L+3A_{3}\eta)\mathbb{E}|\varsigma_{1}(s)|^{4}+4A_{1}\mathbb{E}[|\nabla_{v_{2}}X_{s}^{x}||\nabla_{v_{1}}X_{s}^{x}||\varsigma_{1}(s)|^{3}]\\
&+12A_{4}\eta\mathbb{E}\big[|\varsigma_{1}(s)|^{2}|\nabla_{v_{2}}X_{s}^{x}|^{2}|\nabla_{v_{1}}X_{s}^{x}|^{2}\big],
\end{align*}
By Young's inequality and (\ref{grad3}), we further have
\begin{align*}
\frac{\dif}{\dif s}\mathbb{E}|\varsigma_{1}(s)|^{4}\leq&(4L+3A_{1}+12A_{3}+6A_{4})\mathbb{E}|\varsigma_{1}(s)|^{4}+(A_{1}+6A_{4})\mathbb{E}[|\nabla_{v_{2}}X_{s}^{x}|^{4}|\nabla_{v_{1}}X_{s}^{x}|^{4}]\\
\leq&(4L+3A_{1}+12A_{3}+6A_{4})\mathbb{E}|\varsigma_{1}(s)|^{4}+(A_{1}+6A_{4})e^{8L+28A_3}|v_{1}|^{4}|v_{2}|^{4}.
\end{align*}
This inequality, together with $\varsigma_{1}(0)=0$ and Gronwall's inequality, implies
\begin{align*}
\mathbb{E}|\varsigma_{1}(t)|^{4}\leq&(A_{1}+6A_{4})e^{8L+28A_3}|v_{1}|^{4}|v_{2}|^{4}t+(A_{1}+6A_{4})e^{8L+28A_3}|v_{1}|^{4}|v_{2}|^{4}\\
&\qquad\qquad\qquad\qquad\qquad(4L+3A_{1}+12A_{3}+6A_{4})\int_{0}^{t}se^{(4L+3A_{1}+12A_{3}+6A_{4})(t-s)}\dif s\\
\leq&2(A_{1}+6A_{4})e^{12L+3A_{1}+12A_{3}+6A_{4}+28A_3}|v_{1}|^{4}|v_{2}|^{4}.
\end{align*}

Writing $\varsigma_{2}(t)=\nabla_{v_{3}}\nabla_{v_{2}}\nabla_{v_{1}}X_{t}^{x},$ by It\^{o}'s formula, one has
\begin{align*}
&\frac{\dif}{\dif s}\mathbb{E}|\varsigma_{2}(s)|^{2}\\
=&2\mathbb{E}\big[\langle\nabla B(X_{s}^{x})\varsigma_{2}(s)+\nabla^{2}B(X_{s}^{x})\nabla_{v_{3}}X_{s}^{x}\nabla_{v_{2}}\nabla_{v_{1}}X_{s}^{x},\varsigma_{2}(s)\rangle\big]\\
&+2\mathbb{E}\big[\langle\nabla^{2}B(X_{s}^{x})\nabla_{v_{3}}\nabla_{v_{2}}X_{s}^{x}\nabla_{v_{1}}X_{s}^{x}
+\nabla^{2}B(X_{s}^{x})\nabla_{v_{2}}X_{s}^{x}\nabla_{v_{3}}\nabla_{v_{1}}X_{s}^{x}
,\varsigma_{2}(s)\rangle\big]\\
&+2\mathbb{E}\big[\langle\nabla^{3} B(X_{s}^{x})\nabla_{v_{3}}X_{s}^{x}\nabla_{v_{2}}X_{s}^{x}\nabla_{v_{1}}X_{s}^{x},\varsigma_{2}(s)\rangle\big]\\
&+\eta\mathbb{E}\big[\|\nabla\sigma(X_{s}^{x})\varsigma_{2}(s)+\nabla^{2}\sigma(X_{s}^{x})\big(\nabla_{v_{3}}X_{s}^{x}\nabla_{v_{2}}\nabla_{v_{1}}X_{s}^{x}
+\nabla_{v_{3}}\nabla_{v_{2}}X_{s}^{x}\nabla_{v_{1}}X_{s}^{x}\big)\\
&\qquad\quad+\nabla^{3}\sigma(X_{s}^{x})\nabla_{v_{3}}X_{s}^{x}\nabla_{v_{2}}X_{s}^{x}\nabla_{v_{1}}X_{s}^{x}
+\nabla^{2}\sigma(X_{s}^{x})\nabla_{v_{2}}X_{s}^{x}\nabla_{v_{3}}\nabla_{v_{1}}X_{s}^{x}\|_{{\rm HS}}^{2}\big].
\end{align*}
By (\ref{P1}), (\ref{Q1}) and (\ref{Q2}), we further have
\begin{align*}
\frac{\dif}{\dif s}\mathbb{E}|\varsigma_{2}(s)|^{2}\leq&2L\mathbb{E}|\varsigma_{2}(s)|^{2}+2A_{2}\mathbb{E}\big[|\nabla_{v_{3}}X_{s}^{x}||\nabla_{v_{2}}X_{s}^{x}|
|\nabla_{v_{1}}X_{s}^{x}||\varsigma_{2}(s)|\big]\\
&+2A_{1}\mathbb{E}\big[\big(|\nabla_{v_{3}}X_{s}^{x}||\nabla_{v_{2}}\nabla_{v_{1}}X_{s}^{x}|
+|\nabla_{v_{2}}X_{s}^{x}||\nabla_{v_{3}}\nabla_{v_{1}}X_{s}^{x}|\big)|\varsigma_{2}(s)|\big]\\
&+2A_{1}\mathbb{E}\big[|\nabla_{v_{1}}X_{s}^{x}||\nabla_{v_{3}}\nabla_{v_{2}}X_{s}^{x}||\varsigma_{2}(s)|\big]+5A_{3}\eta\mathbb{E}|\varsigma_{2}(s)|^{2}\\
&+5\eta\big(A_{5}\mathbb{E}\big[|\nabla_{v_{3}}X_{s}^{x}||\nabla_{v_{2}}X_{s}^{x}||\nabla_{v_{1}}X_{s}^{x}|\big]^{2}
+A_{4}\mathbb{E}\big[|\nabla_{v_{3}}X_{s}^{x}||\nabla_{v_{2}}\nabla_{v_{1}}X_{s}^{x}|\big]^{2}\big)\\
&+5A_{4}\eta\mathbb{E}\big[|\nabla_{v_{2}}X_{s}^{x}|^{2}|\nabla_{v_{3}}\nabla_{v_{1}}X_{s}^{x}|^{2}
+|\nabla_{v_{1}}X_{s}^{x}|^{2}|\nabla_{v_{3}}\nabla_{v_{2}}X_{s}^{x}|^{2}\big].
\end{align*}
Using Young's inequality, (\ref{grad3}) and (\ref{gradsec}), we can get
\begin{align*}
\frac{\dif}{\dif s}\mathbb{E}|\varsigma_{2}(s)|^{2}\leq&(2L+A_{2}+3A_{1}+5A_{3})\mathbb{E}|\varsigma_{2}(s)|^{2}+C_{A,L}|v_{1}|^{2}|v_{2}|^{2}|v_{3}|^{2}.
\end{align*}
This inequality, together with $\varsigma_{2}(0)=0$ and Gronwall inequality, implies
\begin{align*}
\mathbb{E}|\varsigma_{2}(t)|^{2}\leq&C_{A,L}|v_{1}|^{2}|v_{2}|^{2}|v_{3}|^{2}.
\end{align*}
\end{proof}

Next, we use Bismut's approach to Malliavin calculus for SDE (\ref{SDE})(\cite{Nor86}). Let $u\in L_{loc}^{2}([0,\infty)\times(\Omega,\mathcal{F},\mathbb{P});\mathbb{R}^{d}),$ i.e., $\mathbb{E}\int_{0}^{t}|u(s)|^{2}\dif s<\infty$ for all $t>0.$ Further assume that $u$ is adapted to the filtration $(\mathcal{F}_{t})_{t\geq0}$ with $\mathcal{F}_{t}:=\sigma(B_{s}:0\leq s\leq t);$ i.e., $u(t)$ is $\mathcal{F}_{t}$ measurable for $t\geq0.$ Define
\begin{align}\label{bismutf}
U_{t}=\int_{0}^{t}u(s)\dif s,\quad t\geq0.
\end{align}
For a $t>0,$ let $F_{t}:\mathcal{C}([0,t],\mathbb{R}^{d})\rightarrow\mathbb{R}$ be a $\mathcal{F}_{t}$ measurable map. If the following limit exists
\begin{align*}
D_{U}F_{t}(B)=\lim_{\epsilon\rightarrow0}\frac{F_{t}(B+\epsilon U)-F_{t}(B)}{\epsilon}
\end{align*}
in $L^{2}((\Omega,\mathcal{F},\mathbb{P});\mathbb{R}),$ then $F_{t}(B)$ is said to be {\it Malliavin differentiable} and $D_{U}F_{t}(B)$ is called the Malliavin derivative of $F_{t}(B)$ in the direction $u$.

Let $F_{t}(B)$ and $G_{t}(B)$ both be Malliavin differentiable, then the following product rule holds:
\begin{align}\label{chain}
D_{U}(F_{t}(B)G_{t}(B))=F_{t}(B)D_{U}G_{t}(B)+G_{t}(B)D_{U}F_{t}(B).
\end{align}
When
\begin{align*}
F_{t}(B)=\int_{0}^{t}\langle a(s),\dif B(s)\rangle,
\end{align*}
where $a(s)=(a_{1}(s),\cdots,a_{d}(s))$ is a deterministic function such that $\int_{0}^{t}|a(s)|^{2}\dif s<\infty$ for all $t>0,$ and it is easy to verify that
\begin{align*}
D_{U}F_{t}(B)=\int_{0}^{t}\langle a(s),u(s)\rangle \dif s.
\end{align*}
Moreover, if $a(s)$ is a stochastic process adapted to the filtration $\mathcal{F}_{s}$ such that $\mathbb{E}\int_{0}^{t}|a(s)|^{2}\dif s<\infty$ for all $t>0$, it is easy to check that
\begin{align}\label{chain2}
D_{U}F_{t}(B)=\int_{0}^{t}\langle a(s),u(s)\rangle \dif s+\int_{0}^{t}\langle D_{U}a(s),\dif B_{s}\rangle.
\end{align}

Then, we consider the following integration by parts formula, which is called Bismut's formula. For Malliavin differentiable $F_{t}(B)$ such that $F_{t}(B),$ $D_{U}F_{t}(B)\in L^{2}((\Omega,\mathcal{F},\mathbb{P});\mathbb{R}),$ we have
\begin{align}\label{bismut}
\mathbb{E}[D_{U}F_{t}(B)]=\mathbb{E}\big[F_{t}(B)\int_{0}^{t}\langle u(s),\dif B_{s}\rangle\big].
\end{align}

Let $\phi\in {\rm Lip}(1)$ and let $F_{t}(B)=(F_{t}^{1}(B),\cdots,F_{t}^{d}(B))$ be a $d$-dimensional Malliavin differentiable functional. The following chain rule holds:
\begin{align*}
D_{U}\phi(F_{t}(B))=\langle\nabla\phi(F_{t}(B)),D_{U}F_{t}(B)\rangle=\sum_{i=1}^{d}\partial_{i}\phi(F_{t}(B))D_{U}F_{t}^{i}(B).
\end{align*}

Now, we come back to the SDE (\ref{SDE}). Fixing $t\geq0$ and $x\in\mathbb{R}^{d},$ the solution $X_{t}^{x}$ is a $d$-dimensional functional of Brownian motion $(B_{s})_{0\leq s\leq t}.$

The following Malliavin derivative of $X_{t}^{x}$ along the direction $U$ exists in $L^{2}((\Omega,\mathcal{F},\mathbb{P});\mathbb{R}^{d})$ and is defined by
\begin{align*}
D_{U}X_{t}^{x}(B)=\lim_{\epsilon\rightarrow0}\frac{X_{t}^{x}(B+\epsilon U)-X_{t}^{x}(B)}{\epsilon}.
\end{align*}
We drop the $B$ in $D_{U}X_{t}^{x}(B)$ and write $D_{U}X_{t}^{x}=D_{U}X_{t}^{x}(B)$ for simplicity. It satisfies the equation
\begin{align*}
\dif D_{U}X_{t}^{x}=\nabla B(X_{t}^{x})D_{U}X_{t}^{x}\dif t+\eta^{\frac{1}{2}}\nabla\sigma(X_{t}^{x})D_{U}X_{t}^{x}\dif B_{t}+\eta^{\frac{1}{2}}\sigma(X_{t}^{x})u(t)\dif t, \quad D_{U}X_{0}^{x}=0,
\end{align*}
and the equation has a unique solution:
\begin{align*}
D_{U}X_{t}^{x}=\int_{0}^{t}J_{r,t}^{x}\eta^{\frac{1}{2}}\sigma(X_{r}^{x})u(r)\dif r.
\end{align*}
Noticing that $\nabla_{v}X_{t}^{x}=J_{0,t}^{x}v,$ if we take
\begin{align}\label{malliavin1}
u(s)=\frac{1}{t}\eta^{-\frac{1}{2}}\sigma(X_{s}^{x})^{-1}\nabla_{v}X_{s}^{x},\quad 0\leq s\leq t,
\end{align}
then (\ref{A1}) and (\ref{grad3}) imply $u\in L_{loc}^{2}([0,\infty)\times(\Omega,\mathcal{F},\mathbb{P});\mathbb{R}^{d}).$ Since $\nabla_{v}X_{r}^{x}=J_{0,r}^{x}v$ and $J_{0,r}^{x}J_{r,t}^{x}=J_{0,t}^{x},$ for all $0\leq r\leq t,$ we have
\begin{align}\label{malliavin2}
D_{U}X_{t}^{x}=\nabla_{v}X_{t}^{x}
\end{align}
and
\begin{align}\label{ma1}
D_{U}X_{s}^{x}=\frac{s}{t}\nabla_{v}X_{s}^{x},\quad 0\leq s\leq t.
\end{align}
Let $v_{1},v_{2}\in\mathbb{R}^{d},$ and define $u_{i}$ and $U_{i}$ as (\ref{malliavin1}) and (\ref{bismutf}), respectively, for $i=1,2.$ We can similarly define $D_{U_{2}}\nabla_{v_{1}}X_{s}^{x},$ which satisfies the following equation: for $s\in[0,t],$
\begin{align}\label{malsec}
&\dif D_{U_{2}}\nabla_{v_{1}}X_{s}^{x}\nonumber\\
=&\big[\nabla B(X_{s}^{x})D_{U_{2}}\nabla_{v_{1}}X_{s}^{x}+\nabla^{2}B(X_{s}^{x})D_{U_{2}}X_{s}^{x}\nabla_{v_{1}}X_{s}^{x}
+\eta^{\frac{1}{2}}\nabla\sigma(X_{s}^{x})\nabla_{v_{1}}X_{s}^{x}u_{2}(s)\big]\dif s\nonumber\\
&+\eta^{\frac{1}{2}}\big[\nabla\sigma(X_{s}^{x})D_{U_{2}}\nabla_{v_{1}}X_{s}^{x}
+\nabla^{2}\sigma(X_{s}^{x})D_{U_{2}}X_{s}^{x}\nabla_{v_{1}}X_{s}^{x}\big]\dif B_{s}\nonumber\\
=&\big[\nabla B(X_{s}^{x})D_{U_{2}}\nabla_{v_{1}}X_{s}^{x}+\frac{s}{t}\nabla^{2}B(X_{s}^{x})\nabla_{v_{2}}X_{s}^{x}\nabla_{v_{1}}X_{s}^{x}\nonumber\\
&\qquad\qquad\qquad\qquad\qquad+\frac{1}{t}\nabla\sigma(X_{s}^{x})\nabla_{v_{1}}X_{s}^{x}\sigma(X_{s}^{x})^{-1}\nabla_{v_{2}}X_{s}^{x}\big]\dif s\nonumber\\
&+\eta^{\frac{1}{2}}\big[\nabla\sigma(X_{s}^{x})D_{U_{2}}\nabla_{v_{1}}X_{s}^{x}
+\frac{s}{t}\nabla^{2}\sigma(X_{s}^{x})\nabla_{v_{2}}X_{s}^{x}\nabla_{v_{1}}X_{s}^{x}\big]\dif B_{s},
\end{align}
with $D_{U_{2}}\nabla_{v_{1}}X_{0}^{x}=0,$ where the second equality is by (\ref{malliavin1}) and (\ref{ma1}).

For further use, we define
\begin{align*}
\mathcal{I}_{v_{1}}^{x}(t):=\frac{1}{t}\int_{0}^{t}\langle \eta^{-\frac{1}{2}}\sigma(X_{s}^{x})^{-1}\nabla_{v_{1}}X_{s}^{x},\dif B_{s}\rangle
\end{align*}
and
\begin{align*}
\mathcal{R}_{v_{1},v_{2}}^{x}(t):=\nabla_{v_{2}}\nabla_{v_{1}}X_{t}^{x}-D_{U_{2}}\nabla_{v_{1}}X_{t}^{x}.
\end{align*}

Then, we have the following upper bounds on Malliavin derivatives.

\begin{lemma}
Let $v_{1},v_{2}\in\mathbb{R}^{d}$ and
\begin{align*}
U_{i,s}=\int_{0}^{s}u_{i}(r)\dif r, \quad 0\leq s\leq t,
\end{align*}
where $u_{i}(r)=\frac{1}{t}\eta^{-\frac{1}{2}}\sigma(X_{r}^{x})^{-1}\nabla_{v_{i}}X_{r}^{x}$ for $0\leq r\leq t$ and $i=1,2.$ Then, as $\eta\in(0,\delta]$ and $t\in(0,1)$, we have
\begin{align}\label{malgrad1}
\mathbb{E}|D_{U_{2}}\nabla_{v_{1}}X_{s}^{x}|^{2}\leq C_{A,L,d}(1+\frac{1}{t})|v_{1}|^{2}|v_{2}|^{2}
\end{align}
and
\begin{align}\label{malgrad2}
\mathbb{E}|D_{U_{2}}\nabla_{v_{1}}X_{s}^{x}|^{4}\leq C_{A,L,d}(1+\frac{1}{t^{3}})|v_{1}|^{4}|v_{2}|^{4}.
\end{align}
\end{lemma}
\begin{proof}
Writing $\zeta(s)=D_{U_{2}}\nabla_{v_{1}}X_{s}^{x}$, by It\^{o}'s formula, Cauchy-Schwarz inequality, (\ref{P1}), (\ref{Q1}) and (\ref{Q2}), we have
\begin{align*}
\frac{\dif}{\dif r}\mathbb{E}|\zeta(r)|^{2}=&2\mathbb{E}[\langle\nabla B(X_{r}^{x})\zeta(r)+\frac{r}{t}\nabla^{2}B(X_{r}^x)\nabla_{v_{2}}X_{r}^{x}\nabla_{v_{1}}X_{r}^{x},\zeta_{r}\rangle]\\
&+2\mathbb{E}[\langle\frac{1}{t}\nabla\sigma(X_{r}^{x})\nabla_{v_{1}}X_{r}^{x}\sigma(X_{r}^{x})^{-1}\nabla_{v_{2}}X_{r}^{x},\zeta_{r}\rangle]\\
&+\eta\mathbb{E}[\|\nabla\sigma(X_{r}^{x})\zeta(r)+\frac{r}{t}\nabla^{2}\sigma(X_{r}^{x})\nabla_{v_{2}}X_{r}^{x}\nabla_{v_{1}}X_{r}^{x}\|_{\rm{HS}}^{2}]\\
\leq&2L\mathbb{E}|\zeta(r)|^{2}+2A_{1}\frac{r}{t}\mathbb{E}[|\zeta(r)||\nabla_{v_{1}}X_{r}^{x}||\nabla_{v_{2}}X_{r}^{x}|]\\
&+2\sqrt{A_{3}}\frac{1}{t}\mathbb{E}[|\zeta(r)||\nabla_{v_{1}}X_{r}^{x}||\nabla_{v_{2}}X_{r}^{x}|\|\sigma(X_{r}^{x})^{-1}\|_{\rm{HS}}]\\
&+2\eta\mathbb{E}\big[A_{3}|\zeta(r)|^{2}+A_{4}\frac{r^{2}}{t^{2}}|\nabla_{v_{1}}X_{r}^{x}|^{2}|\nabla_{v_{2}}X_{r}^{x}|^{2}\big].
\end{align*}
By Young's inequality, (\ref{A1}), Cauchy's inequality, (\ref{grad3}) and the assumption $\eta\leq \delta$, we further have
\begin{align*}
\frac{\dif}{\dif r}\mathbb{E}|\zeta(r)|^{2}\leq&(2L+A_{1}+3A_{3})\mathbb{E}|\zeta(r)|^{2}\\
&+\mathbb{E}\left[\left(A_{1}+2A_{4}+\frac{1}{t^{2}}\|\sigma(X_{r}^{x})^{-1}\|_{\rm{HS}}^{2}\right)|\nabla_{v_{1}}X_{r}^{x}|^{2}|\nabla_{v_{2}}X_{r}^{x}|^{2}\right]\\
\leq&(2L+A_{1}+3A_{3})\mathbb{E}|\zeta(r)|^{2}+\left(A_{1}+2A_{4}+\frac{\eta d}{\delta}\frac{1}{t^{2}}\right)e^{4L+14A_3}|v_{1}|^{2}|v_{2}|^{2}\\
\leq&(2L+A_{1}+3A_{3})\mathbb{E}|\zeta(r)|^{2}+\left(A_{1}+2A_{4}+\frac{d}{t^{2}}\right)e^{4L+14A_3}|v_{1}|^{2}|v_{2}|^{2}.
\end{align*}
This inequality, together with $\zeta(0)=0$ and Gronwall inequality, implies
\begin{align*}
\mathbb{E}|\zeta(s)|^{2}
\leq&\left(A_{1}+2A_{4}+\frac{d}{t^{2}}\right)e^{4L+14A_3}|v_{1}|^{2}|v_{2}|^{2}\left[s
+\int_{0}^{s}r(2L+A_{1}+3A_{3})e^{(2L+A_{1}+3A_{3})(s-r)}\dif r\right]\\
\leq&C_{A,L,d}(1+\frac{1}{t})|v_{1}|^{2}|v_{2}|^{2}.
\end{align*}

For \eqref{malgrad2}, It\^{o}'s formula, the Cauchy-Schwarz inequality and {\bf Assumption A} yield
\begin{align*}
\frac{\dif}{\dif r}\mathbb{E}|\zeta(r)|^{4}=&4\mathbb{E}[|\zeta(r)|^{2}\langle\nabla B(X_{r}^{x})\zeta(r)+\frac{r}{t}\nabla^{2}B(X_{r})\nabla_{v_{2}}X_{r}^{x}\nabla_{v_{1}}X_{r}^{x},\zeta_{r}\rangle]\\
&+4\mathbb{E}[|\zeta(r)|^{2}\langle\frac{1}{t}\nabla\sigma(X_{r}^{x})\nabla_{v_{1}}X_{r}^{x}\sigma(X_{r}^{x})^{-1}\nabla_{v_{2}}X_{r}^{x},\zeta_{r}\rangle]\\
&+2\eta\mathbb{E}[|\zeta(r)|^{2}\|\nabla\sigma(X_{r}^{x})\zeta(r)+\frac{r}{t}\nabla^{2}\sigma(X_{r}^{x})\nabla_{v_{2}}X_{r}^{x}\nabla_{v_{1}}X_{r}^{x}\|_{\rm{HS}}^{2}]\\
&+4\eta\mathbb{E}[|(\nabla\sigma(X_{r}^{x})\zeta(r)+\frac{r}{t}\nabla^{2}\sigma(X_{r}^{x})\nabla_{v_{2}}X_{r}^{x}\nabla_{v_{1}}X_{r}^{x})\zeta(r)|^{2}]\\
\leq&4L\mathbb{E}|\zeta(r)|^{4}+4A_{1}\mathbb{E}[|\zeta(r)|^{3}|\nabla_{v_{1}}X_{r}^{x}||\nabla_{v_{2}}X_{r}^{x}|]\\
&+4\sqrt{A_{3}}\frac{1}{t}\mathbb{E}[|\zeta(r)|^{3}|\nabla_{v_{1}}X_{r}^{x}||\nabla_{v_{2}}X_{r}^{x}|\|\sigma(X_{r}^{x})^{-1}\|_{\rm{HS}}]\\
&+12\eta\mathbb{E}\big[|\zeta(r)|^{2}\big(A_{3}|\zeta(r)|^{2}+A_{4}|\nabla_{v_{1}}X_{r}^{x}|^{2}|\nabla_{v_{2}}X_{r}^{x}|^{2}\big)\big].
\end{align*}
Then, by Young's inequality, (\ref{A1}), the assumption $\eta\leq\delta$ and (\ref{grad3}), we have
\begin{align*}
\frac{\dif}{\dif r}\mathbb{E}|\zeta(r)|^{4}\leq&(4L+3A_{1}+3A_{3}^{\frac{2}{3}}+12A_{3}+6A_{4})\mathbb{E}|\zeta(r)|^{4}\\
&+\mathbb{E}\left[\left(A_{1}+6A_{4}+\frac{1}{t^{4}}\|\sigma(X_{r}^{x})^{-1}\|_{\rm{HS}}^{4}\right)|\nabla_{v_{1}}X_{r}^{x}|^{4}|\nabla_{v_{2}}X_{r}^{x}|^{4}\right]\\
\leq&(4L+3A_{1}+3A_{3}^{\frac{2}{3}}+12A_{3}+6A_{4})\mathbb{E}|\zeta(r)|^{4}\\
&+\left(A_{1}+6A_{4}+\frac{d^{2}}{t^{4}}\right)e^{8L+28A_3}|v_{1}|^{4}|v_{2}|^{4}.
\end{align*}
This inequality, together with $\zeta(0)=0$ and Gronwall inequality, implies
\begin{align*}
\mathbb{E}|\zeta(s)|^{4}
\leq&C_{A,L,d}(1+\frac{1}{t^{3}})|v_{1}|^{4}|v_{2}|^{4}.
\end{align*}
\end{proof}

Based on the results above, we have the following two lemmas:

\begin{lemma}
Let $v_{1},v_{2}\in\mathbb{R}^{d}$ and $x\in\mathbb{R}^{d}.$ Then, for all $\eta\leq(0,\delta]$ and $t\in(0,1]$, we have
\begin{align}\label{1}
\mathbb{E}|\mathcal{I}_{v_{1}}^{x}(t)|^{4}\leq\frac{10}{t^{2}}\frac{d^{2}}{\delta^{2}}e^{4L+14A_3}|v_{1}|^{4},
\end{align}
\begin{align}\label{2}
\mathbb{E}|\nabla_{v_{2}}\mathcal{I}_{v_{1}}^{x}(t)|^{2}\leq \frac{C_{A,L,d}}{t\delta}|v_{1}|^{2}|v_{2}|^{2},
\end{align}
\begin{align}\label{3}
\mathbb{E}|D_{U_{2}}\mathcal{I}_{v_{1}}^{x}(t)|^{2}\leq C_{A,L,d}\frac{1}{t\delta}(1+\frac{1}{t\delta})|v_{1}|^{2}|v_{2}|^{2}.
\end{align}
\end{lemma}
\begin{proof}
By Burkholder's inequality \cite[Theorem 2]{Ren08}, (\ref{A1}) and (\ref{grad3}), we have
\begin{align*}
\mathbb{E}|\mathcal{I}_{v_{1}}^{x}(t)|^{4}=&\mathbb{E}|\frac{1}{t}\int_{0}^{t}\langle\eta^{-\frac{1}{2}}\sigma(X_{s}^{x})^{-1}\nabla_{v_{1}}X_{s}^{x},\dif B_{s}\rangle|^{4}\\
\leq&\frac{4\sqrt{2}}{t^{4}}\eta^{-2}\mathbb{E}\big(\int_{0}^{t}|\sigma(X_{s}^{x})^{-1}\nabla_{v_{1}}X_{s}^{x}|^{2}\dif s\big)^{2}\\
\leq&\frac{10}{t^{4}}\frac{d^{2}}{\delta^{2}}\mathbb{E}\big(\int_{0}^{t}|\nabla_{v_{1}}X_{s}^{x}|^{2}\dif s\big)^{2}\\
\leq&\frac{10}{t^{3}}\frac{d^{2}}{\delta^{2}}\int_{0}^{t}\mathbb{E}|\nabla_{v_{1}}X_{s}^{x}|^{4}\dif s
\leq\frac{10}{t^{2}}\frac{d^{2}}{\delta^{2}}e^{4L+14A_3}|v_{1}|^{4}.
\end{align*}

For \eqref{2}, the definition of $\mathcal{I}_{v_{1}}^{x}(t)$ yields
\begin{align*}
\nabla_{v_{2}}\mathcal{I}_{v_{1}}^{x}(t)=&\frac{1}{t}\int_{0}^{t}\nabla_{v_{2}}\langle \eta^{-\frac{1}{2}}\sigma(X_{s}^{x})^{-1}\nabla_{v_{1}}X_{s}^{x},\dif B_{s}\rangle\\
=&\frac{1}{t}\eta^{-\frac{1}{2}}\int_{0}^{t}\langle -\sigma(X_{s}^{x})^{-1}\nabla_{v_{2}}\sigma(X_{s}^{x})\sigma(X_{s}^{x})^{-1}\nabla_{v_{1}}X_{s}^{x}+\sigma(X_{s}^{x})^{-1}\nabla_{v_{2}}\nabla_{v_{1}}X_{s}^{x},\dif B_{s}\rangle.
\end{align*}
Then, by It\^{o} isometry, (\ref{Q1}) and (\ref{A1}), we have
\begin{align*}
&\mathbb{E}|\nabla_{v_{2}}\mathcal{I}_{v_{1}}^{x}(t)|^{2}\\
=&\frac{1}{t^{2}}\eta^{-1}\int_{0}^{t}\mathbb{E}|\sigma(X_{s}^{x})^{-1}\nabla_{v_{2}}
\sigma(X_{s}^{x})\sigma(X_{s}^{x})^{-1}\nabla_{v_{1}}X_{s}^{x}-\sigma(X_{s}^{x})^{-1}\nabla_{v_{2}}\nabla_{v_{1}}X_{s}^{x}|^{2}\dif s\\
\leq&\frac{2}{t^{2}}\eta^{-1}\int_{0}^{t}\mathbb{E}[|\sigma(X_{s}^{x})^{-1}\nabla_{v_{2}}
\sigma(X_{s}^{x})\sigma(X_{s}^{x})^{-1}\nabla_{v_{1}}X_{s}^{x}|^{2}+|\sigma(X_{s}^{x})^{-1}\nabla_{v_{2}}\nabla_{v_{1}}X_{s}^{x}|^{2}]\dif s\\
\leq&\frac{2}{t^{2}}\eta^{-1}\int_{0}^{t}\mathbb{E}[A_{3}\|\sigma(X_{s}^{x})^{-1}\|_{{\rm HS}}^{4}|\nabla_{v_{2}}X_{s}^{x}|^{2}|\nabla_{v_{1}}X_{s}^{x}|^{2}
+\|\sigma(X_{s}^{x})^{-1}\|_{{\rm HS}}^{2}|\nabla_{v_{2}}\nabla_{v_{1}}X_{s}^{x}|^{2}]\dif s\\
\leq&\frac{2}{t^{2}}\eta^{-1}\int_{0}^{t}\mathbb{E}[A_{3}\frac{\eta^{2}d^{2}}{\delta^{2}}|\nabla_{v_{2}}X_{s}^{x}|^{2}|\nabla_{v_{1}}X_{s}^{x}|^{2}
+\frac{\eta d}{\delta}|\nabla_{v_{2}}\nabla_{v_{1}}X_{s}^{x}|^{2}]\dif s.
\end{align*}
By Cauchy-Schwarz inequality, (\ref{grad3}) and (\ref{gradsec}), we further have
\begin{align*}
\mathbb{E}|\nabla_{v_{2}}\mathcal{I}_{v_{1}}^{x}(t)|^{2}\leq&
C_{A,L,d}\frac{1}{t^{2}}\eta^{-1}\int_{0}^{t}\left[\frac{\eta^{2}}{\delta^{2}}+\frac{\eta}{\delta}\right]|v_{1}|^{2}|v_{2}|^{2}\dif s\leq \frac{C_{A,L,d}}{t\delta}|v_{1}|^{2}|v_{2}|^{2}.
\end{align*}

Recall (\ref{malliavin1}) and (\ref{ma1}), it is easy to see that $D_{U_{2}}\mathcal{I}_{v_{1}}^{x}(t)$ can be computed by (\ref{chain2}) as
\begin{align*}
&D_{U_{2}}\mathcal{I}_{v_{1}}^{x}(t)\\
=&\frac{1}{t}\int_{0}^{t}\langle \eta^{-\frac{1}{2}}\sigma(X_{s}^{x})^{-1}\nabla_{v_{1}}X_{s}^{x},u_{2}(s)\rangle \dif s+\frac{1}{t}\int_{0}^{t}\langle \eta^{-\frac{1}{2}}\sigma(X_{s}^{x})^{-1}D_{U_{2}}\nabla_{v_{1}}X_{s}^{x},\dif B_{s}\rangle\\
&-\frac{1}{t}\int_{0}^{t}\eta^{-\frac{1}{2}}\langle \sigma(X_{s}^{x})^{-1}\nabla\sigma(X_{s}^{x})\sigma(X_{s}^{x})^{-1}D_{U_{2}}X_{s}^{x}\nabla_{v_{1}}X_{s}^{x},\dif B_{s}\rangle\\
=&\frac{\eta^{-1}}{t^{2}}\int_{0}^{t}\langle \sigma(X_{s}^{x})^{-1}\nabla_{v_{1}}X_{s}^{x},\sigma(X_{s}^{x})^{-1}\nabla_{v_{2}}X_{s}^{x}\rangle \dif s+\frac{\eta^{-\frac{1}{2}}}{t}\int_{0}^{t}\langle \sigma(X_{s}^{x})^{-1}D_{U_{2}}\nabla_{v_{1}}X_{s}^{x},\dif B_{s}\rangle\\
&-\frac{1}{t}\eta^{-\frac{1}{2}}\int_{0}^{t}\langle \sigma(X_{s}^{x})^{-1}\nabla\sigma(X_{s}^{x})\sigma(X_{s}^{x})^{-1}\frac{s}{t}\nabla_{v_{2}}X_{s}^{x}\nabla_{v_{1}}X_{s}^{x},\dif B_{s}\rangle.
\end{align*}
Then, by the Cauchy-Schwarz inequality, It\^{o} isometry and (\ref{Q1}), we have
\begin{align*}
&\mathbb{E}|D_{U_{2}}\mathcal{I}_{v_{1}}^{x}(t)|^{2}\\
\leq&\frac{3\eta^{-2}}{t^{3}}\int_{0}^{t}\mathbb{E}|\langle \sigma(X_{s}^{x})^{-1}\nabla_{v_{1}}X_{s}^{x},\sigma(X_{s}^{x})^{-1}\nabla_{v_{2}}X_{s}^{x}\rangle|^{2} \dif s+\frac{3\eta^{-1}}{t^{2}}\int_{0}^{t}\mathbb{E}|\sigma(X_{s}^{x})^{-1}D_{U_{2}}\nabla_{v_{1}}X_{s}^{x}|^{2}\dif s\\
&+\frac{3\eta^{-1}}{t^{2}}\int_{0}^{t}\mathbb{E}|\sigma(X_{s}^{x})^{-1}\nabla\sigma(X_{s}^{x})\sigma(X_{s}^{x})^{-1}\frac{s}{t}\nabla_{v_{2}}X_{s}^{x}\nabla_{v_{1}}X_{s}^{x}|^{2}\dif s\\
\leq&\frac{3\eta^{-2}}{t^{3}}\int_{0}^{t}\mathbb{E}[\|\sigma(X_{s}^{x})^{-1}\|_{{\rm HS}}^{4}|\nabla_{v_{1}}X_{s}^{x}|^{2}|\nabla_{v_{2}}X_{s}^{x}|^{2}] \dif s+\frac{3\eta^{-1}}{t^{2}}\int_{0}^{t}\mathbb{E}[\|\sigma(X_{s}^{x})^{-1}\|_{{\rm HS}}^{2}|D_{U_{2}}\nabla_{v_{1}}X_{s}^{x}|^{2}]\dif s\\
&+A_{3}\frac{3\eta^{-1}}{t^{2}}\int_{0}^{t}\mathbb{E}\big[\|\sigma(X_{s}^{x})^{-1}\|_{{\rm HS}}^{4}|\nabla_{v_{2}}X_{s}^{x}|^{2}|\nabla_{v_{1}}X_{s}^{x}|^{2}\big]\dif s.
\end{align*}
It follows from (\ref{A1}), the Cauchy-Schwarz inequality, (\ref{grad3}) and (\ref{malgrad1}) that
\begin{align*}
&\mathbb{E}|D_{U_{2}}\mathcal{I}_{v_{1}}^{x}(t)|^{2}\\
\leq&\frac{3\eta^{-1}}{t^{2}}(\frac{\eta^{-1}}{t}+A_{3})\int_{0}^{t}\frac{\eta^{2}d^{2}}{\delta^{2}}\mathbb{E}[|\nabla_{v_{1}}X_{s}^{x}|^{2}|\nabla_{v_{2}}X_{s}^{x}|^{2}] \dif s+\frac{3\eta^{-1}}{t^{2}}\int_{0}^{t}\frac{\eta d}{\delta}\mathbb{E}|D_{U_{2}}\nabla_{v_{1}}X_{s}^{x}|^{2}\dif s\\
\leq&C_{A,L,d}\frac{1}{t\delta}(1+\frac{1}{t\delta})|v_{1}|^{2}|v_{2}|^{2}.
\end{align*}
\end{proof}

Furthermore, let $v_{1},v_{2},v_{3}\in\mathbb{R}^{d},$ and define $u_{i}$ and $U_{i}$ as (\ref{malliavin1}) and (\ref{bismutf}), respectively, for $i=1,2,3.$ From (\ref{malsec}), we can similarly define $\nabla_{v_{3}}D_{U_{2}}\nabla_{v_{1}}X_{s}^{x},$ which satisfies the following equation: for $s\in[0,t],$
\begin{align*}
&\dif \nabla_{v_{3}}D_{U_{2}}\nabla_{v_{1}}X_{s}^{x}\\
=&\big[\nabla^{2} B(X_{s}^{x})\nabla_{v_{3}}X_{s}^{x}D_{U_{2}}\nabla_{v_{1}}X_{s}^{x}+\nabla B(X_{s}^{x})\nabla_{v_{3}}D_{U_{2}}\nabla_{v_{1}}X_{s}^{x}\\
&+\frac{s}{t}\nabla^{3}B(X_{s}^{x})\nabla_{v_{3}}X_{s}^{x}\nabla_{v_{2}}X_{s}^{x}\nabla_{v_{1}}X_{s}^{x}
+\frac{s}{t}\nabla^{2}B(X_{s}^{x})\nabla_{v_{3}}\nabla_{v_{2}}X_{s}^{x}\nabla_{v_{1}}X_{s}^{x}\\
&+\frac{s}{t}\nabla^{2}B(X_{s}^{x})\nabla_{v_{2}}X_{s}^{x}\nabla_{v_{3}}\nabla_{v_{1}}X_{s}^{x}
+\frac{1}{t}\nabla^{2}\sigma(X_{s}^{x})\nabla_{v_{3}}X_{s}^{x}\nabla_{v_{1}}X_{s}^{x}\sigma(X_{s}^{x})^{-1}\nabla_{v_{2}}X_{s}^{x}\\
&+\frac{1}{t}\nabla\sigma(X_{s}^{x})\big(\nabla_{v_{3}}\nabla_{v_{1}}X_{s}^{x}\sigma(X_{s}^{x})^{-1}\nabla_{v_{2}}X_{s}^{x}
+\nabla_{v_{1}}X_{s}^{x}\sigma(X_{s}^{x})^{-1}\nabla_{v_{3}}\nabla_{v_{2}}X_{s}^{x}\big)\\
&-\frac{1}{t}\nabla\sigma(X_{s}^{x})\nabla_{v_{1}}X_{s}^{x}\sigma(X_{s}^{x})^{-1}\nabla\sigma(X_{s}^{x})\nabla_{v_{3}}X_{s}^{x}
\sigma(X_{s}^{x})^{-1}\nabla_{v_{2}}X_{s}^{x}\big]\dif s\\
&+\eta^{\frac{1}{2}}\big[\nabla^{2}\sigma(X_{s}^{x})\nabla_{v_{3}}X_{s}^{x}D_{U_{2}}\nabla_{v_{1}}X_{s}^{x}
+\nabla\sigma(X_{s}^{x})\nabla_{v_{3}}D_{U_{2}}\nabla_{v_{1}}X_{s}^{x}\\
&\qquad\quad+\frac{s}{t}\nabla^{3}\sigma(X_{s}^{x})\nabla_{v_{3}}X_{s}^{x}\nabla_{v_{2}}X_{s}^{x}\nabla_{v_{1}}X_{s}^{x}
+\frac{s}{t}\nabla^{2}\sigma(X_{s}^{x})\nabla_{v_{3}}\nabla_{v_{2}}X_{s}^{x}\nabla_{v_{1}}X_{s}^{x}\\
&\qquad\quad+\frac{s}{t}\nabla^{2}\sigma(X_{s}^{x})\nabla_{v_{2}}X_{s}^{x}\nabla_{v_{3}}\nabla_{v_{1}}X_{s}^{x}\big]\dif B_{s}
\end{align*}
with $\nabla_{v_{3}}D_{U_{2}}\nabla_{v_{1}}X_{0}^{x}=0$.

Then, we have the following upper bounds on Malliavin derivatives.

\begin{lemma}
Let $v_{i}\in\mathbb{R}^{d}$ for $i=1,2,3,$ and let
\begin{align*}
U_{i,s}=\int_{0}^{s}u_{i}(r)\dif r, \quad 0\leq s\leq t,
\end{align*}
where $u_{i}(r)=\frac{1}{t}\eta^{-\frac{1}{2}}\sigma(X_{r}^{x})^{-1}\nabla_{v_{i}}X_{r}^{x}$ for $0\leq r\leq t.$ Then, for all $\eta\in(0,\delta]$ and $t\in(0,1]$, we have
\begin{align}\label{malthird}
\mathbb{E}|\nabla_{v_{3}}D_{U_{2}}\nabla_{v_{1}}X_{s}^{x}|^{2}\leq C_{A,L,d}(1+\frac{1}{t})|v_{1}|^{2}|v_{2}|^{2}|v_{3}|^{2}.
\end{align}
\end{lemma}
\begin{proof}
Writing $\tau_{1}(s)=\nabla_{v_{3}}D_{U_{2}}\nabla_{v_{1}}X_{s}^{x},$ by It\^{o}'s formula, we have
\begin{align*}
&\frac{\dif}{\dif r}\mathbb{E}|\tau_{1}(r)|^{2}\\
=&2\mathbb{E}\langle\nabla^{2} B(X_{r}^{x})\nabla_{v_{3}}X_{r}^{x}D_{U_{2}}\nabla_{v_{1}}X_{r}^{x}+\nabla B(X_{r}^{x})\tau_{1}(r)\\
&\qquad\quad+\frac{r}{t}\nabla^{3}B(X_{r})\nabla_{v_{3}}X_{r}^{x}\nabla_{v_{2}}X_{r}^{x}\nabla_{v_{1}}X_{r}^{x}
+\frac{r}{t}\nabla^{2}B(X_{r})\nabla_{v_{3}}\nabla_{v_{2}}X_{r}^{x}\nabla_{v_{1}}X_{r}^{x}\\
&\qquad\quad+\frac{r}{t}\nabla^{2}B(X_{r})\nabla_{v_{2}}X_{r}^{x}\nabla_{v_{3}}\nabla_{v_{1}}X_{r}^{x}
+\frac{1}{t}\nabla^{2}\sigma(X_{r}^{x})\nabla_{v_{3}}X_{r}^{x}\nabla_{v_{1}}X_{r}^{x}\sigma(X_{r}^{x})^{-1}\nabla_{v_{2}}X_{r}^{x}\\
&\qquad\quad+\frac{1}{t}\nabla\sigma(X_{r}^{x})\big(\nabla_{v_{3}}\nabla_{v_{1}}X_{r}^{x}\sigma(X_{r}^{x})^{-1}\nabla_{v_{2}}X_{r}^{x}
+\nabla_{v_{1}}X_{r}^{x}\sigma(X_{r}^{x})^{-1}\nabla_{v_{3}}\nabla_{v_{2}}X_{r}^{x}\big)\\
&\qquad\quad-\frac{1}{t}\nabla\sigma(X_{r}^{x})\nabla_{v_{1}}X_{r}^{x}\sigma(X_{r}^{x})^{-1}\nabla\sigma(X_{r}^{x})\nabla_{v_{3}}X_{r}^{x}
\sigma(X_{r}^{x})^{-1}\nabla_{v_{2}}X_{r}^{x},\tau_{1}(r)\rangle\\
&+\eta\mathbb{E}\|\nabla^{2}\sigma(X_{r}^{x})\nabla_{v_{3}}X_{r}^{x}D_{U_{2}}\nabla_{v_{1}}X_{r}^{x}
+\nabla\sigma(X_{r}^{x})\tau_{1}(r)+\frac{r}{t}\nabla^{2}\sigma(X_{r}^{x})\nabla_{v_{2}}X_{r}^{x}\nabla_{v_{3}}\nabla_{v_{1}}X_{r}^{x}\\
&\qquad\qquad+\frac{r}{t}\nabla^{3}\sigma(X_{r}^{x})\nabla_{v_{3}}X_{r}^{x}\nabla_{v_{2}}X_{r}^{x}\nabla_{v_{1}}X_{r}^{x}
+\frac{r}{t}\nabla^{2}\sigma(X_{r}^{x})\nabla_{v_{3}}\nabla_{v_{2}}X_{r}^{x}\nabla_{v_{1}}X_{r}^{x}\|_{{\rm HS}}^{2}.
\end{align*}
It follows from {\bf Assumption A} and the Cauchy-Schwarz inequality that
\begin{align*}
&\frac{\dif}{\dif r}\mathbb{E}|\tau_{1}(r)|^{2}\\
\leq&2L\mathbb{E}|\tau_{1}(r)|^{2}+2A_{1}\mathbb{E}\big[|\nabla_{v_{3}}X_{r}^{x}||D_{U_{2}}\nabla_{v_{1}}X_{r}^{x}||\tau_{1}(r)|\big]\\
&+2\mathbb{E}\big[A_{2}|\nabla_{v_{3}}X_{r}^{x}||\nabla_{v_{2}}X_{r}^{x}||\nabla_{v_{1}}X_{r}^{x}||\tau_{1}(r)|
+A_{1}|\nabla_{v_{3}}\nabla_{v_{2}}X_{r}^{x}||\nabla_{v_{1}}X_{r}^{x}||\tau_{1}(r)|\big]\\
&+2\mathbb{E}\big[\big(A_{1}|\nabla_{v_{3}}\nabla_{v_{1}}X_{r}^{x}|
+\frac{\sqrt{A_{4}}}{t}|\nabla_{v_{3}}X_{r}^{x}||\nabla_{v_{1}}X_{r}^{x}|\|\sigma(X_{r}^{x})^{-1}\|_{{\rm HS}}\big)|\nabla_{v_{2}}X_{r}^{x}||\tau_{1}(r)|\big]\\
&+2\sqrt{A_{3}}\frac{1}{t}\mathbb{E}\big[\big(|\nabla_{v_{3}}\nabla_{v_{1}}X_{r}^{x}||\nabla_{v_{2}}X_{r}^{x}|
+|\nabla_{v_{1}}X_{r}^{x}||\nabla_{v_{3}}\nabla_{v_{2}}X_{r}^{x}|\big)\|\sigma(X_{r}^{x})^{-1}\|_{{\rm HS}}|\tau_{1}(r)|\big]\\
&+2A_{3}\frac{1}{t}\mathbb{E}\big[|\nabla_{v_{1}}X_{r}^{x}|\|\sigma(X_{r}^{x})^{-1}\|_{{\rm HS}}^{2}|\nabla_{v_{3}}X_{r}^{x}||\nabla_{v_{2}}X_{r}^{x}||\tau_{1}(r)|\big]\\
&+5A_{4}\eta\mathbb{E}\big[|\nabla_{v_{3}}X_{r}^{x}|^{2}|D_{U_{2}}\nabla_{v_{1}}X_{r}^{x}|^{2}
+|\nabla_{v_{2}}X_{r}^{x}|^{2}|\nabla_{v_{3}}\nabla_{v_{1}}X_{r}^{x}|^{2}+|\nabla_{v_{3}}\nabla_{v_{2}}X_{r}^{x}|^{2}|\nabla_{v_{1}}X_{r}^{x}|^{2}\big]\\
&+5\eta\mathbb{E}\big[A_{3}|\tau_{1}(r)|^{2}+A_{5}|\nabla_{v_{3}}X_{r}^{x}|^{2}|\nabla_{v_{2}}X_{r}^{x}|^{2}|\nabla_{v_{1}}X_{r}^{x}|^{2}\big].
\end{align*}
Then, by (\ref{A1}) and Young's inequality, we have
\begin{align*}
&\frac{\dif}{\dif r}\mathbb{E}|\tau_{1}(r)|^{2}\\
\leq&(2L+3A_{1}+A_{2}+8A_{3}+A_{4})\mathbb{E}|\tau_{1}(r)|^{2}\\
&+\left(A_{2}+\frac{1}{t^{2}}\frac{\eta d}{\delta}+\frac{A_{3}}{t^{2}}\frac{\eta^{2}d^{2}}{\delta^{2}}+5A_{5}\right)\mathbb{E}[|\nabla_{v_{1}}X_{r}^{x}|^{2}|\nabla_{v_{2}}X_{r}^{x}|^{2}|\nabla_{v_{3}}X_{r}^{x}|^{2}]\\
&+A_{1}\mathbb{E}\left[|\nabla_{v_{3}}X_{r}^{x}|^{2}|D_{U_{2}}\nabla_{v_{1}}X_{r}^{x}|^{2}+|\nabla_{v_{3}}\nabla_{v_{2}}X_{r}^{x}|^{2}|\nabla_{v_{1}}X_{r}^{x}|^{2}+|\nabla_{v_{3}}\nabla_{v_{1}}X_{r}^{x}|^{2}|\nabla_{v_{2}}X_{r}^{x}|^{2}\right]\\
&+\frac{1}{t^{2}}\frac{\eta d}{\delta}\mathbb{E}\big[|\nabla_{v_{3}}\nabla_{v_{1}}X_{r}^{x}|^{2}|\nabla_{v_{2}}X_{r}^{x}|^{2}
+|\nabla_{v_{1}}X_{r}^{x}|^{2}|\nabla_{v_{3}}\nabla_{v_{2}}X_{r}^{x}|^{2}\big]\\
&+5A_{4}\eta\mathbb{E}\big[|\nabla_{v_{3}}X_{r}^{x}|^{2}|D_{U_{2}}\nabla_{v_{1}}X_{r}^{x}|^{2}
+|\nabla_{v_{2}}X_{r}^{x}|^{2}|\nabla_{v_{3}}\nabla_{v_{1}}X_{r}^{x}|^{2}+|\nabla_{v_{3}}\nabla_{v_{2}}X_{r}^{x}|^{2}|\nabla_{v_{1}}X_{r}^{x}|^{2}\big],
\end{align*}
By the Cauchy-Schwarz inequality, (\ref{grad3}), (\ref{malgrad2}) and (\ref{gradsec}), we further have
\begin{align*}
\frac{\dif}{\dif r}\mathbb{E}|\tau_{1}(r)|^{2}\leq&(2L+3A_{1}+A_{2}+8A_{3}+A_{4})\mathbb{E}|\tau_{1}(r)|^{2}+C_{A,L,d}(1+\frac{1}{t^{2}})|v_{1}|^{2}|v_{2}|^{2}|v_{3}|^{2}.
\end{align*}
This inequality, together with $\tau_{1}(0)=0$ and Gronwall inequality, implies
\begin{align*}
\mathbb{E}|\tau_{1}(s)|^{2}\leq&C_{A,L,d}(1+\frac{1}{t})|v_{1}|^{2}|v_{2}|^{2}|v_{3}|^{2}.
\end{align*}
\end{proof}

Let $v_{1},v_{2},v_{3}\in\mathbb{R}^{d},$ and define $u_{i}$ and $U_{i}$ as (\ref{malliavin1}) and (\ref{bismutf}), respectively, for $i=1,2,3.$ From (\ref{secequation}), we can similarly define $D_{U_{3}}\nabla_{v_{2}}\nabla_{v_{1}}X_{s}^{x},$ which satisfies the following equation: for $s\in[0,t],$
\begin{align*}
&\dif D_{U_{3}}\nabla_{v_{2}}\nabla_{v_{1}}X_{s}^{x}\\
=&\big[\nabla B(X_{s}^{x})D_{U_{3}}\nabla_{v_{2}}\nabla_{v_{1}}X_{s}^{x}+\nabla^{3}B(X_{s}^{x})D_{U_{3}}X_{s}^{x}\nabla_{v_{2}}X_{s}^{x}\nabla_{v_{1}}X_{s}^{x}\big]\dif s\\
&+\nabla^{2}B(X_{s}^{x})\big[D_{U_{3}}X_{s}^{x}\nabla_{v_{2}}\nabla_{v_{1}}X_{s}^{x}
+D_{U_{3}}\nabla_{v_{2}}X_{s}^{x}\nabla_{v_{1}}X_{s}^{x}+\nabla_{v_{2}}X_{s}^{x}D_{U_{3}}\nabla_{v_{1}}X_{s}^{x}\big]\dif s\\
&+\eta^{\frac{1}{2}}\big[\nabla\sigma(X_{s}^{x})\nabla_{v_{2}}\nabla_{v_{1}}X_{s}^{x}
+\nabla^{2}\sigma(X_{s}^{x})\nabla_{v_{2}}X_{s}^{x}\nabla_{v_{1}}X_{s}^{x}\big]u_{3}(s)\dif s\\ &+\eta^{\frac{1}{2}}\big[\nabla\sigma(X_{s}^{x})D_{U_{3}}\nabla_{v_{2}}\nabla_{v_{1}}X_{s}^{x}
+\nabla^{3}\sigma(X_{s}^{x})D_{U_{3}}X_{s}^{x}\nabla_{v_{2}}X_{s}^{x}\nabla_{v_{1}}X_{s}^{x}\big]\dif B_{s}\\
&+\eta^{\frac{1}{2}}\nabla^{2}\sigma(X_{s}^{x})\big[D_{U_{3}}X_{s}^{x}\nabla_{v_{2}}\nabla_{v_{1}}X_{s}^{x}
+D_{U_{3}}\nabla_{v_{2}}X_{s}^{x}\nabla_{v_{1}}X_{s}^{x}+\nabla_{v_{2}}X_{s}^{x}D_{U_{3}}\nabla_{v_{1}}X_{s}^{x}\big]\dif B_{s}\\
=&\big[\nabla B(X_{s}^{x})D_{U_{3}}\nabla_{v_{2}}\nabla_{v_{1}}X_{s}^{x}+\nabla^{3}B(X_{s}^{x})\frac{s}{t}\nabla_{v_{3}}X_{s}^{x}\nabla_{v_{2}}X_{s}^{x}\nabla_{v_{1}}X_{s}^{x}\big]\dif s\\
&+\nabla^{2}B(X_{s}^{x})\big[\frac{s}{t}\nabla_{v_{3}}X_{s}^{x}\nabla_{v_{2}}\nabla_{v_{1}}X_{s}^{x}
+D_{U_{3}}\nabla_{v_{2}}X_{s}^{x}\nabla_{v_{1}}X_{s}^{x}+\nabla_{v_{2}}X_{s}^{x}D_{U_{3}}\nabla_{v_{1}}X_{s}^{x}\big]\dif s\\
&+\big[\nabla\sigma(X_{s}^{x})\nabla_{v_{2}}\nabla_{v_{1}}X_{s}^{x}
+\nabla^{2}\sigma(X_{s}^{x})\nabla_{v_{2}}X_{s}^{x}\nabla_{v_{1}}X_{s}^{x}\big]\frac{1}{t}\sigma(X_{s}^{x})^{-1}\nabla_{v_{3}}X_{s}^{x}\dif s\\ &+\eta^{\frac{1}{2}}\big[\nabla\sigma(X_{s}^{x})D_{U_{3}}\nabla_{v_{2}}\nabla_{v_{1}}X_{s}^{x}
+\nabla^{3}\sigma(X_{s}^{x})\frac{s}{t}\nabla_{v_{3}}X_{s}^{x}\nabla_{v_{2}}X_{s}^{x}\nabla_{v_{1}}X_{s}^{x}\big]\dif B_{s}\\
&+\eta^{\frac{1}{2}}\nabla^{2}\sigma(X_{s}^{x})\big[\frac{s}{t}\nabla_{v_{3}}X_{s}^{x}\nabla_{v_{2}}\nabla_{v_{1}}X_{s}^{x}
+D_{U_{3}}\nabla_{v_{2}}X_{s}^{x}\nabla_{v_{1}}X_{s}^{x}+\nabla_{v_{2}}X_{s}^{x}D_{U_{3}}\nabla_{v_{1}}X_{s}^{x}\big]\dif B_{s}
\end{align*}
with $D_{U_{3}}\nabla_{v_{2}}\nabla_{v_{1}}X_{0}^{x}=0$, where the second equality is by (\ref{malliavin1}) and (\ref{ma1}).

Then, we have the following upper bounds on Malliavin derivatives.

\begin{lemma}
Let $v_{i}\in\mathbb{R}^{d}$ for $i=1,2,3,$ and let
\begin{align*}
U_{i,s}=\int_{0}^{s}u_{i}(r)\dif r, \quad 0\leq s\leq t,
\end{align*}
where $u_{i}(r)=\frac{1}{t}\sigma(X_{r}^{x})^{-1}\nabla_{v_{i}}X_{r}^{x}$ for $0\leq r\leq t.$ Then, for all $\eta\in(0,\delta]$ and $t\in(0,1]$, we have
\begin{align}\label{malthird2}
\mathbb{E}|D_{U_{3}}\nabla_{v_{2}}\nabla_{v_{1}}X_{s}^{x}|^{2}\leq C_{A,L,d}(1+\frac{1}{t})|v_{1}|^{2}|v_{2}|^{2}|v_{3}|^{3}.
\end{align}
\end{lemma}
\begin{proof}
Writing $\tau_{2}(s)=D_{U_{3}}\nabla_{v_{2}}\nabla_{v_{1}}X_{s}^{x},$ by It\^{o}'s formula, we have
\begin{align*}
&\frac{\dif}{\dif r}\mathbb{E}|\tau_{2}(r)|^{2}\\
=&2\mathbb{E}\langle\nabla B(X_{r}^{x})\tau_{2}(r)+\nabla^{3}B(X_{r}^{x})\frac{r}{t}\nabla_{v_{3}}X_{r}^{x}\nabla_{v_{2}}X_{r}^{x}\nabla_{v_{1}}X_{r}^{x},
\tau_{2}(r)\rangle\\
&+2\mathbb{E}\langle\nabla^{2}B(X_{r}^{x})\big[\frac{r}{t}\nabla_{v_{3}}X_{r}^{x}\nabla_{v_{2}}\nabla_{v_{1}}X_{r}^{x}
+D_{U_{3}}\nabla_{v_{2}}X_{r}^{x}\nabla_{v_{1}}X_{r}^{x}+\nabla_{v_{2}}X_{r}^{x}D_{U_{3}}\nabla_{v_{1}}X_{r}^{x}\big],\tau_{2}(r)\rangle\\
&+2\mathbb{E}\langle\big[\nabla\sigma(X_{r}^{x})\nabla_{v_{2}}\nabla_{v_{1}}X_{r}^{x}
+\nabla^{2}\sigma(X_{r}^{x})\nabla_{v_{2}}X_{r}^{x}\nabla_{v_{1}}X_{r}^{x}\big]\frac{1}{t}\sigma(X_{r}^{x})^{-1}\nabla_{v_{3}}X_{r}^{x},
\tau_{2}(r)\rangle\\
&+\eta\mathbb{E}\|\nabla\sigma(X_{r}^{x})\tau_{2}(r)
+\nabla^{3}\sigma(X_{r}^{x})\frac{r}{t}\nabla_{v_{3}}X_{r}^{x}\nabla_{v_{2}}X_{r}^{x}\nabla_{v_{1}}X_{r}^{x}\\
&\qquad\qquad+\nabla^{2}\sigma(X_{r}^{x})\big[\frac{r}{t}\nabla_{v_{3}}X_{r}^{x}\nabla_{v_{2}}\nabla_{v_{1}}X_{r}^{x}
+D_{U_{3}}\nabla_{v_{2}}X_{r}^{x}\nabla_{v_{1}}X_{r}^{x}+\nabla_{v_{2}}X_{r}^{x}D_{U_{3}}\nabla_{v_{1}}X_{r}^{x}\big]\|_{{\rm HS}}^{2}.
\end{align*}
The Cauchy-Schwarz inequality and {\bf Assumption A} imply
\begin{align*}
&\frac{\dif}{\dif r}\mathbb{E}|\tau_{2}(r)|^{2}\\
\leq&2L\mathbb{E}|\tau_{2}(r)|^{2}+2A_{2}\mathbb{E}\big[|\nabla_{v_{3}}X_{r}^{x}||\nabla_{v_{2}}X_{r}^{x}||\nabla_{v_{1}}X_{r}^{x}||\tau_{2}(r)|\big]\\
&+2A_{1}\mathbb{E}\big[\big(|\nabla_{v_{3}}X_{r}^{x}||\nabla_{v_{2}}\nabla_{v_{1}}X_{r}^{x}|
+|D_{U_{3}}\nabla_{v_{2}}X_{r}^{x}||\nabla_{v_{1}}X_{r}^{x}|+|\nabla_{v_{2}}X_{r}^{x}||D_{U_{3}}\nabla_{v_{1}}X_{r}^{x}|\big)|\tau_{2}(r)|\big]\\
&+2\mathbb{E}\big[\big(\sqrt{A_{3}}|\nabla_{v_{2}}\nabla_{v_{1}}X_{r}^{x}|
+\sqrt{A_{4}}|\nabla_{v_{2}}X_{r}^{x}||\nabla_{v_{1}}X_{r}^{x}|\big)\frac{1}{t}\|\sigma(X_{r}^{x})^{-1}\|_{{\rm HS}}|\nabla_{v_{3}}X_{r}^{x}||\tau_{2}(r)|\big]\\
&+5\eta\mathbb{E}\big[A_{3}|\tau_{2}(r)|^{2}
+A_{5}|\nabla_{v_{3}}X_{r}^{x}|^{2}|\nabla_{v_{2}}X_{r}^{x}|^{2}|\nabla_{v_{1}}X_{r}^{x}|^{2}\big]\\
&+5A_{4}\eta\mathbb{E}\big[|\nabla_{v_{3}}X_{r}^{x}|^{2}|\nabla_{v_{2}}\nabla_{v_{1}}X_{r}^{x}|^{2}
+|D_{U_{3}}\nabla_{v_{2}}X_{r}^{x}|^{2}|\nabla_{v_{1}}X_{r}^{x}|^{2}+|\nabla_{v_{2}}X_{r}^{x}|^{2}|D_{U_{3}}\nabla_{v_{1}}X_{r}^{x}|^{2}\big].
\end{align*}
Then, by (\ref{A1}) and Young's inequality, we have
\begin{align*}
&\frac{\dif}{\dif r}\mathbb{E}|\tau_{2}(r)|^{2}\\
\leq&(2L+3A_{1}+A_{2}+6A_{3}+A_{4})\mathbb{E}|\tau_{2}(r)|^{2}+(A_{2}+5A_{5})\mathbb{E}[|\nabla_{v_{3}}X_{r}^{x}|^{2}|\nabla_{v_{2}}X_{r}^{x}|^{2}|\nabla_{v_{1}}X_{r}^{x}|^{2}]\\
&+A_{1}\mathbb{E}\big[|\nabla_{v_{3}}X_{r}^{x}|^{2}|\nabla_{v_{2}}\nabla_{v_{1}}X_{r}^{x}|^{2}
+|D_{U_{3}}\nabla_{v_{2}}X_{r}^{x}|^{2}|\nabla_{v_{1}}X_{r}^{x}|^{2}+|\nabla_{v_{2}}X_{r}^{x}|^{2}|D_{U_{3}}\nabla_{v_{1}}X_{r}^{x}|^{2}\big]\\
&+\mathbb{E}\big[\big(|\nabla_{v_{2}}\nabla_{v_{1}}X_{r}^{x}|^{2}
+|\nabla_{v_{2}}X_{r}^{x}|^{2}|\nabla_{v_{1}}X_{r}^{x}|^{2}\big)\frac{1}{t^{2}}\frac{\eta d}{\delta}|\nabla_{v_{3}}X_{r}^{x}|^{2}\big]\\
&+5A_{4}\eta\mathbb{E}\big[|\nabla_{v_{3}}X_{r}^{x}|^{2}|\nabla_{v_{2}}\nabla_{v_{1}}X_{r}^{x}|^{2}
+|D_{U_{3}}\nabla_{v_{2}}X_{r}^{x}|^{2}|\nabla_{v_{1}}X_{r}^{x}|^{2}+|\nabla_{v_{2}}X_{r}^{x}|^{2}|D_{U_{3}}\nabla_{v_{1}}X_{r}^{x}|^{2}\big].
\end{align*}
It follows the Cauchy-Schwarz inequality, (\ref{grad3}), (\ref{gradsec}) and (\ref{malgrad2}) that
\begin{align*}
\frac{\dif}{\dif r}\mathbb{E}|\tau_{2}(r)|^{2}
\leq&(2L+3A_{1}+A_{2}+6A_{3}+A_{4})\mathbb{E}|\tau_{2}(r)|^{2}+C_{A,L,d}(1+\frac{1}{t^{2}})|v_{1}|^{2}|v_{2}|^{2}|v_{3}|^{3}.
\end{align*}
This inequality, together with $\tau_{2}(0)=0$ and Gronwall inequality, implies
\begin{align*}
\mathbb{E}|\tau_{2}(s)|^{2}\leq&C_{A,L,d}(1+\frac{1}{t})|v_{1}|^{2}|v_{2}|^{2}|v_{3}|^{3}.
\end{align*}
\end{proof}

Let $v_{1},v_{2},v_{3}\in\mathbb{R}^{d},$ and define $u_{i}$ and $U_{i}$ as (\ref{malliavin1}) and (\ref{bismutf}), respectively, for $i=1,2,3.$ From (\ref{malsec}), we can similarly define $D_{U_{3}}D_{U_{2}}\nabla_{v_{1}}X_{s}^{x},$ which satisfies the following equation: for $s\in[0,t],$
\begin{align*}
&\dif D_{U_{3}}D_{U_{2}}\nabla_{v_{1}}X_{s}^{x}\\
=&\big[\frac{s}{t}\nabla^{3}B(X_{s}^{x})D_{U_{3}}X_{s}^{x}\nabla_{v_{2}}X_{s}^{x}\nabla_{v_{1}}X_{s}^{x}
+\nabla B(X_{s}^{x})D_{U_{3}}D_{U_{2}}\nabla_{v_{1}}X_{s}^{x}\big]\dif s\\
&+\nabla^{2}B(X_{s}^{x})\big[D_{U_{3}}X_{s}^{x}D_{U_{2}}\nabla_{v_{1}}X_{s}^{x}
+\frac{s}{t}D_{U_{3}}\nabla_{v_{2}}X_{s}^{x}\nabla_{v_{1}}X_{s}^{x}
+\frac{s}{t}\nabla_{v_{2}}X_{s}^{x}D_{U_{3}}\nabla_{v_{1}}X_{s}^{x}\big]\dif s\\
&+\frac{1}{t}\big[\nabla^{2}\sigma(X_{s}^{x})D_{U_{3}}X_{s}^{x}\nabla_{v_{1}}X_{s}^{x}\sigma(X_{s}^{x})^{-1}\nabla_{v_{2}}X_{s}^{x}
+\nabla\sigma(X_{s}^{x})D_{U_{3}}\nabla_{v_{1}}X_{s}^{x}\sigma(X_{s}^{x})^{-1}\nabla_{v_{2}}X_{s}^{x}\big]\dif s\\
&+\frac{1}{t}\nabla\sigma(X_{s}^{x})\nabla_{v_{1}}X_{s}^{x}\sigma(X_{s}^{x})^{-1}D_{U_{3}}\nabla_{v_{2}}X_{s}^{x}\dif s\\
&-\frac{1}{t}\nabla\sigma(X_{s}^{x})\nabla_{v_{1}}X_{s}^{x}\sigma(X_{s}^{x})^{-1}\nabla\sigma(X_{s}^{x})D_{U_{3}}X_{s}^{x}\sigma(X_{s}^{x})^{-1}\nabla_{v_{2}}X_{s}^{x}\dif s\\
&+\eta^{\frac{1}{2}}\big[\nabla\sigma(X_{s}^{x})D_{U_{2}}\nabla_{v_{1}}X_{s}^{x}+\frac{s}{t}\nabla^{2}\sigma(X_{s}^{x})
\nabla_{v_{2}}X_{s}^{x}\nabla_{v_{1}}X_{s}^{x}\big]u_{3}(s)\dif s\\
&+\eta^{\frac{1}{2}}\nabla^{2}\sigma(X_{s}^{x})\big[D_{U_{3}}X_{s}^{x}D_{U_{2}}\nabla_{v_{1}}X_{s}^{x}
+\frac{s}{t}D_{U_{3}}\nabla_{v_{2}}X_{s}^{x}\nabla_{v_{1}}X_{s}^{x}
+\frac{s}{t}\nabla_{v_{2}}X_{s}^{x}D_{U_{3}}\nabla_{v_{1}}X_{s}^{x}\big]\dif B_{s}\\
&+\eta^{\frac{1}{2}}\big[\frac{s}{t}\nabla^{3}\sigma(X_{s}^{x})D_{U_{3}}X_{s}^{x}\nabla_{v_{2}}X_{s}^{x}\nabla_{v_{1}}X_{s}^{x}
+\nabla\sigma(X_{s}^{x})D_{U_{3}}D_{U_{2}}\nabla_{v_{1}}X_{s}^{x}\big]\dif B_{s}
\end{align*}
with $D_{U_{3}}D_{U_{2}}\nabla_{v_{1}}X_{0}^{x}=0$. Furthermore, by (\ref{malliavin1}) and (\ref{ma1}), we have
\begin{align*}
&\dif D_{U_{3}}D_{U_{2}}\nabla_{v_{1}}X_{s}^{x}\\
=&\big[\frac{s^{2}}{t^{2}}\nabla^{3}B(X_{s}^{x})\nabla_{v_{3}}X_{s}^{x}\nabla_{v_{2}}X_{s}^{x}\nabla_{v_{1}}X_{s}^{x}
+\nabla B(X_{s}^{x})D_{U_{3}}D_{U_{2}}\nabla_{v_{1}}X_{s}^{x}\big]\dif s\\
&+\frac{s}{t}\nabla^{2}B(X_{s}^{x})\big[\nabla_{v_{3}}X_{s}^{x}D_{U_{2}}\nabla_{v_{1}}X_{s}^{x}
+D_{U_{3}}\nabla_{v_{2}}X_{s}^{x}\nabla_{v_{1}}X_{s}^{x}
+\nabla_{v_{2}}X_{s}^{x}D_{U_{3}}\nabla_{v_{1}}X_{s}^{x}\big]\dif s\\
&+\frac{1}{t}\big[\nabla^{2}\sigma(X_{s}^{x})\frac{s}{t}\nabla_{v_{3}}X_{s}^{x}\nabla_{v_{1}}X_{s}^{x}\sigma(X_{s}^{x})^{-1}\nabla_{v_{2}}X_{s}^{x}
+\nabla\sigma(X_{s}^{x})D_{U_{3}}\nabla_{v_{1}}X_{s}^{x}\sigma(X_{s}^{x})^{-1}\nabla_{v_{2}}X_{s}^{x}\big]\dif s\\
&+\frac{1}{t}\nabla\sigma(X_{s}^{x})\nabla_{v_{1}}X_{s}^{x}\sigma(X_{s}^{x})^{-1}D_{U_{3}}\nabla_{v_{2}}X_{s}^{x}\dif s\\
&-\frac{s}{t^{2}}\nabla\sigma(X_{s}^{x})\nabla_{v_{1}}X_{s}^{x}\sigma(X_{s}^{x})^{-1}\nabla\sigma(X_{s}^{x})\nabla_{v_{3}}X_{s}^{x}\sigma(X_{s}^{x})^{-1}\nabla_{v_{2}}X_{s}^{x}\dif s\\
&+\big[\nabla\sigma(X_{s}^{x})D_{U_{2}}\nabla_{v_{1}}X_{s}^{x}+\frac{s}{t}\nabla^{2}\sigma(X_{s}^{x})
\nabla_{v_{2}}X_{s}^{x}\nabla_{v_{1}}X_{s}^{x}\big]\frac{1}{t}\sigma(X_{s}^{x})^{-1}\nabla_{v_{3}}X_{s}^{x}\dif s\\
&+\eta^{\frac{1}{2}}\frac{s}{t}\nabla^{2}\sigma(X_{s}^{x})\big[\nabla_{v_{3}}X_{s}^{x}D_{U_{2}}\nabla_{v_{1}}X_{s}^{x}
+D_{U_{3}}\nabla_{v_{2}}X_{s}^{x}\nabla_{v_{1}}X_{s}^{x}
+\nabla_{v_{2}}X_{s}^{x}D_{U_{3}}\nabla_{v_{1}}X_{s}^{x}\big]\dif B_{s}\\
&+\eta^{\frac{1}{2}}\big[\frac{s^{2}}{t^{2}}\nabla^{3}\sigma(X_{s}^{x})\nabla_{v_{3}}X_{s}^{x}\nabla_{v_{2}}X_{s}^{x}\nabla_{v_{1}}X_{s}^{x}
+\nabla\sigma(X_{s}^{x})D_{U_{3}}D_{U_{2}}\nabla_{v_{1}}X_{s}^{x}\big]\dif B_{s}.
\end{align*}

Then, we have the following upper bounds on Malliavin derivatives.

\begin{lemma}
Let $v_{i}\in\mathbb{R}^{d}$ for $i=1,2,3,$ and let
\begin{align*}
U_{i,s}=\int_{0}^{s}u_{i}(r)\dif r, \quad 0\leq s\leq t,
\end{align*}
where $u_{i}(r)=\frac{1}{t}\sigma(X_{r}^{x})^{-1}\nabla_{v_{i}}X_{r}^{x}$ for $0\leq r\leq t.$ Then, for all $\eta\in(0,\delta]$ and $t\in(0,1]$, we have
\begin{align}\label{malthird3}
\mathbb{E}|D_{U_{3}}D_{U_{2}}\nabla_{v_{1}}X_{s}^{x}|^{2}\leq C_{A,L,d}(1+\frac{1}{t^\frac{5}{2}})|v_{1}|^{2}|v_{2}|^{2}|v_{3}|^{2}.
\end{align}
\end{lemma}
\begin{proof}
Writing $\tau_{3}(s)=D_{U_{3}}D_{U_{2}}\nabla_{v_{1}}X_{s}^{x},$ by It\^{o}'s formula, we have
\begin{align*}
&\frac{\dif}{\dif r}\mathbb{E}|\tau_{3}(r)|^{2}\\
=&2\mathbb{E}\langle\frac{r^{2}}{t^{2}}\nabla^{3}B(X_{r}^x)\nabla_{v_{3}}X_{r}^{x}\nabla_{v_{2}}X_{r}^{x}\nabla_{v_{1}}X_{r}^{x}
+\nabla B(X_{r}^{x})\tau_{3}(r),\tau_{3}(r)\rangle\\
&+2\mathbb{E}\langle\frac{r}{t}\nabla^{2}B(X_{r}^{x})\big[\nabla_{v_{3}}X_{r}^{x}D_{U_{2}}\nabla_{v_{1}}X_{r}^{x}
+D_{U_{3}}\nabla_{v_{2}}X_{r}^{x}\nabla_{v_{1}}X_{r}^{x}
+\nabla_{v_{2}}X_{r}^{x}D_{U_{3}}\nabla_{v_{1}}X_{r}^{x}\big],\tau_{3}(r)\rangle\\
&+2\mathbb{E}\langle\frac{1}{t}\big[\nabla^{2}\sigma(X_{r}^{x})\frac{r}{t}\nabla_{v_{3}}X_{r}^{x}\nabla_{v_{1}}X_{r}^{x}
+\nabla\sigma(X_{r}^{x})D_{U_{3}}\nabla_{v_{1}}X_{r}^{x}\big]\sigma(X_{r}^{x})^{-1}\nabla_{v_{2}}X_{r}^{x},\tau_{3}(r)\rangle\\
&+2\mathbb{E}\langle\frac{1}{t}\nabla\sigma(X_{r}^{x})\nabla_{v_{1}}X_{r}^{x}\sigma(X_{r}^{x})^{-1}D_{U_{3}}\nabla_{v_{2}}X_{r}^{x},
\tau_{3}(r)\rangle\\
&-2\mathbb{E}\langle\frac{r}{t^{2}}\nabla\sigma(X_{r}^{x})\nabla_{v_{1}}X_{r}^{x}\sigma(X_{r}^{x})^{-1}\nabla\sigma(X_{r}^{x})
\nabla_{v_{3}}X_{r}^{x}\sigma(X_{r}^{x})^{-1}\nabla_{v_{2}}X_{r}^{x},\tau_{3}(r)\rangle\\
&+2\mathbb{E}\langle\big[\nabla\sigma(X_{r}^{x})D_{U_{2}}\nabla_{v_{1}}X_{r}^{x}+\frac{r}{t}\nabla^{2}\sigma(X_{r}^{x})
\nabla_{v_{2}}X_{r}^{x}\nabla_{v_{1}}X_{r}^{x}\big]\frac{1}{t}\sigma(X_{r}^{x})^{-1}\nabla_{v_{3}}X_{r}^{x},\tau_{3}(r)\rangle\\
&+\eta\mathbb{E}\|\frac{r}{t}\nabla^{2}\sigma(X_{r}^{x})\big[\nabla_{v_{3}}X_{r}^{x}D_{U_{2}}\nabla_{v_{1}}X_{r}^{x}
+D_{U_{3}}\nabla_{v_{2}}X_{r}^{x}\nabla_{v_{1}}X_{r}^{x}
+\nabla_{v_{2}}X_{r}^{x}D_{U_{3}}\nabla_{v_{1}}X_{r}^{x}\big]\\
&\qquad\qquad\quad+\frac{r^{2}}{t^{2}}\nabla^{3}\sigma(X_{r}^{x})\nabla_{v_{3}}X_{r}^{x}\nabla_{v_{2}}X_{r}^{x}\nabla_{v_{1}}X_{r}^{x}
+\nabla\sigma(X_{r}^{x})\tau_{3}(r)\|_{{\rm HS}}^{2},
\end{align*}
Following  the Cauchy-Schwarz inequality and {\bf Assumption A }, one has
\begin{align*}
&\frac{\dif}{\dif r}\mathbb{E}|\tau_{3}(r)|^{2}\\
\leq&2L\mathbb{E}|\tau_{3}(r)|^{2}
+2A_{2}\mathbb{E}\big[|\nabla_{v_{3}}X_{r}^{x}||\nabla_{v_{2}}X_{r}^{x}||\nabla_{v_{1}}X_{r}^{x}||\tau_{3}(r)|\big]\\
&+2A_{1}\mathbb{E}\big[\big(|\nabla_{v_{3}}X_{r}^{x}||D_{U_{2}}\nabla_{v_{1}}X_{r}^{x}|
+|D_{U_{3}}\nabla_{v_{2}}X_{r}^{x}||\nabla_{v_{1}}X_{r}^{x}|
+|\nabla_{v_{2}}X_{r}^{x}||D_{U_{3}}\nabla_{v_{1}}X_{r}^{x}|\big)|\tau_{3}(r)|\big]\\
&+\frac{2}{t}\mathbb{E}\big[\big(\sqrt{A_{4}}|\nabla_{v_{3}}X_{r}^{x}||\nabla_{v_{1}}X_{r}^{x}|
+\sqrt{A_{3}}|D_{U_{3}}\nabla_{v_{1}}X_{r}^{x}|\big)\|\sigma(X_{r}^{x})^{-1}\|_{{\rm HS}}|\nabla_{v_{2}}X_{r}^{x}||\tau_{3}(r)|\big]\\
&+\frac{2}{t}\sqrt{A_{3}}\mathbb{E}\big[|\nabla_{v_{1}}X_{r}^{x}|\|\sigma(X_{r}^{x})^{-1}\|_{{\rm HS}}|D_{U_{3}}\nabla_{v_{2}}X_{r}^{x}||\tau_{3}(r)|\big]\\
&+\frac{2}{t}A_{3}\mathbb{E}\big[|\nabla_{v_{1}}X_{r}^{x}|\|\sigma(X_{r}^{x})^{-1}\|_{{\rm HS}}^{2}|\nabla_{v_{3}}X_{r}^{x}||\nabla_{v_{2}}X_{r}^{x}||\tau_{3}(r)|\big]\\
&+\frac{2}{t}\mathbb{E}\big[\big(\sqrt{A_{3}}|D_{U_{2}}\nabla_{v_{1}}X_{r}^{x}|+\sqrt{A_{4}}|\nabla_{v_{2}}X_{r}^{x}||\nabla_{v_{1}}X_{r}^{x}|\big)
\|\sigma(X_{r}^{x})^{-1}\|_{{\rm HS}}|\nabla_{v_{3}}X_{r}^{x}||\tau_{3}(r)|\big]\\
&+5A_{4}\eta\mathbb{E}\big[|\nabla_{v_{3}}X_{r}^{x}|^{2}|D_{U_{2}}\nabla_{v_{1}}X_{r}^{x}|^{2}
+|D_{U_{3}}\nabla_{v_{2}}X_{r}^{x}|^{2}|\nabla_{v_{1}}X_{r}^{x}|^{2}
+|\nabla_{v_{2}}X_{r}^{x}|^{2}|D_{U_{3}}\nabla_{v_{1}}X_{r}^{x}|^{2}\big]\\
&+5\eta\mathbb{E}\big[A_{5}|\nabla_{v_{3}}X_{r}^{x}|^{2}|\nabla_{v_{2}}X_{r}^{x}|^{2}|\nabla_{v_{1}}X_{r}^{x}|^{2}
+A_{3}|\tau_{3}(r)|^{2}\big].
\end{align*}
By (\ref{A1}) and Young's inequality, we have
\begin{align*}
&\frac{\dif}{\dif r}\mathbb{E}|\tau_{3}(r)|^{2}\\
\leq&(2L+3A_{1}+A_{2}+9A_{3}+2A_{4})\mathbb{E}|\tau_{3}(r)|^{2}
+(A_{2}+5A_{5})\mathbb{E}[|\nabla_{v_{1}}X_{r}^{x}|^{2}|\nabla_{v_{2}}X_{r}^{x}|^{2}|\nabla_{v_{3}}X_{r}^{x}|^{2}]\\
&+A_{1}\mathbb{E}\big[|\nabla_{v_{3}}X_{r}^{x}|^{2}|D_{U_{2}}\nabla_{v_{1}}X_{r}^{x}|^{2}
+|D_{U_{3}}\nabla_{v_{2}}X_{r}^{x}|^{2}|\nabla_{v_{1}}X_{r}^{x}|^{2}
+|\nabla_{v_{2}}X_{r}^{x}|^{2}|D_{U_{3}}\nabla_{v_{1}}X_{r}^{x}|^{2}\big]\\
&+\frac{1}{t^{2}}\mathbb{E}\big[\big(|\nabla_{v_{3}}X_{r}^{x}|^{2}|\nabla_{v_{1}}X_{r}^{x}|^{2}
+|D_{U_{3}}\nabla_{v_{1}}X_{r}^{x}|^{2}\big)\frac{\eta d}{\delta}|\nabla_{v_{2}}X_{r}^{x}|^{2}\big]\\
&+\frac{1}{t^{2}}\mathbb{E}\big[|\nabla_{v_{1}}X_{r}^{x}|^{2}|D_{U_{3}}\nabla_{v_{2}}X_{r}^{x}|^{2}\frac{\eta d}{\delta}+A_{3}\frac{\eta^{2}d^{2}}{\delta^{2}}|\nabla_{v_{1}}X_{r}^{x}|^{2}|\nabla_{v_{2}}X_{r}^{x}|^{2}|\nabla_{v_{3}}X_{r}^{x}|^{2}\big]\\
&+\frac{1}{t^{2}}\mathbb{E}\big[\big(|D_{U_{2}}\nabla_{v_{1}}X_{r}^{x}|^{2}+|\nabla_{v_{2}}X_{r}^{x}|^{2}|\nabla_{v_{1}}X_{r}^{x}|^{2}\big)
|\nabla_{v_{3}}X_{r}^{x}|^{2}\frac{\eta d}{\delta}\big]\\
&+5A_{4}\eta\mathbb{E}\big[|\nabla_{v_{3}}X_{r}^{x}|^{2}|D_{U_{2}}\nabla_{v_{1}}X_{r}^{x}|^{2}
+|D_{U_{3}}\nabla_{v_{2}}X_{r}^{x}|^{2}|\nabla_{v_{1}}X_{r}^{x}|^{2}
+|\nabla_{v_{2}}X_{r}^{x}|^{2}|D_{U_{3}}\nabla_{v_{1}}X_{r}^{x}|^{2}\big].
\end{align*}
Then, by the Cauchy-Schwarz inequality, (\ref{grad3}) and (\ref{malgrad2}), we have
\begin{align*}
\frac{\dif}{\dif r}\mathbb{E}|\tau_{3}(r)|^{2}
\leq&(2L+3A_{1}+A_{2}+9A_{3}+2A_{4})\mathbb{E}|\tau_{3}(r)|^{2}+C_{A,L,d}(1+\frac{1}{t^{\frac{7}{2}}})|v_{1}|^{2}|v_{2}|^{2}|v_{3}|^{2}.
\end{align*}
This inequality, together with $\tau_{3}(0)=0$ and Gronwall inequality, implies
\begin{align*}
\mathbb{E}|\tau_{3}(s)|^{2}\leq&C_{A,L,d}(1+\frac{1}{t^\frac{5}{2}})|v_{1}|^{2}|v_{2}|^{2}|v_{3}|^{2}.
\end{align*}
\end{proof}

With the above results, we have the following estimates:

\begin{lemma}
Let $v_{1},v_{2},v_{3}\in\mathbb{R}^{d}$ and $x\in\mathbb{R}^{d}.$ Then, as $\eta\in(0,\delta]$ and $t\in(0,1]$, we have
\begin{align}\label{third1}
\mathbb{E}|\mathcal{R}_{v_{1},v_{2}}^{x}(t)|^{2}\leq C_{A,L,d}(1+\frac{1}{t})|v_{1}|^{2}|v_{2}|^{2},
\end{align}
\begin{align}\label{third2}
\mathbb{E}|\nabla_{v_{3}}\mathcal{R}_{v_{1},v_{2}}^{x}(t)|^{2}
\leq C_{A,L,d}(1+\frac{1}{t})|v_{1}|^{2}|v_{2}|^{2}|v_{3}|^{2}
\end{align}
and
\begin{align}\label{third3}
\mathbb{E}|D_{U_{3}}\mathcal{R}_{v_{1},v_{2}}^{x}(t)|^{2}
\leq C_{A,L,d}(1+\frac{1}{t^\frac{5}{2}})|v_{1}|^{2}|v_{2}|^{2}|v_{3}|^{2}.
\end{align}
\end{lemma}
\begin{proof}
By the Cauchy-Schwarz inequality, (\ref{gradsec}) and (\ref{malgrad1}), we have
\begin{align*}
\mathbb{E}|\mathcal{R}_{v_{1},v_{2}}^{x}(t)|^{2}=&\mathbb{E}|\nabla_{v_{2}}\nabla_{v_{1}}X_{t}^{x}-D_{U_{2}}\nabla_{v_{1}}X_{t}^{x}|^{2}\\
\leq&2\mathbb{E}|\nabla_{v_{2}}\nabla_{v_{1}}X_{t}^{x}|^{2}+2\mathbb{E}|D_{U_{2}}\nabla_{v_{1}}X_{t}^{x}|^{2}\\
\leq&2\sqrt{\mathbb{E}|\nabla_{v_{2}}\nabla_{v_{1}}X_{t}^{x}|^{4}}+2\mathbb{E}|D_{U_{2}}\nabla_{v_{1}}X_{t}^{x}|^{2}\\
\leq&C_{A,L,d}(1+\frac{1}{t})|v_{1}|^{2}|v_{2}|^{2}.
\end{align*}

Noticing that $t\leq\frac{4}{\gamma}e^{\frac{\gamma}{4}t},$ by (\ref{gradthird}) and (\ref{malthird}), we have
\begin{align*}
\mathbb{E}|\nabla_{v_{3}}\mathcal{R}_{v_{1},v_{2}}^{x}(t)|^{2}=&\mathbb{E}|\nabla_{v_{3}}\nabla_{v_{2}}\nabla_{v_{1}}X_{t}^{x}-\nabla_{v_{3}}D_{U_{2}}\nabla_{v_{1}}X_{t}^{x}|^{2}\\
\leq&2\mathbb{E}|\nabla_{v_{3}}\nabla_{v_{2}}\nabla_{v_{1}}X_{t}^{x}|^{2}+2\mathbb{E}|\nabla_{v_{3}}D_{U_{2}}\nabla_{v_{1}}X_{t}^{x}|^{2}\\
\leq&C_{A,L,d}(1+\frac{1}{t})|v_{1}|^{2}|v_{2}|^{2}|v_{3}|^{2}.
\end{align*}

By (\ref{malthird2}) and (\ref{malthird3}), we have
\begin{align*}
\mathbb{E}|D_{U_{3}}\mathcal{R}_{v_{1},v_{2}}^{x}(t)|^{2}=&\mathbb{E}|D_{U_{3}}\nabla_{v_{2}}\nabla_{v_{1}}X_{t}^{x}-D_{U_{3}}D_{U_{2}}\nabla_{v_{1}}X_{t}^{x}|^{2}\\
\leq&2\mathbb{E}|D_{U_{3}}\nabla_{v_{2}}\nabla_{v_{1}}X_{t}^{x}|^{2}+2\mathbb{E}|D_{U_{3}}D_{U_{2}}\nabla_{v_{1}}X_{t}^{x}|^{2}\\
\leq&C_{A,L,d}(1+\frac{1}{t^\frac{5}{2}})|v_{1}|^{2}|v_{2}|^{2}|v_{3}|^{2}.
\end{align*}
\end{proof}

\subsection{Proof of Lemma \ref{mainlem1}}

Recall $P_{t}h(x)=\mathbb{E}[h(X_{t}^{x})]$ for $h\in Lip(1),$ by Lebesgue's dominated convergence theorem, the Cauchy-Schwarz inequality and (\ref{grad3}), we have
\begin{align*}
|\nabla_{v}\mathbb{E}[h(X_{t}^{x})]|=|\mathbb{E}[\nabla h(X_{t}^{x})\nabla_{v}X_{t}^{x}]|
\leq\|\nabla h\|\mathbb{E}|\nabla_{v}X_{t}^{x}|\leq e^{L+4},
\end{align*}
(\ref{gradient2}) is proved.

Denote
\begin{align}\label{ma2}
h_{\epsilon}(x)=\int_{\mathbb{R}^{d}}f_{\epsilon}(y)h(x-y)\dif y,
\end{align}
with $\epsilon>0$ and $f_{\epsilon}$ is the density of the normal distribution $N(0,\epsilon^{2}I_{d}).$ It is easy to see that $h_{\epsilon}$ is smooth, $\lim_{\epsilon\rightarrow0}h_{\epsilon}(x)=h(x),$ $\lim_{\epsilon\rightarrow0}\nabla h_{\epsilon}(x)=\nabla h(x)$ and $|h_{\epsilon}(x)|\leq C(1+|x|)$ for all $x\in\mathbb{R}^{d}$ and some $C>0.$ Moreover, $\|\nabla h_{\epsilon}\|\leq\|\nabla h\|\leq1.$ Then, by Lebesgue's dominated convergence theorem, we have
\begin{align*}
\nabla_{v_{2}}\nabla_{v_{1}}\mathbb{E}\big[h_{\epsilon}(X_{t}^{x})\big]=\mathbb{E}\big[\nabla^{2}h_{\epsilon}(X_{t}^{x})\nabla_{v_{2}}X_{t}^{x}
\nabla_{v_{1}}X_{t}^{x}\big]+\mathbb{E}\big[\nabla h_{\epsilon}(X_{t}^{x})\nabla_{v_{2}}\nabla_{v_{1}}X_{t}^{x}\big],
\end{align*}
by (\ref{malliavin1}) and (\ref{malliavin2}), we further have
\begin{align*}
\mathbb{E}\big[\nabla^{2}h_{\epsilon}(X_{t}^{x})\nabla_{v_{2}}X_{t}^{x}
\nabla_{v_{1}}X_{t}^{x}\big]=&\mathbb{E}\big[\nabla^{2}h_{\epsilon}(X_{t}^{x})D_{U_{2}}X_{t}^{x}
\nabla_{v_{1}}X_{t}^{x}\big]\\
=&\mathbb{E}\big[D_{U_{2}}\big(\nabla h_{\epsilon}(X_{t}^{x})\big)\nabla_{v_{1}}X_{t}^{x}\big]\\
=&\mathbb{E}\big[D_{U_{2}}\big(\nabla h_{\epsilon}(X_{t}^{x})\nabla_{v_{1}}X_{t}^{x}\big)\big]-\mathbb{E}\big[\nabla h_{\epsilon}(X_{t}^{x})D_{U_{2}}\nabla_{v_{1}}X_{t}^{x}\big]\\
=&\mathbb{E}\big[\nabla h_{\epsilon}(X_{t}^{x})\nabla_{v_{1}}X_{t}^{x}\mathcal{I}_{v_{2}}^{x}(t)\big]-\mathbb{E}\big[\nabla h_{\epsilon}(X_{t}^{x})D_{U_{2}}\nabla_{v_{1}}X_{t}^{x}\big].
\end{align*}
where the last equality is by Bismut's formula (\ref{bismut}). These imply
\begin{align}\label{seccrucial}
\nabla_{v_{2}}\nabla_{v_{1}}\mathbb{E}\big[h_{\epsilon}(X_{t}^{x})\big]=\mathbb{E}\big[\nabla h_{\epsilon}(X_{t}^{x})\nabla_{v_{1}}X_{t}^{x}\mathcal{I}_{v_{2}}^{x}(t)\big]+\mathbb{E}\big[\nabla h_{\epsilon}(X_{t}^{x})\mathcal{R}_{v_{1},v_{2}}^{x}(t)\big]
\end{align}
Therefore, by Lebesgue's dominated convergence theorem, the Cauchy-Schwarz inequality, (\ref{grad3}), (\ref{1}) and (\ref{third1}), we have
\begin{align*}
|\nabla_{v_{2}}\nabla_{v_{1}}\mathbb{E}[h(X_{t}^{x})]|=&|\lim_{\epsilon\rightarrow0}\nabla_{v_{2}}\nabla_{v_{1}}\mathbb{E}[h_{\epsilon}(X_{t}^{x})]|\\
\leq&\mathbb{E}|\nabla h_{\epsilon}(X_{t}^{x})\nabla_{v_{1}}X_{t}^{x}\mathcal{I}_{v_{2}}^{x}(t)|+\mathbb{E}|\mathcal{R}_{v_{1},v_{2}}^{x}(t)|\\
\leq&\sqrt{\mathbb{E}|\nabla_{v_{1}}X_{t}^{x}|^{2}\mathbb{E}|\mathcal{I}_{v_{2}}^{x}(t)|^{2}}+\sqrt{\mathbb{E}|\mathcal{R}_{v_{1},v_{2}}^{x}(t)|^{2}}
\leq C_{A,L,d}\frac{1}{\sqrt{\delta t}},
\end{align*}
(\ref{sec}) is proved.

By (\ref{seccrucial}) and Lebesgue's dominated convergence theorem, we have
\begin{align*}
\nabla_{v_{3}}\nabla_{v_{2}}\nabla_{v_{1}}\mathbb{E}\big[h_{\epsilon}(X_{t}^{x})\big]=&\nabla_{v_{3}}\mathbb{E}\big[\nabla h_{\epsilon}(X_{t}^{x})\nabla_{v_{1}}X_{t}^{x}\mathcal{I}_{v_{2}}^{x}(t)\big]+\nabla_{v_{3}}\mathbb{E}\big[{\nabla h_{\epsilon}(X_{t}^{x})}\mathcal{R}_{v_{1},v_{2}}^{x}(t)\big]\\
=&\mathbb{E}\big[\nabla^{2} h_{\epsilon}(X_{t}^{x})\nabla_{v_{3}}X_{t}^{x}\nabla_{v_{1}}X_{t}^{x}\mathcal{I}_{v_{2}}^{x}(t)\big]+\mathbb{E}\big[\nabla h_{\epsilon}(X_{t}^{x})\nabla_{v_{3}}\nabla_{v_{1}}X_{t}^{x}\mathcal{I}_{v_{2}}^{x}(t)\big]\\
&+\mathbb{E}\big[\nabla h_{\epsilon}(X_{t}^{x})\nabla_{v_{1}}X_{t}^{x}\nabla_{v_{3}}\mathcal{I}_{v_{2}}^{x}(t)\big]+\mathbb{E}\big[\nabla^{2} h_{\epsilon}(X_{t}^{x})\nabla_{v_{3}}X_{t}^{x}\mathcal{R}_{v_{1},v_{2}}^{x}(t)\big]\\
&+\mathbb{E}\big[\nabla h_{\epsilon}(X_{t}^{x})\nabla_{v_{3}}\mathcal{R}_{v_{1},v_{2}}^{x}(t)\big],
\end{align*}
by (\ref{malliavin1}), (\ref{malliavin2}) and (\ref{bismut}), we further have
\begin{align*}
&\mathbb{E}\big[\nabla^{2} h_{\epsilon}(X_{t}^{x})\nabla_{v_{3}}X_{t}^{x}\nabla_{v_{1}}X_{t}^{x}\mathcal{I}_{v_{2}}^{x}(t)\big]\\
=&\mathbb{E}\big[\nabla^{2} h_{\epsilon}(X_{t}^{x})D_{U_{3}}X_{t}^{x}\nabla_{v_{1}}X_{t}^{x}\mathcal{I}_{v_{2}}^{x}(t)\big]\\
=&\mathbb{E}\big[D_{U_{3}}\big(\nabla h_{\epsilon}(X_{t}^{x})\big)\nabla_{v_{1}}X_{t}^{x}\mathcal{I}_{v_{2}}^{x}(t)\big]\\
=&\mathbb{E}\big[D_{U_{3}}\big(\nabla h_{\epsilon}(X_{t}^{x})\nabla_{v_{1}}X_{t}^{x}\mathcal{I}_{v_{2}}^{x}(t)\big)\big]-\mathbb{E}\big[\nabla h_{\epsilon}(X_{t}^{x})D_{U_{3}}\nabla_{v_{1}}X_{t}^{x}\mathcal{I}_{v_{2}}^{x}(t)\big]\\
&-\mathbb{E}\big[\nabla h_{\epsilon}(X_{t}^{x})\nabla_{v_{1}}X_{t}^{x}D_{U_{3}}\mathcal{I}_{v_{2}}^{x}(t)\big]\\
=&\mathbb{E}\big[\nabla h_{\epsilon}(X_{t}^{x})\nabla_{v_{1}}X_{t}^{x}\mathcal{I}_{v_{2}}^{x}(t)\mathcal{I}_{v_{3}}^{x}(t)\big]-\mathbb{E}\big[\nabla h_{\epsilon}(X_{t}^{x})D_{U_{3}}\nabla_{v_{1}}X_{t}^{x}\mathcal{I}_{v_{2}}^{x}(t)\big]\\
&-\mathbb{E}\big[\nabla h_{\epsilon}(X_{t}^{x})\nabla_{v_{1}}X_{t}^{x}D_{U_{3}}\mathcal{I}_{v_{2}}^{x}(t)\big]
\end{align*}
and
\begin{align*}
&\mathbb{E}\big[\nabla^{2} h_{\epsilon}(X_{t}^{x})\nabla_{v_{3}}X_{t}^{x}\mathcal{R}_{v_{1},v_{2}}^{x}(t)\big]\\
=&\mathbb{E}\big[\nabla^{2} h_{\epsilon}(X_{t}^{x})D_{U_{3}}X_{t}^{x}\mathcal{R}_{v_{1},v_{2}}^{x}(t)\big]\\
=&\mathbb{E}\big[D_{U_{3}}\big(\nabla h_{\epsilon}(X_{t}^{x})\big)\mathcal{R}_{v_{1},v_{2}}^{x}(t)\big]\\
=&\mathbb{E}\big[D_{U_{3}}\big(\nabla h_{\epsilon}(X_{t}^{x})\mathcal{R}_{v_{1},v_{2}}^{x}(t)\big)\big]-\mathbb{E}\big[\nabla h_{\epsilon}(X_{t}^{x})D_{U_{3}}\mathcal{R}_{v_{1},v_{2}}^{x}(t)\big]\\
=&\mathbb{E}\big[\nabla h_{\epsilon}(X_{t}^{x})\mathcal{R}_{v_{1},v_{2}}^{x}(t)\mathcal{I}_{v_{3}}^{x}(t)\big]-\mathbb{E}\big[\nabla h_{\epsilon}(X_{t}^{x})D_{U_{3}}\mathcal{R}_{v_{1},v_{2}}^{x}(t)\big].
\end{align*}
These imply
\begin{align*}
\nabla_{v_{3}}\nabla_{v_{2}}\nabla_{v_{1}}\mathbb{E}\big[h_{\epsilon}(X_{t}^{x})\big]
=&\mathbb{E}\big[\nabla h_{\epsilon}(X_{t}^{x})\nabla_{v_{1}}X_{t}^{x}\mathcal{I}_{v_{2}}^{x}(t)\mathcal{I}_{v_{3}}^{x}(t)\big]-\mathbb{E}\big[\nabla h_{\epsilon}(X_{t}^{x})D_{U_{3}}\nabla_{v_{1}}X_{t}^{x}\mathcal{I}_{v_{2}}^{x}(t)\big]\nonumber\\
&-\mathbb{E}\big[\nabla h_{\epsilon}(X_{t}^{x})\nabla_{v_{1}}X_{t}^{x}D_{U_{3}}\mathcal{I}_{v_{2}}^{x}(t)\big]+\mathbb{E}\big[\nabla h_{\epsilon}(X_{t}^{x})\nabla_{v_{3}}\nabla_{v_{1}}X_{t}^{x}\mathcal{I}_{v_{2}}^{x}(t)\big]\nonumber\\
&+\mathbb{E}\big[\nabla h_{\epsilon}(X_{t}^{x})\nabla_{v_{1}}X_{t}^{x}\nabla_{v_{3}}\mathcal{I}_{v_{2}}^{x}(t)\big]+\mathbb{E}\big[\nabla h_{\epsilon}(X_{t}^{x})\mathcal{R}_{v_{1},v_{2}}^{x}(t)\mathcal{I}_{v_{3}}^{x}(t)\big]\nonumber\\
&-\mathbb{E}\big[\nabla h_{\epsilon}(X_{t}^{x})D_{U_{3}}\mathcal{R}_{v_{1},v_{2}}^{x}(t)\big]+\mathbb{E}\big[\nabla h_{\epsilon}(X_{t}^{x})\nabla_{v_{3}}\mathcal{R}_{v_{1},v_{2}}^{x}(t)\big].
\end{align*}
Therefore, by Lebesgue's dominated convergence theorem, we have
\begin{align*}
|\nabla_{v_{3}}\nabla_{v_{2}}\nabla_{v_{1}}\mathbb{E}[h(X_{t}^{x})]|=&|\lim_{\epsilon\rightarrow0}\nabla_{v_{3}}\nabla_{v_{2}}\nabla_{v_{1}}\mathbb{E}[h_{\epsilon}(X_{t}^{x})]|\\
\leq&\mathbb{E}|\nabla_{v_{1}}X_{t}^{x}\mathcal{I}_{v_{2}}^{x}(t)\mathcal{I}_{v_{3}}^{x}(t)|+\mathbb{E}|D_{U_{3}}\nabla_{v_{1}}X_{t}^{x}\mathcal{I}_{v_{2}}^{x}(t)|\\
&+\mathbb{E}|\nabla_{v_{1}}X_{t}^{x}D_{U_{3}}\mathcal{I}_{v_{2}}^{x}(t)|+\mathbb{E}|\nabla_{v_{3}}\nabla_{v_{1}}X_{t}^{x}\mathcal{I}_{v_{2}}^{x}(t)|\\
&+\mathbb{E}|\nabla_{v_{1}}X_{t}^{x}\nabla_{v_{3}}\mathcal{I}_{v_{2}}^{x}(t)|+\mathbb{E}|\mathcal{R}_{v_{1},v_{2}}^{x}(t)\mathcal{I}_{v_{3}}^{x}(t)|\\
&+\mathbb{E}|D_{U_{3}}\mathcal{R}_{v_{1},v_{2}}^{x}(t)|+\mathbb{E}|\nabla_{v_{3}}\mathcal{R}_{v_{1},v_{2}}^{x}(t)|.
\end{align*}
Then, by the Cauchy-Schwarz inequality, (\ref{grad3}) and (\ref{1}), we have
\begin{align*}
\mathbb{E}|\nabla_{v_{1}}X_{t}^{x}\mathcal{I}_{v_{2}}^{x}(t)\mathcal{I}_{v_{3}}^{x}(t)|\leq C_{L,d}\frac{1}{\delta t}.
\end{align*}
By (\ref{malgrad1}) and (\ref{1}), we have
\begin{align*}
\mathbb{E}|D_{U_{3}}\nabla_{v_{1}}X_{t}^{x}\mathcal{I}_{v_{2}}^{x}(t)|\leq& C_{A,L,d}\frac{1}{\sqrt{\delta t}}(1+\frac{1}{\sqrt{t}}).
\end{align*}
By (\ref{grad3}) and (\ref{3}), we have
\begin{align*}
\mathbb{E}|\nabla_{v_{1}}X_{t}^{x}D_{U_{3}}\mathcal{I}_{v_{2}}^{x}(t)|\leq&C_{A,L,d}\frac{1}{\sqrt{\delta t}}(1+\frac{1}{\sqrt{\delta t}}).
\end{align*}
By (\ref{gradsec}) and (\ref{1}), we have
\begin{align*}
\mathbb{E}|\nabla_{v_{3}}\nabla_{v_{1}}X_{t}^{x}\mathcal{I}_{v_{2}}^{x}(t)|\leq&C_{A,L,d}\frac{1}{\sqrt{\delta t}}.
\end{align*}
By (\ref{grad3}) and (\ref{2}), we have
\begin{align*}
\mathbb{E}|\nabla_{v_{1}}X_{t}^{x}\nabla_{v_{3}}\mathcal{I}_{v_{2}}^{x}(t)|\leq \frac{C_{A,L,d}}{\sqrt{\delta t}}.
\end{align*}
By (\ref{third1}) and (\ref{1}), we have
\begin{align*}
\mathbb{E}|\mathcal{R}_{v_{1},v_{2}}^{x}(t)\mathcal{I}_{v_{3}}^{x}(t)|\leq&C_{A,L,d}\frac{1}{\sqrt{\delta t}}(1+\frac{1}{\sqrt{t}}).
\end{align*}
By (\ref{third3}), we have
\begin{align*}
\mathbb{E}|D_{U_{3}}\mathcal{R}_{v_{1},v_{2}}^{x}(t)|\leq C_{A,L,d}(1+\frac{1}{t^\frac{5}{4}}).
\end{align*}
By (\ref{third2}), we have
\begin{align*}
\mathbb{E}|\nabla_{v_{3}}\mathcal{R}_{v_{1},v_{2}}^{x}(t)|\leq C_{A,L,d}(1+\frac{1}{\sqrt{t}}).
\end{align*}
These imply
\begin{align*}
|\nabla_{v_{3}}\nabla_{v_{2}}\nabla_{v_{1}}\mathbb{E}[h(X_{t}^{x})]|\leq C_{A,L,d}\left(1+\frac{1}{\delta t}+\frac{1}{t^{\frac{5}{4}}}\right).
\end{align*}
\qed

\section{the distribution of the solution of the SDDE (\ref{sfde})}\label{invariant measure}

In this section, we will prove that when $K$ in Assumption \ref{assum2} and $\delta$ in the SDDE (\ref{sfde}) both are small enough, the distribution of the solution $\tilde{X}_{s}$ of the SDDE (\ref{sfde}) is close to the minimizer $\omega^*$ as $s\rightarrow\infty$.

\begin{lemma}  \label{l:XtEst}
Under the Assumptions \ref{assum1} and \ref{assum2}. Let $b:=e^{-2(\gamma-L^{2}\eta)m\eta}+\frac{\eta L^{2}}{\gamma-L^{2}\eta}$. For any $s\in\mathbb{N}$, as $\eta\leq\frac{\gamma}{3L^{2}}$ and $b<1$, we have
\begin{eqnarray*}
\E|\tilde X_{s}-\omega^{*}|^{2}
& \le & b^s |x-\omega^{*}|^{2}+\frac{2K+\delta d}{2(\gamma-L^{2}\eta)}\frac{1}{1-b}.
\end{eqnarray*}	
\end{lemma}

\begin{proof}
By It\^{o}'s formula, for any $s,t \in [0,m\eta]$ with $s \le t$ we have
\begin{align*}
\E|X_{t}-\omega^{*}|^{2}=&\mathbb{E}|X_s-\omega^{*}|^{2}-2\int_{s}^{t} \E \langle X_{r}-\omega^{*}, \nabla P(X_{r})\rangle \dif r\\
&+\eta \int_{s}^{t} \E {\rm tr} \left[\Sigma(X_{r},x)\right] \dif r+\delta d(t-s).
\end{align*}
Since $\nabla P(\omega^{*})=0$ for the minimizer $\omega^{*}$, (\ref{integration1}) implies
\begin{align*}
\langle X_{r}-\omega^{*}, \nabla P(X_{r})\rangle=\langle X_{r}-\omega^{*}, \nabla P(X_{r})-\nabla P(\omega^{*})\rangle\ge \gamma |X_{r}-\omega^{*}|^{2}-K.
\end{align*}
By (\ref{e:Tr=L2Norm}), one has
\begin{align*}
\E {\rm tr} \left[\Sigma(X_{r},x)\right]\le L^{2} \E|X_{r}-x|^{2}\le 2L^{2}\E |X_{r}-\omega^{*}|^{2}+2L^{2}|x-\omega^{*}|^{2}.
\end{align*}
Therefore,
\begin{align*}
&\E|X_{t}-\omega^{*}|^{2}-|X_s-\omega^{*}|^{2}\\
\le&-2 (\gamma-L^{2}\eta) \int_{s}^{t} \E |X_{r}-\omega^{*}|^{2}\dif r+(2K+2\eta L^{2}|x-\omega^{*}|^{2}+\delta d)(t-s).
	\end{align*}
 Denote $F(t)=\E|X_{t}-\omega^{*}|^{2}$ for $t>0$, then for any $0 \le s \le t \le m\eta$,
 \begin{align}  \label{e:F(t)-F(s)}
 	F(t)-F(s)\le-2 (\gamma-L^{2}\eta) \int_{s}^{t} F(r)\dif r+(2K+2\eta L^{2}|x-\omega^{*}|^{2}+\delta d)(t-s).
 \end{align}
Let $s=0$ and $\eta<\frac{\gamma}{L^2}$, the relation above yields
 \begin{eqnarray*}
	F(t)& \le &F(0)+(2K+2\eta L^{2}|x-\omega^{*}|^{2}+\delta d) m\eta, \ \ \ \quad t \in [0,m\eta],
\end{eqnarray*}
which, together with \eqref{e:F(t)-F(s)}, further implies $F(t)$ is absolutely continuous on $[0,m\eta]$. So $F(t)$ is differentiable on $[0,\eta]$ a.e.. Moreover, it is easy to see that the following differential inequality holds by \eqref{e:F(t)-F(s)},
\begin{eqnarray*}
	F'(t)& \le & -2(\gamma-L^{2}\eta) F(t)+2K+2\eta L^{2}|x-\omega^{*}|^{2}+\delta d \ \ \ \ \quad t \in (0,m\eta].
	\end{eqnarray*}	
Solving this differential inequality further gives
\begin{eqnarray*}
	F(t) \ \le \  e^{-2(\gamma-L^{2}\eta)t}  F(0)+\frac{2K+2\eta L^{2}|x-\omega^{*}|^{2}+\delta d}{2(\gamma-L^{2}\eta)},
\end{eqnarray*}	
that is,
\begin{eqnarray*}
	\E|X_{t}-\omega^{*}|^{2}\
	&\le &   \left[e^{-2(\gamma-L^{2}\eta)t}+\frac{\eta L^{2}}{\gamma-L^{2}\eta}\right]|x-\omega^{*}|^2+\frac{2K+\delta d}{2(\gamma-L^{2}\eta)}, \quad \quad t \in [0, m\eta].
\end{eqnarray*}	
Inductively, for any $s \ge 1$ we have
\begin{eqnarray*}
		\E|\tilde X_s-\omega^{*}|^{2}&=& \E \left[\E\left[|\tilde X_s-\omega^{*}|^{2}|\tilde X_{s-1}\right]\right] \\
		& \le & b   \E |X_{(s-1)m\eta}-\omega^{*}|^{2}+\frac{2K+\delta d}{2(\gamma-L^{2}\eta)}\\
		& \le & b^{s}|x-\omega^{*}|^{2}+\frac{2K+\delta d}{2(\gamma-L^{2}\eta)} \sum_{k=0}^{s-1} b^{k}  \\
		& \le & b^s|x-\omega^{*}|^{2}+\frac{2K+\delta d}{2(\gamma-L^{2}\eta)} \frac{1}{1-b}.
	\end{eqnarray*}	
The case $s=0$ holds obviously and we finish the proof.
\end{proof}
\end{appendix}

\bibliographystyle{amsplain}

\end{document}